\documentclass{amsbook}

\usepackage[utf8]{inputenc}
\usepackage{microtype}
\usepackage{tikz}
\usepackage{amssymb}
\usepackage[colorlinks=true,linkcolor=blue,citecolor=blue]{hyperref}

\renewcommand{\le}{\leqslant}
\renewcommand{\ge}{\geqslant}
\newcommand{\ptl}{\partial}
\newcommand{\rr}{{\mathbb{R}}}
\newcommand{\la}{\lambda}
\newcommand{\hh}{{\mathcal{H}}}
\newcommand{\h}{\mathcal{H}}
\newcommand{\esf}{\mathbb{S}}

\newcommand{\nn}{\mathbb{N}}
\newcommand{\pp}{P}

\newcommand{\escpr}[1]{\big<#1\big>}
\newcommand{\seq}[1]{\{#1_i\}_{i\in\nn}}

\newcommand{\sg}{\sigma}
\newcommand{\Om}{\Omega}
\newcommand{\eps}{\varepsilon}

\newcommand{\Ga}{\Gamma}
\newcommand{\de}{\delta}

\newcommand{\btheta}{{\bar{\theta}}}
\newcommand{\vol}[1]{|#1|}

\newcommand{\cl}[1]{\text{\rm cl}(#1)}
\newcommand{\clb}{\mkern2mu\overline{\mkern-2mu B\mkern-2mu}\mkern2mu}
\newcommand{\K}{\mathcal{K}}

\DeclareMathOperator{\divv}{div}

\DeclareMathOperator{\hd}{\delta}
\DeclareMathOperator{\intt}{int}

\DeclareMathOperator{\inr}{inr}

\DeclareMathOperator{\diam}{diam}

\newenvironment{enum}{\begin{enumerate}
}{\end{enumerate}}

\newtheorem{theorem}{Theorem}[chapter]
\newtheorem{lemma}[theorem]{Lemma}
\newtheorem{proposition}[theorem]{Proposition}
\newtheorem{corollary}[theorem]{Corollary}

\theoremstyle{definition}

\newtheorem{example}[theorem]{Example}

\theoremstyle{remark}
\newtheorem{remark}[theorem]{Remark}

\numberwithin{section}{chapter}
\numberwithin{equation}{chapter}
\numberwithin{figure}{chapter}

\makeindex

\begin{document}

\frontmatter

\title{Isoperimetric inequalities in unbounded convex bodies}

\author[G.~P.~Leonardi]{Gian Paolo Leonardi}
\address{Dipartimento di Scienze Fisiche, Informatiche e Matematiche\\ Universit\`a di Modena e Reggio Emilia \\ 41100 Modena \\ Italy}
\email{gianpaolo.leonardi@unimore.it}
\thanks{Gain Paolo Leonardi has been supported by GNAMPA (project name: \textit{Problemi isoperimetrici e teoria della misura in spazi metrici}, 2015) and by PRIN (project name: \textit{Calcolo delle Variazioni}, 2010-2011).}

\author[M.~Ritor\'e]{Manuel Ritor\'e} 
\address{Departamento de Geometr\'{\i}a y Topolog\'{\i}a \\ Universidad de Granada \\ E--18071 Granada \\ Espa\~na}
\email{ritore@ugr.es}
\thanks{Manuel Ritoré has been supported by MICINN-FEDER grant {MTM2013-48371-C2-1-P} and Junta de Andaluc\'{\i}a grants FQM-325 and P09-FQM-5088}

\author[S.~Vernadakis]{Efstratios Vernadakis}
\address{Departamento de
Geometr\'{\i}a y Topolog\'{\i}a \\
Universidad de Granada \\ E--18071 Granada \\ Espa\~na}
\email{stratos@ugr.es}
\thanks{Efstratios Vernadakis has been supported by Junta de Andaluc\'{\i}a grant P09-FQM-5088}

\date{June 23, 2016}



\maketitle

\tableofcontents

\begin{abstract}
We consider the problem of minimizing the relative perimeter under a volume constraint in an unbounded convex body $C\subset \rr^{n+1}$, without assuming any further regularity on the boundary of $C$. Motivated by an example of an unbounded convex body with null isoperimetric profile, we introduce the concept of unbounded convex body with \textit{uniform geometry}. We then provide a handy characterization of the uniform geometry property and, by exploiting the notion of \textit{asymptotic cylinder} of $C$, we prove existence of isoperimetric regions in a generalized sense. By an approximation argument we show the strict concavity of the isoperimetric profile and, consequently, the connectedness of generalized isoperimetric regions. We also focus on the cases of small as well as of large volumes; in particular we show existence of isoperimetric regions with sufficiently large volumes, for special classes of unbounded convex bodies. We finally address some questions about isoperimetric rigidity and analyze the asymptotic behavior of the isoperimetric profile in connection with the notion of \textit{isoperimetric dimension}.
\end{abstract}

\mainmatter

\chapter{Introduction}

Given a closed set $C\subset\rr^{n+1}$ with non-empty interior, the \emph{relative isoperimetric problem} on $C$ looks for sets $E\subset C$ of given finite volume $|E|$ minimizing the \emph{relative perimeter} $P_C(E)$ of $E$ in the interior of $C$. When the boundaries of $E$ and $C$ are regular hypersurfaces, it is known that $\ptl E\cap C$ is a constant mean curvature hypersurface and its closure meets $\ptl C$ orthogonally.

The \emph{isoperimetric profile function} $I_C$ assigns to each positive volume $0<v<|C|$ the infimum of the relative perimeter of sets $F\subset C$ of volume $v$. An \emph{isoperimetric region} is a set $E\subset C$ whose relative perimeter coincides with $I_{C}(|E|)$. The function $I_C$ provides an optimal isoperimetric inequality on $C$ since $P_C(F)\ge I_C(|F|)$ for any set $F\subset C$.

In this work, we consider the relative isoperimetric problem in \emph{unbounded convex bodies}, i.e. unbounded closed convex sets with non-empty interior in Euclidean space, without assuming any further regularity of their boundaries. We focus on existence of isoperimetric regions, concavity properties of the isoperimetric profile, questions related to isoperimetric rigidity (i.e. given an isoperimetric inequality valid for a convex body $C$, does equality implies a geometric characterization of $C$?), as well as asymptotic isoperimetric inequalities in connection with the problem of determining the isoperimetric dimension of an unbounded convex body.

\section{Historical background} 
Isoperimetric sets are at once a modern and classical topic: they arise in many fields, from physics of interfaces to optimal design of structures and shapes, and have fascinated scientists since antiquity. For instance, they appear in physical systems where surface tension is one of the main driving forces acting in the system. Surface tension was related to the mean curvature of a boundary interface by Young and Laplace in the equation $\Delta p=\sg H$, where $\Delta p$ is the difference of the internal and the external pressures, $H$ is the mean curvature of the interface and $\sg$ is the surface tension.

The capillarity phenomenon is one of the most relevant examples where relative isoperimetric problems come into play. There one observes a liquid and a gas constrained by a solid support whose shape is determined by surface tension of the liquid-gas interface and by the wetting properties of the support, see Michael \cite{michael}, Bostwick and Steen \cite{bostwick-steen}, and Finn \cite{MR816345}. Other examples related to the isoperimetric problem include: the Van der Waals-Cahn-Hilliard theory of phase transitions \cite{MR1803974} (see in particular the $\Ga$-convergence results by Modica \cite{MR866718} and Sternberg \cite{MR930124}, as well as the construction of solutions of the Allen-Cahn equation by Pacard and Ritor\'e \cite{MR2032110}); the shape of A/B block copolymers consisting of two macromolecules bonded together, which separate into distinct phases (Thomas et al. \cite{hoffman-nature}, Ohta-Kawasaki \cite{OhtaKawasaki1986}).

Moreover, isoperimetric problems are relevant for their close, and deep, connection with functional inequalities of paramount importance in analysis, mathematical physics, and probability (like for instance the Sobolev-Poincar\'e inequality, the Faber-Krahn inequality and the Cheeger inequality, see \cite{MR2229062}, \cite{MR1849187}, \cite{Talenti1976best}, \cite{krahn1925rayleigh}, \cite{Cheeger}.

Central questions for the relative isoperimetric problem are the existence, regularity and geometric properties of isoperimetric regions, as well as the properties of the isoperimetric profile function.

For \emph{bounded} convex bodies many results are known. When the boundary is \emph{smooth}, the concavity of the isoperimetric profile and the connectedness of the reduced boundary of isoperimetric regions was proved by Sternberg and Zumbrun \cite{MR1674097}, while the concavity of the function $I_C^{(n+1)/n}$ was proved by Kuwert \cite{MR2008339}. The behavior of the isoperimetric profile for small volumes was considered by B\'erard and Meyer \cite{be-me}, and the behavior of isoperimetric regions for small volumes by Fall \cite{fall}. Connectedness of isoperimetric regions and of their complements was obtained by Ritor\'e and Vernadakis \cite{MR3335407}. See also the works by Bayle \cite{bayle}, Bayle and Rosales \cite{bay-rosal} and Morgan and Johnson \cite{MR1803220}. The results in all these papers make a strong use of the regularity of the boundary. In particular, as shown in \cite{MR1674097} and \cite{MR2008339}, the $C^{2,\alpha}$ regularity of the boundary implies a strong regularity of the isoperimetric regions up to the boundary, except for a singular set of large Hausdorff codimension, that allows the authors to apply the classical first and second variation formulas for volume and perimeter. The convexity of the boundary then implies the concavity of the profile and the connectedness of the regular part of the free boundary.

Up to our knowledge, the only known results for non-smooth boundary are the ones by Bokowski and Sperner \cite{bo-sp} on isoperimetric inequalities for the Minkowski content in Euclidean convex bodies; the isoperimetric inequality for convex cones by Lions and Pacella \cite{lions-pacella} using the Brunn-Minkowski inequality, with the characterization of isoperimetric regions by Figalli and Indrei \cite{FI}; the extension of Levy-Gromov inequality, \cite[App.~C]{grom}, to arbitrary convex sets given by Morgan \cite{MR2438911}; the extension of the concavity of the $\big(\tfrac{n+1}{n}\big)$ power of the isoperimetric profile to arbitrary convex bodies by E.~Milman \cite[\S~6]{MR2507637}. In his work on the isoperimetric profile for small volumes in the \emph{boundary} of a polytope, Morgan mentions that his techniques can be adapted to handle the case of small volumes in a solid polytope, \cite[Remark~3.11]{morganpolytops}, without uniqueness, see Remark after Theorem~3.8 in \cite{morganpolytops}. Previous estimates on least perimeter in convex bodies have been obtained by Dyer and Frieze \cite{MR1141926}, Kannan, Lov\'asz and Simonovits \cite{MR1318794} and Bobkov \cite{MR2347041}. Outside convex bodies in Euclidean space, isoperimetric inequalities have been established by Choe, Ghomi and Ritor\'e \cite{MR2215458}, \cite{MR2329803}, and, in the case of $3$-dimensional Hadamard manifolds, by Choe and Ritor\'e \cite{MR2338131}.

In the case of unbounded convex bodies, several results on the isoperimetric profile of \emph{cylindrically bounded} convex bodies have been obtained in \cite{rv2} and for \emph{conically bounded} ones in \cite{MR3441524}. In convex cones, the results by Lions and Pacella \cite{lions-pacella} were recovered by Ritor\'e and Rosales \cite{r-r} using stability techniques.

It is important to mention that existence of isoperimetric regions in non-compact spaces is not always guaranteed. For instance, complete planes of revolution with (negative) increasing Gauss curvature are known to have no isoperimetric regions for any value of the two-dimensional volume as shown in \cite[Theorem~2.16]{MR1883725}. While general existence of solutions of variational problems in non-compact spaces is generally treated by means of concentration-compactness arguments (\cite{MR834360}, \cite{MR850686}), the use of geometric arguments in the theory of isoperimetric inequalities to study the behavior of minimizing sequences is quite old and can be traced back to Fiala's paper \cite{MR0006422}, where it was shown that in a complete surface with non-negative Gauss curvature, a sequence of discs escaping to infinity have worse isoperimetric ratio that some compact regions of the same area. This argument was exploited in \cite{MR1857855} to prove existence of isoperimetric regions in complete surfaces with non-negative Gauss curvature. An essential ingredient in this proof was the geometric description of the behavior of minimizing sequences given in \cite[Lemma~1.8]{MR1883725} and used in \cite[Theorem~2.8]{MR1883725} to show existence of isoperimetric regions in complete planes of revolution with non-decreasing Gauss curvature. Lemma~1.8 in \cite{MR1883725} was extended to Riemannian manifolds in \cite[Theorem~2.1]{r-r} and used to prove existence of isoperimetric regions in convex cones with smooth boundary in the same paper. More modern existence results can be traced back to Almgren \cite[Chapter~VI]{MR0420406}, who proved existence of solutions of the isoperimetric problem in $\rr^{n+1}$ for multiple volume constraints as a particular case of a more general theory for elliptic integrands. Morgan \cite{MR1286892}, based on Almgren's results, proved existence and regularity of clusters of prescribed volume in $\rr^3$ minimizing area plus length of singular curves. The same author obtained in \cite{MR2455580} existence of isoperimetric regions in a Riemannian manifold whose quotient by its isometry group is compact, see also \cite[\S~4.5]{MR1286892}. Eichmair and Metzger showed in \cite{MR0006422} that the leaves of the canonical foliation by stable constant mean curvature spheres in asymptotically flat manifolds asymptotic to a Schwarzschild space of positive mass are the only isoperimetric boundaries for the (large) volume they enclose, thus proving existence for large volumes. The same authors proved in \cite{MR0006422} that an asymptotically flat Riemannian 3-manifold with non-negative scalar curvature contains a sequence of isoperimetric regions whose volumes diverge to infinity. Mondino and Nardulli \cite{1210.0567} showed existence of isoperimetric regions of any volume in asymptotically flat manifolds with Ricci curvature uniformly bounded below.

As for the regularity of isoperimetric regions, the classical result on interior regularity was proved by Gonzalez, Massari and Tamanini \cite{MR684753}, after the pioneering work by De Giorgi on regularity of perimeter-minimizing sets without a volume constraint. The boundary regularity for perimeter minimizing sets under a volume constraint inside a set with smooth boundary follows by the work of Gr\"uter \cite{MR862549}, see also \cite{MR1674097}.

As for the geometric characterization of isoperimetric regions in convex sets, isoperimetric solutions in a half-space are easily shown to be half-balls by reflecting an isoperimetric set with respect to the boundary hyperplane and applying the classical isoperimetric inequality in Euclidean space. The characterization of spherical caps as isoperimetric boundaries in balls was given by Bokowski and Sperner \cite{bo-sp} (see also Burago and Zalgaller \cite{MR936419}) as an application of spherical symmetrization. Results on smooth second order minimizers of the perimeter in balls were given by Ros and Vergasta \cite{MR1338315}, see also \cite{MR2167260}. In a slab, the non-empty intersection of two half-spaces determined by two parallel hyperplanes, isoperimetric solutions were classified by Vogel \cite{MR889635} and Athanassenas \cite{MR887402} in the $3$-dimensional case. In both papers, the problem is reduced by symmetrization to axially symmetric sets. The only solutions of the problem are half-balls and tubes around segments connecting orthogonally the two boundary hyperplanes. Both results were later extended by Pedrosa and Ritoré \cite{pedri} to higher dimensional Euclidean spaces using stability techniques. The three-dimensional case also follows from the results in \cite{MR1161286}. In Theorem~4.2 in \cite{MR1483543} it was proved the existence of a constant $\eps>0$ such that isoperimetric regions in $[0,1]\times [0,\delta]\times \rr$, with $\delta>\eps$ are half-balls, tubes around closed segments connecting vertical walls and horizontal slabs, see also \cite{MR1472144} and \cite{MR1322955}. A similar result can be proved for cuboids. Ros proved estimates on the genus of isoperimetric surfaces in quotients of $\rr^3$ by crystallographic groups \cite{MR2051615}

Finally, we would like to remark that the relative isoperimetric problem is quite different from the minimization of the Euclidean perimeter under a volume constraint inside $C$, a problem considered by several authors, \cite{MR1669207}, \cite{MR2010323}, \cite{MR2178065}.

\section{Outline of contents} 

This work has been organized into several chapters.

In Chapter~\ref{sec:preliminaries} the notation used along the manuscript is fixed and basic definitions and facts about convex bodies and finite perimeter sets are presented. In particular, the notion of local convergence in Hausdorff distance is introduced at the beginning of \S~\ref{sub:localh}, followed by the proof of some useful properties of this notion of convergence. Given a convex body $C$, the concavity of the function $(x,r)\in C\times\rr^+\mapsto |\clb(x,r)\cap C|^1/(n+1)$ using the Brunn-Minkowski inequality is proved in Lemma~\ref{lemma:fxrconcave}.

In Chapter~\ref{sec:unbounded} we introduce and study some key concepts, in particular the notion of unbounded convex body of \textit{uniform geometry} and its close relationship with the non-triviality of the isoperimetric profile of $C$. We first realized the importance of uniform geometry after the discovery of an unbounded convex body whose isoperimetric profile is identically zero, see Example \ref{ex:ex}. The definition is as follows: we say that an unbounded convex body $C$ is of uniform geometry if the volume of a relative ball $B_{C}(x,r)$ of a fixed radius $r>0$ cannot be made arbitrarily small by letting $x\in C$ go off to infinity, see \eqref{eq:mainhyp}. We remark that this condition does not require any further regularity of $\partial C$. Proposition \ref{prop:maincond} shows a useful characterization of the uniform geometry assumption on an unbounded convex body and, in particular, its equivalence to the positivity of the isoperimetric profile $I_C$ for any given volume. Moreover, uniform geometry is proved to be equivalent to the fact that any \textit{asymptotic cylinder} of $C$ is a convex body (with nonempty interior). The concept of asymptotic cylinder is introduced at the beginning of Section \ref{sec:asymptoticcyl} as a local limit in Hausdorff distance of a sequence of translations $\{-x_{j}+C\}$, with $\{x_{j}\}_{j\in\nn}\subset C$ being a divergent sequence of points. As mentioned above, this concept turns out to be crucial also in exploiting the connection between uniform geometry and non-degeneracy of the isoperimetric profile $I_{C}$ (see again Proposition \ref{prop:maincond}). Moreover it can be shown by a slight modification of Example \ref{ex:ex} that the assumption of uniform geometry is stronger than simply requiring a uniform lower bound on the solid angles of the tangent cones to $C$. Various classes of unbounded convex bodies with uniform geometry are presented in Examples \ref{ex:cylindrically}, {\ref{ex:revolution}, \ref{ex:noregbdy} and \ref{ex:convexcones}}. 
A detailed account of asymptotic cylinders of a convex cone is done in Proposition \ref{prop:convexcone}. A consequence of this proposition is that whenever the boundary of a convex cone is $C^{1}$ out of a vertex, then all its asymptotic cylinders are either half-spaces or the whole Euclidean space. Moreover in Proposition \ref{prop:NDACsmooth} we prove that unbounded convex bodies with non degenerate asymptotic cone having a $C^{1}$ boundary out of a given vertex present the same type of asymptotic cylinders. Then in Section \ref{sec:density} we exploit some further consequences of uniform geometry, that will be of use in Chapters \ref{sec:min} and \ref{sec:cont} (in particular some uniform density and concentration estimates for sets of finite perimeter in $C$, see Lemma \ref{lem:lrlemme31} and Corollaries \ref{cor:lemme1} and \ref{cor:lemme2-b}, as well as the boundedness of isoperimetric regions, see Proposition \ref{prp:isopbound}).

In Chapter \ref{sec:min} we prove a generalized existence result, Theorem \ref{thm:genexist}, in the spirit of Nardulli's existence theorem \cite{nardulli2014generalized} for non-compact Riemannian manifolds (without boundary) satisfying a so-called \textit{smoothly bounded geometry} property. Theorem \ref{thm:genexist} says that the isoperimetric profile $I_{C}(v)$ of an unbounded convex body $C$ is attained for any fixed $v>0$ by a \textit{generalized isoperimetric region} consisting of an array of sets $(E_{0},\dots,E_{\ell})$, such that $E_{0}\subset C =K_{0}$ and $E_{i}\subset K_{i}$ for $i=1,\dots,\ell$ and for suitable asymptotic cylinders $K_{1},\dots, K_{\ell}$, which satisfy $\sum_{i=0}^{\ell}|E_{i}| = v$ and $\sum_{i=0}^{\ell}P_{K_{i}}(E_{i}) = I_{C}(v)$. As Theorem \ref{thm:conY} will later show, this result can be significantly improved as soon as the concavity of the isoperimetric profile is proved (which in turn requires a continuity result, Theorem \ref{thm:genexist}, as an essential intermediate step). For the proof of Theorem \ref{thm:genexist} we closely follow the scheme adopted by Galli and Ritor\'e \cite{gr}. Essentially, we combine the uniform Poincar\'e inequality stated in Lemma \ref{lem:inrad}, a doubling property on $C$ proved in Lemma~\ref{lem:doubling}, an upper bound on $I_{C}(v)$ stated in Remark \ref{rem:half-plane}, and a well-known volume fixing deformation where the perimeter change is controlled by the volume change, up to a multiplicative constant (see Lemma \ref{lem:niceapprox}). 

In Chapter \ref{sec:cont} we show the strict concavity of the isoperimetric profile of an unbounded convex body $C$ of uniform geometry. To this aim we first approximate $C$ by
a sequence $\{C_i\}_{i\in\nn}$ of unbounded convex bodies in the Hausdorff distance, so that all $C_i$ and all their asymptotic cylinders are of class $C^{2,\alpha}$, see Lemma \ref{lem:fund}. Then we prove that the renormalized isoperimetric profiles $Y_{C_i}:=I_{C_i}^{(n+1)/n}$ are concave, Lemma \ref{lem:imp} (the function $I_C$ is $(\tfrac{n+1}{n})$-concave in the terminology of Brascamp and Lieb \cite{MR0450480}), and then obtain $Y_C$ also concave by passing to the limit. An implication of the concavity of $Y_{C}$ is the strict subadditivity of the isoperimetric profile $I_{C}$, yielding a further refinement of the generalized existence Theorem \ref{thm:genexist}, i.e., that $I_{C}(v)$ is attained by a single, connected isoperimetric set of volume $v$ contained either in $C$ or in some asymptotic cylinder of $C$ (see Theorem \ref{thm:conY}). We stress that an essential ingredient in the proof of Lemma~\ref{lem:imp} is the continuity of the isoperimetric profile of $C$ proved in Theorem~\ref{thm:contprof}. We also remark that such a continuity may fail in a non-compact space, as shown by the recent example by Nardulli and Pansu \cite{1506.04892}. However, the existence of a Lipschitz continuous strictly convex exhaustion function on a manifold guarantees the continuity of the profile \cite{1503.07014}. Conditions on the sectional curvature of a complete manifold, such as non-negativity or non-positivity, imply the existence of such an exhaustion function.

Chapter~\ref{sec:rigid} contains several new isoperimetric inequalities and rigidity results for the equality cases. First it is proved in Theorem~\ref{thm:optnondeg} that for a convex body $C$ with non-degenerate asymptotic cone $C_\infty$, the inequality $I_C\ge I_{C_\infty}$ holds and that the quotient $I_C(v)/I_{C_\infty}(v)$ tends to $1$ as $v\to +\infty$. Existence of isoperimetric regions for large volumes, as well as convergence of rescalings to balls in $C_\infty$, are also shown. Apart from its own interest, Theorem~\ref{thm:optnondeg} is also used as a tool in Theorem~\ref{thm:limI} to prove that
\[
\lim_{v\to 0}\frac{I_C(v)}{I_{C_{\min}}(v)}=1,
\]
where $C_{\min}$ is a tangent cone to $C$ or to an asymptotic cylinder of $C$ with minimum solid angle. The existence of such a cone is established in Lemma~\ref{lem:losemcontangcon}. An interesting consequence of Theorem~\ref{thm:limI} is a new proof of the characterization of isoperimetric regions of small volume in polytopes or prisms given in Corollary~\ref{cor:pol}: they are relative balls in $C$ centered at vertices of $C$ with the smallest solid angle. Finally, some rigidity results are given. In Theorem~\ref{thm:novo}, given a convex body $C$ with non-degenerate asymptotic cone $C_\infty$, it is shown that, if the equality $I_C(v_0)=I_{C_\infty}(v_0)$ holds for some $v_0>0$ then $C$ is isometric to $C_\infty$. Then in Theorem~\ref{thm:rigid1-a} we prove that, whenever the equality holds for some $v_0\in (0,|C|]$ in any of the inequalities
\[
I_C\le I_{C_{\min}},\quad I_C\ge I_{\la C},
\]
or for $v<w$ in the inequality
\[
\frac{I_C(v)}{v^{n/(n+1)}}\ge \frac{I_C(w)}{w^{n/(n+1)}},
\]
then $I_C=I_{C_{\min}}$ in the interval $(0,v_0]$. Moreover, if $K$ is either $C$ or an asymptotic cylinder where the minimum of the solid angle is attained for some $p\in K$ then $K\cap B(p,r)=K_p\cap B(p,r)$ for any $r>0$ such that $|B_K(p,r)|\le v_0$, and in this case $B_{K}(p,r)$ is a generalized isoperimetric region. Then, Corollary~\ref{cor:rigid1-b} and Theorem \ref{thm:rigid1-c} show that, if equality holds in any of the inequalities
\[
I_C\le I_H,\quad I_C\le I_{\rr^{n+1}\setminus C},
\]
then $C$ is a closed half-space or a slab. A consequence of Corollary~\ref{cor:rigid1-b} is a proof of existence of isoperimetric regions in convex bodies whose asymptotic cylinders are either closed half-spaces or the entire space $\rr^{n+1}$, as it happens for convex bodies of revolution that are not cylindrically bounded, or for convex bodies with a non degenerate asymptotic cone that is of class $C^{1}$ outside a vertex. These results are proven in Theorem~\ref{thm:existence} and Corollary~\ref{cor:existforNDACsmooth}. We finally consider the case of cylindrically bounded convex bodies and we prove in Theorem 
\ref{thm:exiscyl} a generalization of the existence of isoperimetric regions of large volumes shown in \cite{rv2} and a rigidity result in the same spirit of Theorems \ref{thm:rigid1-a} and \ref{thm:rigid1-c}.

Finally, Chapter~\ref{sec:isopdim} is focused on the problem of estimating the isoperimetric dimension of $C$, which can be defined as the number $\alpha>0$ such that there exist $0<\lambda_{1}<\lambda_{2}$ and $v_{0}>0$ with the property
\begin{equation}\label{isopdim}
\lambda_{1}v^{(\alpha -1)/\alpha} \le I_{C}(v)\le \lambda_{2}v^{(\alpha -1)/\alpha}\qquad \forall\, v>v_{0}\,.
\end{equation}
In general the isoperimetric dimension is not well-defined, moreover the crucial estimate in \eqref{isopdim} is the first one, i.e., the lower bound on $I_{C}(v)$. Therefore it seems more convenient to formulate the problem in terms of an asymptotic isoperimetric inequality exploiting the growth rate of the volume of relative balls, as estimated by a non-decreasing function $V(r)$ depending only on the radius $r$. This is the approach followed by Coulhon and Saloff-Coste in \cite{MR1232845}, where isoperimetric-like inequalities are proved for graphs, groups, and manifolds in the large volume regime. Given a non-decreasing function $V(r)$, such that $V(r)\to+\infty$ as $v\to +\infty$ and $|B_{C}(x,r)|\ge V(r)$ for all $x\in C$, one introduces its \textit{reciprocal function} $\phi_{V}$ as
\[
\phi_{V}(v) = \inf \{r>0:\ V(r)\ge v\}\,.
\]
In Corollary \ref{thm:isopphi} the following result is proved: if $C$ is an unbounded convex body of uniform geometry, then for the optimal choice $V(r)= b(r) = \inf_{x\in C}|B_{C}(x,r)|$, and denoting by $\phi(v)$ the reciprocal function of $b(r)$, one has
\begin{equation}\label{eq:corisopphi}
(n+1)\frac{v}{\phi(v)}\ge I_{C}(v)\ge 24^{-(n+1)}\frac{v}{\phi(v)}\,.
\end{equation}
Clearly, in order to derive from \eqref{eq:corisopphi} an (asymptotic) estimate of $I_{C}(v)$ in terms of some explicit function of $v$ one would need to compute $b(r)$ with sufficient precision, which is not an easy task for a generic convex body $C$. We thus focus on a special class of three-dimensional convex bodies of revolution: we let
\[
C = \{(x,y) = (x_{1},x_{2},y)\in \rr^{3}:\ y\ge f(|x|)\}\,,
\]
where $f:[0,+\infty)\to[0,+\infty)$ is a convex function such that $f(0)=0$ and $\lim_{s\to +\infty}s^{-1}f(s) = +\infty$. Any such $C$ is an unbounded convex body whose asymptotic cone is a half-line, therefore one expects an isoperimetric dimension strictly smaller than $3$. If one further assumes $f$ strictly convex, of class $C^{3}(0,+\infty)$ and such that $f'''\le 0$ on $(0,+\infty)$, then Theorem \ref{thm:brorigin} proves that $b(r) = |B_{C}(0,r)|$, hence the explicit computation of $V(r) = b(r)$ becomes quite easy in this case. In particular, in Example \ref{ex:isopdimrev} we compute the isoperimetric dimension of 
\[
C_{a} = \{(x,y)\in \rr^{3}:\ y\ge |x|^{a}\}
\]
for $1<a\le 2$ and show that it is given by $\frac{a+2}{a}$, thus the isoperimetric dimension continuously changes from $3$ to $2$ as the parameter $a$ goes from $1$ to $2$.

Several interesting problems remain open. A first one is the existence of isoperimetric regions in an unbounded convex body $C$. Although we have given conditions on particular classes of convex bodies ensuring existence for any volume, like the one in Theorem~\ref{thm:existence}, it is not clear whether an example can be provided exhibiting an isoperimetric region in an asymptotic cylinder of $C$, but not in $C$. Another interesting open question is the range of possible isoperimetric dimensions for an unbounded convex body in $\rr^{n+1}$. It would be reasonable to expect that the range is exactly $[1,n+1]$. However, the results in Chapter~\ref{sec:isopdim} only show that the range includes the set $[1,2]\cup\{3\}$ in the three-dimensional case.

\chapter{Convex bodies and finite perimeter sets}
\label{sec:preliminaries}

\section{Convex bodies and local convergence in Hausdorff distance}
\label{sub:localh}

In this paper, a \emph{convex body} $C\subset\rr^{n+1}$ is defined as a closed convex set with non-empty interior. The interior of $C$ will be denoted by $\intt C$. In the following we shall distinguish between \emph{bounded} and \emph{unbounded} convex bodies. Given $x\in C$ and $r>0$, we define the \emph{intrinsic ball} $B_C(x,r)=B(x,r)\cap C$, and  the corresponding closed ball $\clb_C(x,r)=C\cap\clb(x,r)$. For $E \subset C$, the \emph{relative boundary} of $E$ in the interior of $C$ is $\ptl_{C}E=\ptl E\cap \intt {C}$. 

Given a convex set $C$, and $r>0$, we define $C_r=\{p\in\rr^{n+1}: d(p,C)\le r\}$. The set $C_r$ is the tubular neighborhood of radius $r$ of $C$ and is a closed convex set. Given two convex sets $C$, $C'$, we define their \emph{Hausdorff distance} $\hd(C,C')$ by
\begin{equation*}
\hd(C,C')=\inf\{r>0: C\subset (C')_r, C'\subset C_r\}.
\end{equation*}
We shall say that a sequence $\{C_i\}_{i\in\nn}$ of convex sets converges to a convex set $C$ in \emph{Hausdorff distance} if $\lim_{i\to\infty}\hd(C_i,C)=0$. 
Then we will say that a sequence $\{C_i\}_{i\in\nn}$ of convex bodies \emph{converges locally in Hausdorff distance} to a convex body $C$ if, for every open ball $B$ such that $C\cap B\neq\emptyset$, the sequence of bounded convex bodies $\{C_i\cap\clb\}_{i\in\nn}$ converges to $C\cap\clb$ in Hausdorff distance.

\begin{lemma}
\label{lem:localconv}
Let $\{C_i\}_{i\in\nn}$ be a sequence of convex bodies converging to a convex body $C$ in Hausdorff distance. Then $\{C_i\}_{i\in\nn}$ converges locally in Hausdorff distance to $C$.
\end{lemma}

\begin{proof}
Consider an open ball $B$ such that $C\cap B\neq\emptyset$. To prove that $C_i\cap\clb$ converges to $C\cap\clb$ in Hausdorff distance, we shall use the Kuratowski criterion: every point in $C\cap\clb$ is the limit of a sequence of points $x_i\in C_i\cap\clb$, and the limit $x$ of a subsequence $x_{i_j}\in C_{i_j}\cap\clb$ belongs to $C\cap\clb$, see \cite[Theorem~1.8.7]{sch}.

The second assertion is easy to prove: if $x\not\in C$, then there exists some $\rho>0$ such that $B(x,\rho)\cap (C+\rho\clb)=\emptyset$ (just take $\rho>0$ so that $\clb(x,2\rho)\cap C=\emptyset$). This is a contradiction since $x_{i_j}\in B(x,\rho)$ and $C_{i_j}\subset C+\rho\clb$ for $j$ large enough.

To prove the first one, take $x\in C\cap\clb$. Let $x_i$ be the point in $C_i$ at minimum distance from $x$. Such a point is unique by the convexity of $C_i$. It is clear that $\{x_i\}_{i\in\nn}$ converges to $x$ since $|x-x_i|\le\delta(C,C_i)\to 0$. If $x\in B$, then $x_i\in  C_i\cap B\subset C_i\cap\clb$ for $i$ large enough. If $x\in\ptl\clb$, then take a point $y\in C\cap B$ and a sequence $y_i\in C_i\cap\clb$ converging to $y$. The segment $[x_i,y_i]$ is contained in $C_i$. If $x_i\in\clb$ we take $z_i=x_i$. In case $x_i\not\in\clb$, we choose $z_{i}\in [x_{i},y_{i}]\cap \ptl B$ as in this case the intersection is nonempty (more precisely, the intersection contains exactly one point for $i$ large enough). We claim that $z_i\to x$. Otherwise there is a subsequence $z_{i_j}\in\ptl\clb$ (we may assume $z_{i_j}\neq x_{i_j}$) converging to some $q\in C\cap\ptl\clb$ different from $x$. But
\[
 \frac{q-y}{|q-y|}=\lim_{j\to\infty}\frac{z_{i_j}-y_{i_j}}{|z_{i_j}-y_{i_j}|}=\lim_{j\to\infty}\frac{x_{i_j}-y_{i_j}}{|x_{i_j}-y_{i_j}|}=\frac{x-y}{|x-y|}.
 \]
This implies that the points $q$ and $x$ lie in the same half-line leaving from $y$. Since $y\in B$ and $q$, $x\in\ptl\clb$, we get $q=x$, a contradiction.
\end{proof}

\begin{remark}
If condition $C\cap B\neq\emptyset$ is not imposed in the definition of local convergence in Hausdorff distance, Lemma~\ref{lem:localconv} does not hold. Simply consider the sequence $C_i:=\{x\in\rr^n:x_n\ge 1+1/i\}$, converging in Hausdorff distance to $C:=\{x\in\rr^n: x_n\ge 1\}$. Then $C_i\cap \clb(0,1)=\emptyset$ does not converge to $C\cap\clb (0,1)=\{(0,\ldots,0,1)\}$.
\end{remark}

A condition guaranteeing local convergence in Hausdorff distance is the following

\begin{lemma}
\label{lem:pointedconv}
Let $\{C_i\}_{i\in\nn}$ be a sequence of convex bodies, and $C$ a convex body. Assume that there exists $p\in\rr^n$ and $r_0\ge 0$ such that, for every $r>r_0$, the sequence $C_i\cap\clb(p,r)$ converges to $C\cap\clb(p,r)$ in Hausdorff distance. Then $C_i$ converges locally in Hausdorff distance to $C$.
\end{lemma}

\begin{proof}
Let $B$ be an open ball so that $C\cap B\neq\emptyset$. Choose $r>r_0$ so that $\clb\subset\clb(p,r)$. By hypothesis, $C_i\cap\clb(p,r)$ converges to $C\cap\clb(p,r)$ in Hausdorff distance. Lemma~\ref{lem:localconv} implies that $C_i\cap\clb(p,r)\cap\clb=C_i\cap\clb$ converges to $C\cap\clb(p,r)\cap\clb=C\cap\clb$ in Hausdorff distance.
\end{proof}

The following two lemmata will be used in Chapter \ref{sec:unbounded}.
\begin{lemma}
\label{lem:deltaFG}
Let $F$, $G$ be closed convex sets containing $0$, then
\begin{equation}
\label{eq:deltaFG}
\delta(F\cap\clb(0,r),G\cap\clb(0,r))\le\delta(F,G),\quad\text{for all }r>0. 
\end{equation}
\end{lemma}

\begin{proof}
To prove \eqref{eq:deltaFG} consider a point $x\in F\cap\clb(0,r)$ and its metric projection $y$ to $G$. If $y\in G\cap \clb(0,r)$ then the conclusion is trivial. Otherwise we have in particular that $|y|> r$. The point $\mu y$ in the line $0y$ at minimum distance from $x$ is a critical point of the function $\la\mapsto |x-\la y|^2$ and so $\mu=\escpr{x,y}/|y|^2$. Since $|y|>r$ and $|x|\le r$, by Schwarz's inequality we have $|\mu|\le |\escpr{x,y}|/|y|^2\le |x|/|y|<1$. In case $\mu>0$, the point $\mu y$ belongs to the segment $[0,y]$, included in $G$ by convexity. Since $|x-\mu y|<|x-y|$, we get a contradiction to the fact that $y$ is the metric projection of $x$ to $G$. In case $\mu<0$, Pythagoras' Theorem easily implies $|x|<|x-y|$, a contradiction since $0\in G$ would be closer to $x$ than $y$. Both contradictions imply $|y|\le r$ and so $y\in G\cap\clb(0,r)$. Then $d(x,G\cap\clb(0,r))=|x-y|\le \delta(F,G)$ and
\[
\sup_{x\in F\cap\clb(0,r)} d(x,G\cap\clb(0,r))\le \delta(F,G).
\]
The corresponding inequality inverting the roles of $F$ and $G$ holds, thus proving \eqref{eq:deltaFG}.
\end{proof}

\begin{lemma}\label{lem:haustrans}
Let $A\subset\rr^{n+1}$ be a convex body with $0\in A$ and let $r>0$ and $v\in \rr^{n+1}$ be such that $|v|<r/2$. Then
\begin{equation}\label{eq:hdshift}
\delta\big(A\cap \clb(0,r),(v+A)\cap \clb(0,r)\big) \le 2|v|\,.
\end{equation}
\end{lemma}
\begin{proof}
In order to prove \eqref{eq:hdshift} we first show that given $x\in (v+A)\cap\clb(0,r)$ there exists $a\in A\cap\clb(0,r)$ such that $|x-a|\le 2|v|$. Indeed, there exists $\tilde a\in A$ such that $x=\tilde a+v$ and thus $|\tilde a|\le |x|+|v|\le r+|v|$ and $|\tilde a -x|\le |v|$. Now, if $\tilde a\in A\cap \clb(0,r)$ then we are done, otherwise we have $r<|\tilde a|\le r+|v|$ so that setting $a = \frac{r\tilde a}{|\tilde a|}$ we note that $a\in A\cap\clb(0,r)$ by convexity of $A$ and the fact that $0\in A$, and therefore we find
\[
|a-x| \le |a-\tilde a| + |\tilde a - x| \le 2|v|,
\]
as wanted. This would imply
\begin{equation}
\label{eq:haus1}
\sup_{x\in (v+A)\cap\clb(0,r)} d(x,A\cap\clb(0,r))\le 2|v|.
\end{equation}

Now we prove that for any $a\in A\cap\clb(0,r)$ there exists $x\in (v+A)\cap\clb(0,r)$ such that $|a-x|\le 2|v|$. If $a+v\in \clb(0,r)$, then clearly $x=a+v$ satisfies the claim. If instead $a+v\notin \clb(0,r)$ then $|a|>r-|v|$ and thus by the assumption $|v|<r/2$ we deduce that $|a|> |v|$. By defining $x = v+\frac{|a|-|v|}{|a|}\,a$ we have $x\in (v+A)\cap\clb(0,r)$ and 
\[
|x-a| \le |v|+\Big|a- \frac{|a|-|v|}{|a|}\, a\Big| \le 2|v|.
\]
Hence we get
\begin{equation}
\label{eq:haus2}
\sup_{x\in A\cap\clb(0,r)} d(x,(v+A)\cap\clb(0,r))\le 2|v|.
\end{equation}
Inequalities \eqref{eq:haus1}, \eqref{eq:haus2} imply that \eqref{eq:hdshift} holds. This concludes the proof.
\end{proof}

We define the \emph{tangent cone} $C_{p}$ of a convex body $C$ at a given boundary point $p\in\ptl C$ as the closure of the set
\[
\bigcup_{\la > 0} h_{p,\la} (C).
\]
Tangent cones of convex bodies have been widely considered in convex geometry under the name of supporting cones \cite[\S~2.2]{sch} or projection cones \cite{MR920366}. From the definition it follows easily that $C_p$ is the smallest cone, with vertex $p$, that includes $C$. 

We define the \emph{asymptotic cone} $C_{\infty}$ of an unbounded  convex body $C$ by
\begin{equation}
\label{eq:cinfty}
C_{\infty} = \bigcap_{\lambda > 0} \lambda C,
\end{equation}
where $\lambda C = \{\lambda x : x\in C \} $ is the image of $C$ under the homothety of center $0$ and ratio $\la$. If $p\in\rr^{n+1}$ and $h_{p,\la}$ is the homothety of center $p$ and ratio $\la$, defined as $h_{p,\la}(x) = p+\la(x-p)$, then $\bigcap_{\la>0} h_{p,\la}(C)=p+C_\infty$ is a translation of $C_\infty$. Hence the shape of the asymptotic cone is independent of the chosen origin. When $C$ is bounded the set $C_{\infty}$ defined by \eqref{eq:cinfty} is a point.
Observe that ${\la}C$ converges, locally in Hausdorff sense, to the asymptotic cone $C_{\infty}$ \cite{bbi} and hence it satisfies $\dim C_\infty\le\dim C$. We shall say that the asymptotic cone is \emph{non-degenerate} if $\dim C_{\infty}=\dim C$. From the definition it follows easily that $p+C_\infty\subset C$ whenever $p\in C$, and that $p+C_{\infty}$ is the largest cone, with vertex $p$, included in $C$.

The \emph{volume} of a measurable set $E\subset \rr^{n+1}$ is defined as the Lebesgue measure of $E$ and will be denoted by $|E|$. The $r$-dimensional Hausdorff measure in $\rr^{n+1}$ will be denoted by $\h^r$. We recall the well-known identity $|E| = \h^{n+1}(E)$ for all measurable $E\subset \rr^{n+1}$. 

\begin{lemma}
\label{lem:doubling}
Let $C\subset \rr^{n+1}$ be an unbounded convex body. Given $r>0$, $\la>1$, we have
\begin{equation}
\label{eq:doubling}
|B_C(x,\la r)|\le\la^{n+1}|B_C(x,r)|,
\end{equation}
for any $x\in C$. In particular, $C$ is a doubling metric space with constant $2^{-(n+1)}$.
\end{lemma}

\begin{proof}
Since $\la>1$, the convexity of $C$ implies $B_C(x,\la r)\subset h_{x,\la}(B_C(x,r))$, where $h_{x,\la}$ is the homothety of center $x$ and ratio $\la$.
\end{proof}

Using Brunn-Minkowski Theorem we can prove the following concavity property for the power $1/(n+1)$ of the volume of relative balls in $C$. 

\begin{lemma}
\label{lemma:fxrconcave}
Let $C\subset\rr^{n+1}$ be a convex body. Then the function $F:C\times\rr^+\to\rr$ defined by $F(x,r):=|\overline{B}_C(x,r)|^{1/(n+1)}$ is concave. 
\end{lemma}

\begin{proof}
Take $(x,r), (y,s)\in C\times\rr^+$, and $\la\in [0,1]$. Assume $z\in \la \overline{B}_C(x,r)+(1-\la)\overline{B}_C(y,s)$. Then there exist $z_1\in \overline{B}_C(x,r)$, $z_2\in \overline{B}_C(y,s)$ such that $z=\la z_1+(1-\la) z_2$. The point $z$ belongs to $C$ by the convexity of $C$. Moreover
\[
|z-(\la x+(1-\la) y|\le \la |z_1-x|+(1-\la)|z_2-y|\le \la r+ (1-\la)s.
\]
This implies $z\in \overline{B}_C(\la x+(1-\la) y,\la r+(1-\la) s)$ and so
\begin{equation}
\label{eq:inclusion0}
\la \overline{B}_C(x,r)+(1-\la) \overline{B}_C(y,s)\subset \overline{B}_C(\la x+(1-\la)y,\la r+(1-\la) s).
\end{equation}
So we obtain from \eqref{eq:inclusion0} and the Brunn-Minkowski inequality \cite[Thm.~6.11]{sch}
\begin{align*}
F(\la x+(1-\la) y,\la r+(1-\la) s)&=|\overline{B}_C(\la x+(1-\la) y,\la r+(1-\la) s)|^{1/(n+1)}
\\
&\ge |\la \overline{B}_C(x,r)+(1-\la) \overline{B}_C(y,s)|^{1/(n+1)}
\\
&\ge \la |\overline{B}_C(x,r)|^{1/(n+1)}+(1-\la)|\overline{B}_C(y,s)|^{1/(n+1)}
\\
&=\la F(x,r)+(1-\la) F(y,s),
\end{align*}
which proves the concavity of $F$. 
\end{proof}

Given a convex set $C\subset\rr^{n+1}$, and $x\in \ptl C$, we shall say that $u\in\rr^{n+1}\setminus\{0\}$ is \emph{an outer normal vector} to $C$ at $x$ if $C$ is contained in the closed half-space $H_{x,u}^-:=\{y\in\rr^{n+1}:\escpr{y-x,u}\le 0\}$. The set $H_{x,u}^-$ is a supporting half-space of $C$ at $x$ and the set $\{y\in\rr^{n+1}: \escpr{y-x,u}=0\}$ is a supporting hyperplane of $C$ at $x$, see \cite[\S~1.3]{sch}. The \emph{normal cone} of $C$ at $x$, denoted by $N(C,x)$, is the union of $\{0\}$ and all outer normal vectors of $C$ at $x$, see \cite[\S~2.2]{sch}.

Given a convex function $f:\Om\to\rr$ defined on a convex domain $\Om\subset\rr^n$, and a point $x\in\Om$, the \emph{subdifferential} of $f$ at $x$ is the set
\[
\ptl f(x)=\{u\in\rr^n: f(y)\ge f(x)+\escpr{u,y-x}\ \text{for all }y\in\Om\},
\]
see \cite[p.~30]{sch} and also \cite{roc} and \cite{MR1058436}. Given a convex function, its epigraph $\{(x,y)\in\rr^{n}\times\rr: y\ge f(x)\}$ is a convex set. A vector $u$ belongs to $\ptl f(x)$ if and only if the vector $(u,-1)$ is an outer normal vector to the epigraph at the point $(x,f(x))$.

For future reference we shall need a technical lemma about the Painlev\'e-Kuratowski convergence of the graphs of the subdifferentials of convex functions that locally converge to a convex function. This lemma, that we state and prove for the reader's convenience, is well-known for convex functions defined on Banach spaces (see \cite{AttouchBeer}).

\begin{lemma}\label{lem:maybeattouch}
Let $\{f_{i}\}_{i}$ be a sequence of convex functions defined on some fixed ball $B_{R}\subset \rr^{n}$, that uniformly converge to a convex function $f$. Then for any $x\in B_{R/2}$ and $u\in \ptl f(x)$ there exist sequences $\{x_{i}\}_{i}\subset B_{R}$ and $\{u_{i}\}_{i}$ such that $u_{i}\in \ptl f_{i}(x_{i})$ and $(x_{i},u_{i})\to (x,u)$ as $i\to\infty$.
\end{lemma}
\begin{proof}
We split the proof in two steps.

\textit{Step one.} We show that for any $r,\eps>0$ there exists $i_{r,\eps}$ such that for all $i\ge i_{r,\eps}$ we can find $y_{i}\in B_{r}(x)$ and $u_{i}\in \ptl f_{i}(y_{i})$ such that $|u-u_{i}|\le 2r\eps$. 

Up to a translation we may assume that $x=0$. Let us fix $\eps>0$ and define $f^{\eps}(y) = f(y) + \eps |y|^{2}$ and, similarly, $f^{\eps}_{i}(y) = f_{i}(y) + \eps |y|^{2}$. By assumption we have $f(y)\ge f(0) + \langle u,y\rangle$, hence for all $y\neq 0$ we have
\begin{equation}\label{strictlyconvex}
f^{\eps}(y) \ge f(0) + \langle u,y\rangle + \eps |y|^{2} > f^{\eps}(0) + \langle u,y\rangle\,,
\end{equation}
which means in particular that $u\in \ptl f^{\eps}(0)$. We now define the set
\[
N_{i,\eps} = \left\{y\in B_{r}:\ f_{i}^{\eps}(y) \le f^{\eps}(0) + \langle u,y\rangle + \frac{r^{2}\eps}{16}\right\}\,.
\]
We notice that, since $f_{i}^{\eps}\to f^{\eps}$ uniformly on $B_{r}$ as $i\to\infty$, the set $N_{i,\eps}$ is not empty and contained in $B_{r/2}$ for $i$ large enough. Indeed, up to taking $i$ large enough we can assume that $\sup_{y\in B_{r}}|f_{i}(y) - f(y)|\le \frac{\eps r^{2}}{16}$, so that we obtain for any $y\in N_{i,\eps}$
\begin{align*}
f(0) + \langle u,y\rangle +\eps|y|^{2} - \frac{\eps r^{2}}{16} &\le f(y) + \eps|y|^{2} - \frac{\eps r^{2}}{16}\\
&\le f_{i}(y) + \eps|y|^{2} = f_{i}^{\eps}(y)\\
&\le f(0) + \langle u,y\rangle + \frac{\eps r^{2}}{16}\,,
\end{align*}
whence the inequality $|y|^{2}\le \frac{r^{2}}{8} \le \frac{r^{2}}{4}$. 
Now set 
\[
t_{i,\eps} = \sup\{t\in \rr:\ f_{i}^{\eps}(y)\ge f^{\eps}(0) + \langle u,y\rangle + t,\ \forall\, y\in B_{r}\}\,.
\]
We remark that $-\infty<t_{i,\eps}<\frac{r^{2}\eps}{16}$, as $N_{i,\eps}$ is nonempty and contained in $B_{r/2}$ for $i$ large enough, moreover its value is obtained by minimizing the function $f_{i}^{\eps}(y) - f^{\eps}(0) - \langle u,y\rangle$ on $B_{r}$, so that there exists $y_{i}\in \clb_{r/2}$ such that for all $y\in B_{r}$ we have
\[
f_{i}^{\eps}(y) - f^{\eps}(0) - \langle u,y\rangle \ge t_{i,\eps} = f_{i}^{\eps}(y_{i}) - f^{\eps}(0) - \langle u,y_{i}\rangle\,,
\]
that is, 
\[
f_{i}^{\eps}(y) \ge f_{i}^{\eps}(y_{i}) + \langle u,y-y_{i}\rangle\,,
\]
which shows that $u\in \ptl f_{i}^{\eps}(y_{i})$. Now let us define $u_{i} = u - 2\eps y_{i}$ and notice that
\begin{align*}
f_{i}(y) &\ge f_{i}(y_{i}) +2\eps \langle y,y_{i}-y\rangle + \langle u,y-y_{i}\rangle\\
&= f_{i}(y_{i}) + \langle u_{i}, y-y_{i}\rangle -2\eps |y-y_{i}|^{2}\\
&= f_{i}(y_{i}) + \langle u_{i}, y-y_{i}\rangle + o(|y-y_{i}|)\qquad \text{as }y\to y_{i}.
\end{align*}
Replacing $y$ by $y_i+t(y-y_i)$, for $t\in [0,1]$, dividing both sides of the inequality by $t$ and taking limits when $t\downarrow 0$ we get
\begin{equation}
\label{eq:dirderf_i}
f_i'(y_i;y-y_i)\ge\escpr{u_i,y-y_i},
\end{equation}
where $f'(y_i;y-y_i)$ is the directional derivative of $f$ at the point $y_i$ in the direction of $y-y_i$. By the convexity of $f_i$, see Theorem~24.1 in \cite{roc}, we have
\[
f_i(y)-f_i(y_i)\ge f_i'(y_i;y-y_i),
\]
which together with \eqref{eq:dirderf_i} implies that $u_i$ is a subgradient of $f_i$ at the point $y_i$. See also Proposition~2.2.7 in \cite{MR1058436} for a variant of this argument. Finally, $|u_{i}-u| = 2\eps|y_{i}|\le 2r\eps$, as wanted.
\medskip

\textit{Step two.} In order to complete the proof of the lemma, we argue by contradiction assuming the existence of $r_{0},\eps_{0}>0$ and of a subsequence $\{f_{i_{k}}\}_{k}$, such that for all $y\in B_{r_{0}}(x)$ and all $v\in \ptl f_{i_{k}}(y)$ we have $|v-u|\ge \eps_{0}$. By applying Step one to the subsequence, with parameters $r=r_{0}$ and $\eps = \eps_{0}/(4r_{0})$, we immediately find a contradiction.
\end{proof}

\section{Finite perimeter sets and isoperimetric profile}

Given a convex body $C$ and a measurable set $E \subset C$, we define the \emph{relative perimeter} of $ E$ in $\intt(C)$ by
\[
\pp_C(E) = \sup \Big \{ \int_E\divv \xi\, d{\h}^{n+1}, \xi \in \Gamma_0(C) ,\, |\xi| \le 1 \Big \},
\]
where $\Gamma_0(C)$ is the set of smooth vector fields with compact support in $\intt (C)$. We shall say that $E$ has \emph{finite perimeter} in $C$ if $\pp_C(E)<\infty$. When $C = \rr^{n+1}$ we simply write $P(E)$ instead of $P_{\rr^{n+1}}(E)$. We refer the reader to Maggi's book \cite{maggi} for an up-to-date reference on sets of finite perimeter. In particular, we shall denote the \emph{reduced boundary} of $E$ in the interior of $C$ by $\ptl^*E$, see Chapter~15 in \cite{maggi}.

For future reference we denote by $\omega_{n+1}$ the volume of the unit ball in $\rr^{n+1}$, and notice that its perimeter (i.e., the $n$-dimensional Hausdorff measure of its boundary) is $(n+1)\omega_{n+1}$.

We define the \emph{isoperimetric profile} of $C$ by
\begin{equation*}
I_C(v)=\inf \Big \{ \pp_C(E) : E \subset C, \vol{E} = v \Big \}.
\end{equation*}
We shall say that $E \subset C$ is an \emph{isoperimetric region} if $\pp_C(E)=I_C(\vol{E})$. We also recall in the next lemma a scaling property of the isoperimetric profile, which directly follows from the homogeneity of the perimeter with respect to homotheties.
\begin{lemma}
\label{lem:link I laC I C}
Let $C$ be a convex body, and $\la> 0$. Then, for all $v>0$,
\begin{equation}
\label{eq:proflac}
I_{\lambda C}(\la^{n+1}v)={\lambda}^nI_{C}(v),
\end{equation}
\end{lemma}
For future reference, we define the \emph{renormalized isoperimetric profile} of $C\subset\rr^{n+1}$ by
\begin{equation*}
Y_C:=I_C^{(n+1)/n}.
\end{equation*}

The known results on the regularity of isoperimetric regions are summarized in the following Theorem.
\begin{theorem}[{\cite{MR684753}, \cite{MR862549}, \cite[Thm.~2.1]{MR1674097}}]
\label{thm:n-7}
Let $C\subset \rr^{n+1}$ be a $($possible unbounded$)$  convex body and $E\subset C$ an isoperimetric region.
Then $\ptl E\cap\intt(C) = S_0\cup S$, where  $S_0\cap S=\emptyset$ and
\begin{enum}
\item $S$ is an embedded $C^{\infty}$ hypersurface of constant mean curvature.\item $S_0$ is closed and $H^{s}(S_0)=0$ for any $s>n-7$.
\end{enum}
Moreover, if the boundary of $C$ is of class $C^{2,\alpha}$ then $\cl{\ptl E\cap\intt(C)}=S\cup S_0$, where
\begin{enum}
\item[(iii)] $S$ is an embedded $C^{2,\alpha}$ hypersurface of constant mean curvature
\item[(iv)] $S_0$ is closed and $H^s(S_0)=0$ for any $s>n-7$
\item[(v)] At points of $ S \cap \ptl C$,  $ S$ meets $\ptl C$ orthogonally.
\end{enum}
\end{theorem}
We remark that the existence of isoperimetric regions in a compact convex body $C$ is a straightforward consequence of the lower semicontinuity of the relative perimeter and of the compactness properties of sequences of sets with equibounded (relative) perimeter. On the other hand, the existence of isoperimetric regions in an unbounded convex body $C$ is a quite delicate issue, that will be discussed later on. For future reference we recall in the following proposition a useful property related to minimizing sequences for the relative perimeter.
\begin{proposition}
\label{prp:seqbound}{\cite[Remark 3.2]{MR3385175}}
Let $C\subset \rr^{n+1}$ be an unbounded convex body and $v>0$. Then there exists a minimizing sequence, for volume $v$, consisting of bounded sets.
\end{proposition}

Let now $K\subset \rr^{n+1}$ be a closed solid cone with vertex $p$. Let $\alpha(K)=\hh^n(\ptl{B}(p,1)\cap \intt(K))$ be the \emph{solid angle} of $K$. If $K$ is also convex then it is known that the intrinsic balls centered at the vertex are isoperimetric regions in $K$, \cite{lions-pacella}, \cite{r-r}, and that they are the only ones \cite{FI} for general convex cones, without any regularity assumption on the boundary. The invariance of $K$ by dilations centered at the  vertex $p$ yields
\begin{equation}
\label{eq:solangle}
P_{K}(\clb_{K}(p,r))={\alpha(K)}^{1/(n+1)}\,(n+1)^{n/(n+1)}\,|\clb_{K}(p,r)|^{n/(n+1)}
\end{equation}

And if $K$ is also convex, then by the above equality and the fact that intrinsic balls centered at $p$ are isoperimetric, we obtain,
\begin{equation}
\label{eq:isopsolang}
I_K(v)={\alpha(K)}^{1/(n+1)}\,(n+1)^{n/(n+1)}v^{n/(n+1)}=I_K(1)\,v^{n/(n+1)}.
\end{equation}
Consequently the isoperimetric profile of a convex cone is completely determined by its solid angle.

We end this section with some definitions. We say that $C$ satisfies an \emph{$m$-dimensional isoperimetric inequality} if there are positive constants $\la$, $v_0$ such that
\[
I_C(v)\ge \la v^{(m-1)/m}\quad\text{for all }v\ge v_0.
\]
The \emph{isoperimetric dimension of $C$} is the supremum of the real $m>0$ such that $C$ satisfies an $m$-dimensional isoperimetric inequality.

These definitions have sense in any metric measure space, see Gromov \cite[Chap.~6.B, p.~322]{grom}, Coulhon and Saloff-Coste \cite{MR1232845}, and Chavel \cite{MR1849187}. It is immediate to show that if $C$ satisfies an $m$-dimensional isoperimetric inequality then, for any intrinsic ball $B_C(x,r)$, we have
\[
\la_n r^{n}\ge P(B_C(x,r))\ge \la |B_C(x,r)|^{(m-1)/m},
\] 
where $\la_n=H^n(\esf^n(1))$. Hence the growth of the volume of intrinsic balls is uniformly controlled (i.e., it does not depend on the center of the ball) in terms of $r^{nm/(m-1)}$.

\chapter{Unbounded convex bodies of uniform geometry}
\label{sec:unbounded}

In this chapter we collect various key definitions and results concerning the asymptotic properties of unbounded convex bodies. In particular we will show the crucial role played by property \eqref{eq:mainhyp} in order to ensure the non-triviality of the isoperimetric profile function $I_{C}$ of an unbounded convex body $C$.

\section{Asymptotic cylinders}
\label{sec:asymptoticcyl}
We start by introducing the notion of \textit{asymptotic cylinder} of $C$, which requires some preliminaries. Following Schneider \cite[\S~1.4]{sch}, we shall say that a subset $A\subset\rr^{n+1}$ is \emph{line-free} if it does not contain a line. According to Lemma~1.4.2 in \cite{sch}, every closed convex set $A\subset\rr^{n+1}$ can be written as $B\oplus V$, where $V$ is a linear subspace of $\rr^{n+1}$ and $B$ is a line-free closed convex set contained in a linear subspace orthogonal to $V$. Throughout this work, a \emph{convex cylinder} will be a set of this type, i.e., the direct sum of a \emph{non-trivial} linear subspace $V\subset\rr^{n+1}$ and a closed convex set $B$ contained in a linear subspace orthogonal to $V$.

The following two lemmas will play an important role in the sequel. Given $v\in\rr^{n+1}\setminus\{0\}$, we shall denote by $L(v)$ the vector space generated by $v$. Thus $L(v):=\{\la v:\la\in\rr\}$.

\begin{lemma}
\label{lem:gencyl}
Let $C\subset\rr^{n+1}$ be an unbounded convex body. Consider a bounded sequence $\{\la_i\}_{i\in\nn}$ of positive real numbers and an unbounded sequence of points $x_i\in\la_iC$. Further assume that
\begin{enumerate}
\item[(i)] $\lim_{i\to\infty}\frac{x_i}{|x_i|}=v$,
\item[(ii)] $-x_i+\la_iC\to K$ locally in Hausdorff distance.
\end{enumerate}
Then $L(v)\subset K$ and thus $K$ is a convex cylinder.
\end{lemma}

\begin{proof}
Since $\{\la_i\}_{i\in\nn}$ is bounded and $\{x_i\}_{i\in\nn}$ is unbounded, the sequence $\{\la_i^{-1}x_i\}_{i\in\nn}$ of points in $C$ is unbounded. For any point $x_0\in C$ we have
\[
v_i:=\frac{\la_i^{-1}x_i-x_0}{|\la_i^{-1}x_i-x_0|}=\frac{\frac{x_i}{|x_i|}-\frac{x_0}{|\la_i^{-1}x_i|}}{|1-\frac{x_0}{|\la_i^{-1}x_i|}|}\to v.
\]
The convexity of $C$ then implies that the half-line $x_0+\{\la v:\la\ge 0\}$ is contained in $C$ for any $x_0\in C$. Scaling by $\la_i$ we get that $z_i+\{\la v:\la \ge 0\}$ is contained in $\la_iC$ for any $z_i\in \la_i C$. Taking $z_i=x_i$, the local convergence in Hausdorff distance of $-x_i+\la_iC$ to $K$ implies that $\{\la v:\la \ge0\}$ is contained in $K$.

On the other hand, given $x_0\in C$ fixed, the convexity of $C$ implies that $[x_0,\la_i^{-1}x_i]\subset C$ and so $[\la_ix_0-x_i,0]\subset -x_i+\la_iC$. The segment $[\la_ix_0-x_i,0]$ is contained in the half-line $\{-\la v_i:\la\ge 0\}$ and their lengths $|\la_ix_0-x_i|\to +\infty$. Hence its pointwise limit is the half-line $\{-\la v:\la\ge 0\}$, which is contained in $K$ because of the local convergence in Hausdorff distance of $-x_i+\la_iC$ to $K$.

Summing up, we conclude that $L(v)\subset K$. By \cite[Lemma~1.4.2]{sch}, $K$ is a convex cylinder.
\end{proof}

\begin{lemma}
\label{lem:cyl}
Let $C\subset\rr^{n+1}$ be an unbounded convex body, and let $\{x_i\}_{i\in\nn}\subset C$ be a divergent sequence. Then $\{-x_i+C\}_{i\in\nn}$ subconverges locally in Hausdorff distance to an unbounded closed convex set $K$. Moreover, $K$ is a convex cylinder.
\end{lemma}

\begin{proof}
First we set $C_{i} = -x_{i}+C$. Then we observe that, for every $j\in\nn$, the sequence of convex bodies $\{C_{i}\cap \clb(0,j)\}_{i\in\nn}$ is bounded in Hausdorff distance. By Blaschke's Selection Theorem \cite[Thm.~1.8.4]{sch}, there exists a convergent subsequence. By a diagonal argument and Lemma~\ref{lem:pointedconv}, we obtain that $C_i$ subconverges locally in Hausdorff distance to a limit convex set $K$.

Passing again to a subsequence we may assume that $x_i/|x_i|$ subconverges to some $v\in\esf^n$. Now we apply Lemma~\ref{lem:gencyl} to the selected subsequence (taking $\la_i=1$ for all $i\in\nn$) to conclude that $K$ is a convex cylinder.
\end{proof}

Given an unbounded convex body $C$, we shall denote by $\mathcal{K}(C)$ the set of convex cylinders that can be obtained as local Hausdorff limits of sequences $\{C_i\}_{i\in\nn}$, where $C_i=-x_i+C$, and $\{x_i\}_{i\in\nn}\subset C$ is a divergent sequence. Any element of $\mathcal{K}(C)$ will be called an \emph{asymptotic cylinder} of $C$.

\begin{remark}
\label{rem:bdycylinder}
In a certain sense, it is enough to consider diverging sequences contained in the boundary $\ptl C$ to obtain asymptotic cylinders. Let $\{x_i\}_{i\in\nn}$ be a diverging sequence of points in $C$, and assume that $-x_i+C$ locally converges in Hausdorff distance to an asymptotic cylinder $K$ of $C$. Take a sequence of points $y_i\in\ptl C$ such that $r_i:=d(x_i,\ptl C)=|x_i-y_i|$.

If $\limsup_{i\to\infty} r_i=+\infty$ then, after passing to a subsequence, we may assume that $r_i$ is increasing and $\lim_{i\to\infty} r_i=+\infty$. As $\clb(x_i,r_i)\subset C$ we have $\clb(0,r_i)\subset -x_i+C$ and, since $r_i$ has been taken increasing we have $\clb(0,r_i)\subset -x_j+C$ for all $j\ge i$. Taking limits in $j$ we get $\clb(0,r_i)\subset K$ for all $i\in\nn$. Since $\lim_{i\to\infty}r_i=+\infty$ we have $K=\rr^{n+1}$.

If $\limsup_{i\to\infty}r_i<+\infty$ then the sequence $x_i-y_i$ is bounded. Passing to a subsequence we may assume that $y_i-x_i$ converges to $z$, and we get $K=z+K'$, where $K'$ is the local Hausdorff limit of a subsequence of $-y_i+C$.
\end{remark}

The following lemma is a refinement of Lemma \ref{lem:cyl} above.

\begin{lemma}
\label{lem:ciave}
Let $\{C_i\}_{i\in\nn}$ be a sequence of unbounded convex bodies converging in Hausdorff distance to an unbounded convex body $C$.
\begin{enumerate}
\item[(i)] Let $\{x_i\}_{i\in\nn}$ be a divergent sequence with $x_i\in C_i$ for $i\in\nn$. Then there exists an~asymptotic cylinder $K\in\mathcal{K}(C)$ such that $\{-x_i+C_i\}_{i\in\nn}$ subconverges locally in Hausdorff distance to $K$.
\item[(ii)] Let $K_i\in \K(C_i)$ for $i\in \nn$. Then there exists $K\in \K(C)$ such that $\{K_i\}_{i\in\nn}$ subconverges locally in Hausdorff distance to $K$.
\end{enumerate}
\end{lemma}

\begin{proof}
For each $i\in\nn$, let $z_i$ be the metric projection of $x_i$ to $C$. Since $|x_i - z_i|\le\delta(C_i,C)\to 0$, the sequence $\{z_i\}_{i\in\nn}$ is also divergent. By Lemma~\ref{lem:cyl}, the sequence $\{-z_i+C\}_{i\in\nn}$ subconverges locally in Hausdorff distance to some $K\in\mathcal{K}(C)$. Let us prove the first claim, i.e., that $\{-x_i+C_i\}_{i\in\nn}$ subconverges locally in Hausdorff distance to $K$. By Lemma~\ref{lem:pointedconv} it is enough to show that, for any $r>0$, the sequence $\{(-x_i+C_i)\cap\clb(0,r)\}_{i\in\nn}$ subconverges in Hausdorff distance to $K\cap\clb(0,r)$.

First, one can observe that the sequence $\{(-z_i+C)\cap \clb(0,r)\}_{i\in\nn}$ subconverges in Hausdorff distance to $K\cap\clb(0,r)$, since for some natural $j>r$ we know that $\{(-z_i+C)\cap\clb(0,j)\}_{i\in\nn}$ subconverges in Hausdorff distance to $K\cap\clb(0,j))$, and we can apply Lemma~\ref{lem:localconv}.

Second, the Hausdorff distance between $(-x_i+C_i)\cap\clb(0,r)$ and $(-z_i+C)\cap\clb(0,r)$ converges to $0$ by Lemma~\ref{lem:deltaFG} since
\[
\delta((-x_i+C_i)\cap\clb(0,r),(-z_i+C)\cap\clb(0,r))\le\delta(-x_i+C_i,-z_i+C),
\]
and the term on the right converges to $0$ when $i\to\infty$ because of the Hausdorff convergence of $C_i$ to $C$ and the convergence of $x_i-z_i$ to $0$. From these two observations, (i) follows.

Let us now prove (ii). For every $i\in\nn$, Lemma~\ref{lem:cyl} implies the existence of a divergent sequence $\{x_j^i\}_{j\in\nn}\subset C_i$ so that $\{-x_j^i+C_i\}_{j\in\nn}$ converges locally in Hausdorff distance to $K_i$. For every $i\in\nn$, we choose increasing $j(i)$ so that the Hausdorff distance between $K_i\cap \clb(0,i)$ and $(-x_{j(i)}^i+C_i)\cap \clb(0,i)$ is less than $1/i$. If we fix some positive $r>0$, Lemma~\ref{lem:deltaFG} implies
\begin{multline*}
\limsup_{i\to\infty}\delta(K_i\cap\clb(0,r),(-x_{j(i)}^i+C_i)\cap\clb(0,r))
\\
\le\lim_{i\to\infty}\delta(K_i\cap\clb(0,i),(-x_{j(i)}^i+C_i)\cap\clb(0,i))=0.
\end{multline*}

Let $z_{i}$ be the metric projection of $x_{j(i)}^i$ onto $C$. By assumption we have $|z_{i} - x_{j(i)}^i|\le\delta (C,C_i)\to 0$. Hence, for any $r>0$
\[
\lim_{i\to\infty}\delta((-x_{j(i)}^i+C_i)\cap\clb(0,r),(-z_i+C)\cap\clb(0,r))=0.
\]

Finally, Lemma~\ref{lem:cyl} implies that the sequence $\{-z_i+C\}_{i\in\nn}$ subconverges locally in Hausdorff distance to some $K\in\mathcal{K}(C)$.

An application of the triangle inequality to the sets $K_i$, $-x_{j(i)}^i+C_i$, $-z_i+C$ and $K$, intersected with $\clb(0,r)$, yields (ii).
\end{proof}

The following result implies that certain translations of asymptotic cylinders are also asymptotic cylinders. This is not true in general, as shown by horizontal translations by large vectors of asymptotic cylinders of cylindrically bounded convex bodies.

\begin{lemma}
\label{lem:zK}
Let $C$ be an unbounded convex body. If $K\in \K(C)$ and $z\in K$, then $-z+K\in \K(C)$.
\end{lemma}
\begin{proof}
As $K\in\mathcal{K}(C)$, there exists an unbounded sequence $\{x_i\}_{i\in\nn}$ of points in $C$ such that $-x_i+C$ converges locally in Hausdorff distance to $K$. We can use the criterion in Lemma~\ref{lem:pointedconv} (taking balls centered at the point $-z$) to infer that $-(x_i+z)+C$ converges locally in Hausdorff distance to $-z+K$. 

Let us apply Kuratowski criterion and find a sequence $\{y_{i}\}_{i\in \nn}$ such that $y_{i}\in -x_{i} +C$ and $y_{i}\to z$ as $i\to\infty$. Since $x_{i}$ diverges and $y_{i}$ converges, the sequence $x_{i}+y_{i}\in C$ is divergent, thus we let $C_{i} = -(x_{i}+y_{i}) +C$ and show that $C_{i}$ converges locally in Hausdorff distance to $-z+K$, as $i\to\infty$. Let us fix $r>0$ and notice that $0\in C_{i}$ for all $i$, then setting $\delta_{r}(F,G) = \delta(F\cap\clb(0,r),G\cap\clb(0,r))$ we have
\begin{align*}
\delta_{r}(C_{i},-z+K) \le \delta_{r}(C_{i},-(z -y_{i}) + C_{i}) + \delta_{r}(-z +(-x_{i} + C),-z+K)
\end{align*}
so that by Lemmata \ref{lem:pointedconv} and \ref{lem:haustrans} we find that the left-hand side of the above inequality is infinitesimal as $i\to\infty$, which concludes the proof.
\end{proof}

The next Lemma generalizes \cite[Lemma 6.1]{MR3335407}%
\begin{lemma}
\label{lem:losemcontangcon}
Let $C\subset \rr^{n+1}$ be an unbounded convex body. Then there exists~$K\in\{C\}\cup\mathcal{K}(C)$ and $p\in K$ such that 
\[
\alpha(K_p)=\min\{\alpha(L_q) : L \in \{C\} \cup \mathcal{K}(C), q\in L\}.
\]
\end{lemma}

\begin{proof}
For every solid cone $V\subset \rr^{n+1}$ with vertex $p$ the co-area formula implies that $|V\cap\clb(p,1)|=(n+1)^{-1}\alpha(V)$. Our problem is then equivalent to minimizing $|L_p\cap\clb(p,1)|$ when $L\in\{C\}\cup\mathcal{K}(C)$ and $p\in L$.

Consider a sequence $K_i\in \{C\}\cup \mathcal{K}(C)$ and a sequence of points $p_i\in K_i$ such that
\[\lim_{i\to\infty}\alpha((K_i)_{p_i})=\inf\{\alpha(L_q) : L \in \{C\} \cup \mathcal{K}(C), q\in L\}.
\]
Let us see that $L_i:=-p_i+K_i$ subconverges locally in Hausdorff distance either to a translation of $C$ or to an asymptotic cylinder of $C$. Assume first that there is a subsequence so that $K_i=C$. If the corresponding subsequence $p_i$ is bounded, it subconverges to some $p\in C$ and then $L_i$ subconverges to $-p+C$. In case $p_i$ is unbounded then $L_i$ subconverges to an asymptotic cylinder $L$. So we can suppose that $K_i\neq C$ for all $i$. By Lemma~\ref{lem:zK}, $L_i\in\mathcal{K}(C)$ for all $i$. By Lemma~\ref{lem:ciave}(ii), $L_i$ subconverges to an asymptotic cylinder $L$.

Let us denote by $L$ the local limit in Hausdorff distance of a subsequence of $L_i$. The set $L$ is either $-p+C$, for some $p\in C$, or an asymptotic cylinder of $C$.  Passing again to a subsequence, the tangent cone of $L_i$ at the origin, $(L_i)_0$, locally converges in Hausdorff distance to a convex cone $L'\subset\rr^{n+1}$ with vertex $0$. Because of this convergence and the inclusion $L_i\subset(L_i)_0$, we get $L\subset L'$. Hence $L_0\subset L'$ since $L_0$ is the smallest cone including $L$. By the continuity of the volume with respect to Hausdorff convergence we have
\[
\vol{L_0\cap\clb(0,1)}\le\vol{L'\cap\clb(0,1)}\le\lim_{i\to\infty} \vol{(L_i)_0\cap\clb(0,1)}=\lim_{i\to\infty}\alpha((K_i)_{p_i}).
\]
Thus $\alpha(L_0)$ is a minimum for the solid angle.
\end{proof}

\begin{remark}
\label{rem:ICmin}
By \eqref{eq:isopsolang} the isoperimetric profiles of tangent cones which are minima of the solid angle function coincide. The common profile will be denoted by $I_{C_{\min}}$. By the above proof  we get that
\begin{equation*}
I_{C_{\min}}\le I_{K_p},
\end{equation*}
for every $K\in \{C\} \cup \mathcal{K}(C)$ and $p\in K$.
\end{remark}

Now we proceed to build an example of unbounded convex body $C$ for which the isoperimetric profile $I_C$ is identically zero. The following result is essential for the construction.

\begin{proposition}[{\cite[Prop.~6.2]{MR3441524}}]
\label{prp:ICleICp}
Let $C\subset\rr^{n+1}$ be a convex body and $p\in \ptl C$. Then every intrinsic ball in $C$ centered at $p$ has no more perimeter than an intrinsic ball of the same volume in $C_p$. Consequently
\begin{equation}
\label{eq:ICleICp}
I_C(v)\le I_{C_{p}}(v),
\end{equation}
for all $0< v<\vol{C}$.
\end{proposition}
\begin{proof}
Let $p\in\ptl C$, and $0<v<|C|$. Take $r>0$ so that $|\clb_C(p,r)|=v$. Let $L_p$ be the closed cone centered at $p$ subtended by $\clb(p,r)\cap C$. Then $L_p\subset C_p$ and, by convexity, $\clb_{L_p}(p,r)\subset \clb_C(p,r)$. By \eqref{eq:solangle} and \eqref{eq:isopsolang} we have
\begin{align*}
I_C(v)\le P_C(\clb_C(p,r))&=P_{L_p}(\clb_{L_p}(p,r))
\\
&=(n+1)^{n/(n+1)}\,\alpha(L_p)^{1/(n+1)}\,|\clb_{L_p}(p,r)|^{n/(n+1)}
\\
&\le (n+1)^{n/(n+1)}\,\alpha(C_p)^{1/(n+1)}\,|\clb_{L_p}(p,r)|^{n/(n+1)}
\\
&=I_{C_p}(|\clb_{L_p}(p,r)|)\le I_{C_p}(v).
\end{align*}
\end{proof}

\begin{remark}
\label{rem:half-plane}
A closed half-space $H\subset\rr^{n+1}$ is a convex cone with the largest possible solid angle. Hence, for any convex body $C\subset\rr^{n+1}$, we have
\begin{equation*}
I_C(v)\le I_H(v),
\end{equation*}
for all $0< v<\vol{C}$. 
\end{remark}

\begin{remark}
Proposition~\ref{prp:ICleICp} implies that $E\cap\ptl C\neq\emptyset$ when $E\subset C$ is isoperimetric. Since  in case $E\cap\ptl C$ is empty, then $E$ is an Euclidean ball. Moreover, as the isoperimetric profile of Euclidean space is strictly larger than that of the half-space, a set whose perimeter is close to the the value of the isoperimetric profile of $C$ must touch the boundary of $C$.
\end{remark}

\section{Convex bodies of uniform geometry}

The isoperimetric profile of an unbounded convex body can be identically zero, as shown by the following example. Note that some asymptotic cylinder has no interior points in this case. Such ``bad'' asymptotic cylinders are obtained by sequences $x_{i}$ such that $\min(|x_{i}|,d(x_{i},Q)) \to \infty$, where $Q$ is the half-cylinder.

\begin{figure}[h]
\begin{tikzpicture}
\draw (0,0,0) -- (6,0,0);
\draw (0,0,0) -- (0,5.4,0);
\draw (0,0,0) -- (0,0,4);
\draw[line width=1.5pt] (2,0) -- (2,4);
\draw[line width=1.5pt] (-2,0) -- (-2,4);
\draw[line width=1.5pt] (2,0) parabola (6,5);
\draw[line width=1.5pt] (0,0,0) ellipse (2 and 0.5);
\draw[line width=1.5pt] (0,4,0) ellipse (2 and 0.5);
\draw (-2.4,4) node {$Q$};
\draw (5.9,4) node {$P$};
\draw[fill=black] (4,0) circle (2pt);
\draw (4.2,-0.3) node {$x$};
\draw[fill=black] (4,1.25) circle (2pt);
\draw (4.2,0.9) node {$p_x$};
\draw (2.88,0) -- (5.6,3);
\draw[fill=black] (2.88,0) circle (2pt);
\draw (3.2,-0.3) node {$q_x$};
\draw(2.88,0) -- (-1,0.95);
\draw(2.88,0) -- (-1.4,-1.05);
\draw[fill=black] (1.4,0.37) circle (2pt);
\draw (1.6,0.7) node {$t_x^1$};
\draw[fill=black] (1.4,-0.37) circle (2pt);
\draw (1.6,-0.68) node {$t_x^2$};
\draw (-1,0.95) -- (1.72,3.95);
\draw (1.72,3.95) -- (5.6,3);
\draw (-1.4,-1.05) -- (1.32,2.05);
\draw (1.32,2.05) -- (5.6,3);
\draw (3.2,4) node {$A_x^1$};
\draw (3.2,2) node {$A_x^2$};
\end{tikzpicture}
\caption{Example~\ref{ex:ex}\label{fig:exex}}
\end{figure}

\begin{example}
\label{ex:ex}
We consider in $\rr^3$ the half-cylinder
\[
Q=\{(x,y,z)\in\rr^3: x^2+y^2\le 1, z\ge 0\},
\]
and the parabolic curve
\[
P=\{(x,y,z)\in\rr^3: z=(x-1)^2, y=0, x\ge 1\}
\]
(see Figure \ref{fig:exex}). 
Let $C$ be the closed convex envelope of $Q\cup P$. For a given coordinate $x>1$ the corresponding point on the parabola $P$ is denoted by $p_{x}=(x,0,(x-1)^2)$. The tangent line to $P$ at $p_{x}$ contained in the $y=0$ plane intersects the $x$-axis at the point $q_{x} = ((1+x)/2,0,0)$, which is of course in the $z=0$ plane and outside the unit disk $D = \{(x,y,z):\ x^{2}+y^{2}\le 1,\ z=0\}$. Therefore, one can consider the two tangent lines from $q_{x}$ to the boundary circle $\partial D$, meeting the circle at the two tangency points $t_{x}^{1} = (2/(1+x),\sqrt{1-4/(1+x)^2},0)$ and $t_{x}^{2}= (2/(1+x),-\sqrt{1-4/(1+x)^2},0)$. The (unique) affine planes $A_{x}^{1},A_{x}^{2}$ containing, respectively, the points $p_{x},q_{x},t_{x}^{1}$ and $p_{x},q_{x},t_{x}^{2}$ are supporting planes for $C$. The corresponding half-spaces bounded by $A_{x}^{1},A_{x}^{2}$ and containing $C$ are denoted by $H_{x}^{1},H_{x}^{2}$. It follows in particular that $p_{x}$ is a boundary point of $C$. The solid angle of the tangent cone $C_{p_{x}}$ of $C$ at $p_{x}$ is smaller than or equal to the solid angle of the wedge $W_{x}=H_{x}^{1}\cap H_{x}^{2}$, which trivially goes to $0$ as $x\to\infty$. By Proposition \ref{prp:ICleICp} we get for any $v>0$
\[
0\le I_C(v)\le\inf_{x>1} I_{C_{p_{x}}}(v)=0.
\]
Note that with some extra work it is possible to prove that $C_{p_{x}}=W_{x}$.
\end{example}

In the following proposition we give some conditions equivalent to the non-triviality of the isoperimetric profile.

We shall say that $C$ is a \emph{convex body of uniform geometry} if it is unbounded and for some $r_{0}>0$ there holds
\begin{equation}
\label{eq:mainhyp}
b(r_{0}):=\inf_{x\in C}|\clb_C(x,r_0)|>0\,.
\end{equation}

\begin{remark}
\label{rem:bconv}
By Lemma \ref{lemma:fxrconcave} one immediately deduce that $b(r)$ is a concave function (indeed it is the infimum of a family of concave functions).
\end{remark}

\begin{proposition}
\label{prop:maincond}
Let $C$ be an unbounded convex body. The following assertions are equivalent:
\begin{enumerate}
\item[(i)] for all $r>0$, $\inf_{x\in C} |\clb_C(x,r)|=b(r)>0$;
\item[(ii)] $C$ is a convex body of uniform geometry;
\item[(iii)] all asymptotic cylinders of $C$ are convex bodies;
\item[(iv)] for all $v>0$, $I_{C}(v) >0$;
\item[(v)] there exists $v_{0}>0$ such that $I_{C}(v_{0})>0$.
\end{enumerate}
Moreover, any of these conditions imply $\inf_{x\in\ptl C}\alpha(C_x)=\alpha_0>0$.
\end{proposition}

\begin{proof}
We shall first prove that (i), (ii) and (iii) are equivalent, then prove the implications (iv) $\Rightarrow$ (v), (v) $\Rightarrow$ (iii), and (ii) $\Rightarrow$(iv).

The fact that (i) implies (ii) is obvious. Let us prove that (ii) implies (iii). Let $K\in\mathcal{K}(C)$ be an asymptotic cylinder obtained as the limit of the sequence $\{-x_i+C\}_{i\in\nn}$ under local convergence in Hausdorff distance. Then
\[
|K\cap\clb(0,r)|=\lim_{i\to\infty}|(-x_i+C)\cap\clb(0,r)|=\lim_{i\to\infty}|C\cap\clb(x_i,r)|\ge b(r)>0.
\]
This implies that $K$ has interior points, i.e., that it is a convex body.

Assume now (iii) holds. Let us prove (i) reasoning by contradiction. Take $r>0$. In case $b(r)=0$, we take a sequence $\{x_i\}_{i\in\nn}$ so that $\lim_{i\to\infty}|\clb_C(x_i,r)|=0$. This sequence is divergent since otherwise we could extract a subsequence converging to some $x\in C$ with $|\clb_C(x_i,r)|$ subconverging to $|\clb_C(x,r)|>0$. Consider the asymptotic cylinder $K\in\mathcal{K}(C)$ obtained as the limit of a subsequence $\{-x_{i_j}+C\}_{j\in\nn}$. Since
\[
|K\cap\clb(0,r)|=\lim_{j\to\infty}|C\cap\clb(x_{i_j},r)|=0,
\]
the cylinder $K$ would have not interior points, contradicting assumption (ii). This completes the proof of the equivalences (i)$\Leftrightarrow$(ii)$\Leftrightarrow$(iii).

The fact that (iv) implies (v) is obvious. Now, we show that (v) implies (iii). To this aim, we argue by contradiction, i.e., we assume the existence of an asymptotic cylinder $K$ of $C$ with empty interior. Therefore, there exists a sequence $x_{j}\in C$ going off to infinity, such that $C_{j} = -x_{j}+C$ converges locally Hausdorff to $K$, as $j\to\infty$. 
Now, for $\eps>0$ small enough, we construct a set $E_{\eps}\subset C$ such that $|E_{\eps}|=v_{0}$ but $P_{C}(E_{\eps})\le \eps$, thus implying $I_{C}(v_{0})=0$, a contradiction. %
To this aim we fix $z_{0}\in C$ and define $r_{\eps} = \frac{(n+1)v_{0}}{\eps}$, then we assume $\eps$ small enough, so that $|B_{C}(z_{0},r_{\eps})| > v_{0}$. Since $\lim_{j\to\infty}|B_C(x_j,r_\eps)|=0$, by continuity of the volume of intrinsic balls we can choose $z_{\eps}\in C$ such that $E_{\eps} = B_{C}(z_{\eps},r_{\eps})$ satisfies $|E_{\eps}| = v_{0}$. 
By comparison with the cone ${\mathcal C}_{\eps}$ over $\partial E_{\eps}\cap \intt(C)$ with vertex $z_{\eps}$, taking into account $P_C(E_{\eps})=P_{\mathcal{C}_\eps}(B_{\mathcal{C}_\eps}(z_\eps,r_\eps))$ and $|E_{\eps}|\ge |B_{\mathcal{C}_\eps}(z_\eps,r_\eps)|$ we get
\begin{equation*}
v_{0} = |E_{\eps}| \ge |B_{\mathcal{C}_\eps}(z_\eps,r_\eps)| = \frac{r_{\eps}}{n+1} P_{\mathcal{C}_\eps}(B_{\mathcal{C}_\eps}(z_\eps,r_\eps)) = \frac{r_{\eps}}{n+1} P_{C}(E_{\eps}),
\end{equation*}
whence
\[
P_{C}(E_{\eps}) \le \frac{n+1}{r_{\eps}} v_{0} = \eps.
\]

This shows that $I_{C}(v_{0}) \le \eps$ for all $\eps>0$, thus $I_{C}(v_{0}) = 0$, i.e., a contradiction with (ii).

Let us finally prove that (ii) implies (iv). Fix some volume $v>0$ and consider a set $E\subset C$ of finite relative perimeter and volume $|E|=v$. We shall show that there exists a constant $\Lambda(C,v)>0$, only depending on the geometry of $C$ and $v$, such that $P_C(E)\ge\Lambda(C,v)$. This would imply $I_C(v)\ge\Lambda(C,v)>0$, as desired. Recall first that (ii) implies the existence of a positive radius $r_0>0$ satisfying \eqref{eq:mainhyp}. An application of Fubini's Theorem, \cite[Lemme~6.2]{gallot} yields
\begin{equation}
\label{eq:fubini}
\int_C |E\cap B_C(y,r_0)|\,dH^{n+1}(y)=\int_E |B_C(x,r_0)|\,dH^{n+1}(x).
\end{equation}
Since $E$ has finite volume and $|B_C(x,r_0)|\le\ell_2 r_0^{n+1}$ by \eqref{eq:isnqgdbl1a}, the function $f(y):=|E\cap B_C(y,r_0)|$ is in $L^1(C)$. Hence, for any $\eps>0$ the set $f^{-1}([\eps,+\infty))=\{y\in C: |E\cap B_C(y,r_0)|\ge \eps\}$ has finite volume and we get
\begin{equation}
\label{eq:claim1}
\inf_{x\in C} |E\cap B_{C}(x,r_{0})| = 0\,.
\end{equation}

Let us assume first that there exists $x_{0}\in C$ such that
\[
\frac{|E\cap B_C(x,r_0)|}{|B_C(x_{0},r_0)|}\ge \frac{1}{2}.
\]
By \eqref{eq:claim1}, \eqref{eq:isnqgdbl1a} and a continuity argument we get a point $z_{0}\in C$ so that
\[
\frac{|E\cap B_C(z_{0},r_0)|}{|B_C(z_{0},r_0)|}=\frac{1}{3}.
\]
By Lemma \ref{lem:inrad} below, we obtain
\begin{equation}
\label{eq:1stlowerbound}
\begin{split}
P_C(E)\ge P(E,B_C(z_{0},r_0))&\ge M\,\bigg(\frac{|B_C(z_{0},r_0)|}{3}\bigg)^{\frac{n}{n+1}}
\\
&\ge M\bigg(\frac{\ell_1r_0^{n+1}}{3}\bigg)^{\frac{n}{n+1}}>0.
\end{split}
\end{equation}
Therefore the perimeter of $P_C(E)$ is bounded from below by a constant only depending on the geometry of $C$. 
Now assume that
\[
\frac{|E\cap B_C(x,r_0)|}{|B_C(x,r_0)|}< \frac{1}{2}
\]
holds for all $x\in C$. Let $\{B_C(x_i,r_0/2)\}_{i\in I}$ be a maximal family of disjoint intrinsic open balls centered at points of $C$. Then the family $\{B_C(x_i,r_0)\}_{i\in I}$ is an open covering of $C$. The overlapping of sets in this family can be estimated in the following way. For $x\in C$, define
\[
A(x)=\{i\in I: x\in B_C(x_i,r_0)\}.
\]
When $i\in A(x)$, it is immediate to check that $B_C(x_{i},r_0/2)\subset B_C(x,2r_0)$ (if $y\in B_C(x_i,r_0/2)$, then $d(y,x)\le d(y,x_i)+d(x_i,x)<2r_0$). Hence $\{B_C(x_i,r_0/2)\}_{i\in A(x)}$ is a disjoint family of balls contained in $B_C(x,2r_0)$. Coupling the estimate \eqref{eq:isnqgdbl1a} in Lemma \ref{lem:inrad} with \eqref{eq:doubling}, we get
\[
\# A(x)\, \ell_1\bigg(\frac{r_0}{2}\bigg)^{n+1}\le \sum_{i\in A(x)} |B_C(x_i,r_0/2)|\le |B_C(x,2r_0)|\le 2^{n+1}\ell_2\, r_0^{n+1},
\]
which implies the uniform bound $\# A(x)\le K(C,n):=4^{n+1}\ell_2\ell_1^{-1}$. Finally, the overlapping estimate and the relative isoperimetric inequality in $B_C(x_i,r_0)$ (see Theorem~4.11 in \cite{MR3335407}) imply
\begin{equation}
\label{eq:2ndlowerbound}
\begin{split}
K(C,n)\,P_C(E)&\ge \sum_{i\in I} P_C(E,B_C(x_i,r_0))
\\
&\ge M\sum_{i\in I} |E\cap B_C(x_i,r_0)|^{\frac{n}{n+1}}\ge M\,|E|^{\frac{n}{n+1}}.
\end{split}
\end{equation}

From \eqref{eq:1stlowerbound} and \eqref{eq:2ndlowerbound} we obtain
\begin{equation}
\label{stimaisopsmall}
P_C(E)\ge \Lambda(C,v):=\min\bigg\{\bigg(\frac{\ell_1r_0^{n+1}}{4}\bigg)^{\frac{n}{n+1}},\frac{Mv^{\frac{n}{n+1}}}{K(C,n)}\bigg\}>0.
\end{equation}
This completes the proof of (ii)$\Rightarrow$(iv) and thus we conclude the proof of the proposition.

Finally, assume (i) holds and consider a point $x\in\ptl C$. Inequalities
\[
0<b(r)\le |\clb_C(x,r)|\le |\clb_{C_x}(x,r)|=\int_0^r\alpha(C_x)\,s^n\,ds=\frac{\alpha(C_x)\,r^{n+1}}{n+1}
\]
imply that $\alpha(C_x)$ is estimated uniformly from below by the positive constant $\alpha_0=(n+1)\,b(r)\,r^{-(n+1)}$.
\end{proof}

\begin{remark}
By a slight variant of Example~\ref{ex:ex}, one sees that only assuming condition $\inf_{p\in\ptl C} \alpha(C_p)>0$ is not enough to ensure $I_C>0$: indeed it is sufficient to intersect the unbounded convex body $C$ constructed in Example~\ref{ex:ex} with the one-parameter family of half-spaces $A_{x}$ having the point $m_x:=p_{x} - (1,0,0)$ on their boundary and inner normal vector $N_{x} = (2-2x,0,1)$, for $x>1$. The resulting set $\hat C$ satisfies $\inf_{p\in\ptl \hat C}\alpha(\hat C_{p})>0$ but still has a null isoperimetric profile. This is easy to check  using (iii) in Proposition~\ref{prop:maincond} since the asymptotic cylinder of $\hat C$ obtained as a limit of a convergent subsequence of $-m_x+\hat C$ (when $x$ diverges) has empty interior as it is contained in a sequence of wedges with solid angles going to zero.
\end{remark}

\section{Density estimates and a concentration lemma}
\label{sec:density}

Next we show that whenever $C$ is a convex body of uniform geometry, we can obtain uniform lower (and upper) density estimates for the volume of $B_{C}(x,r_0)$, as well as uniform relative isoperimetric inequalities on $B_{C}(x,r_0)$.

\begin{lemma}
\label{lem:inrad}
Let $C$ be a convex body of uniform geometry satisfying \eqref{eq:mainhyp}. Then
\begin{itemize}
\item[(i)] $\displaystyle\inf_{x\in C} \inr(\clb_C(x,r_0))\ge \frac{b(r_{0})}{(n+1)\omega_{n+1}r_0^n}$;

\item[(ii)] there exists $M>0$ only depending on $n$, $r_0/b(r_0)$, such that for all $x\in C$, $0<r\le r_0$, and $0<v<|B_C(x,r)|$, one finds
\begin{equation}
\label{eq:isnqgdbl1}
I_{\clb_C(x,r)}(v)\ge M\, \min \{v,|\clb_C(x,r)|-v\}^{n/(n+1)}\,;
\end{equation}

\item[(iii)] there exist $\ell_1>0$ only depending on $n,r_{0}, b(r_{0})$ and $\ell_2>0$ only depending on $n$, such that for all $x\in C$ and $0<r\le r_0$ one has
\begin{equation}
\label{eq:isnqgdbl1a}
\ell_1 r^{n+1} \le \vol{\clb_C(x,r)} \le \ell_2 r^{n+1}\,.
\end{equation}
\end{itemize}
\end{lemma}

\begin{proof}

To prove (i) we let $D_{t}$ be the set of points in $D:= \clb_{C}(x,r_{0})$ whose distance from $\partial D$ is at least $t$. Then $D_{t}$ is convex and nonempty for any $t\in [0,\inr(D)]$, while it is empty as soon as $t>\inr(D)$. The coarea formula applied to the distance function from $\partial D$ yields
\begin{equation}
\label{eq:inrestimate}
b(r_{0})\le |D| = \int_{0}^{\inr(D)}P(D_{t})\, dt \leq P(D) \inr(D)\le P(\clb(0,r_0))\inr(D),
\end{equation}
since $D\subset \clb(0,r_0)$ implies $P(D)\le P(\clb(0,r_0))=(n+1)\omega_{n+1}r_0^n$. Therefore we find
\[
\inr(\clb_C(x,r_0)) \ge \frac{b(r_{0})}{(n+1)\omega_{n+1}r_0^n}\,,
\]
thus proving (i). 

In order to prove (ii) we shall use Theorem~4.11 in \cite{MR3335407}: if $K\subset\rr^{n+1}$ is a bounded convex body, $x,y\in K$, and $0<\rho_1<\rho_2$ satisfy $\clb(y,\rho_1)\subset K\subset\clb(x,\rho_2)$, then there exists a constant $M>0$ given as a explicit function of $n$ and $\rho_2/\rho_1$ such that
\[
I_K(v)\ge M\min\{v,|K|-v\}^{n/(n+1)},
\]
for all $0\le v\le |K|$. The proof of this result makes use of the bilipschitz map $f:K \to\clb(y,\rho_2)$ defined in \cite[(3.9)]{MR3335407} (with $r=\rho_1/2$) and the estimates on the lipschitz constants in Corollary~3.9 of \cite{MR3335407}, that depend on $\rho_2/\rho_1$. The dependence of the constant $M$ on the dimension $(n+1)$ of the ambient Euclidean space is due to the dependence of $M$ on the optimal constant $M_0$ in the isoperimetric inequality
\[
I_{\clb(y,\rho_2)}(v)\ge M_0\min\{v,|\clb(y,\rho_2)|-v\}^{n/(n+1)}, \quad v\in (0,|\clb(y,\rho_2)|).
\]
The constant $M_0$ is invariant by translations and dilations in $\rr^{n+1}$ and hence valid for any closed ball.

Using this result, we choose $\rho_1=b(r_0)$, $K=\clb_C(x,r_0)$, $\rho_2=r_0$ so that the constant $M>0$ in the inequality
\[
I_{\clb_C(x,r_0)}(v)\ge M \min \{v,|\clb_C(x,r_0)|-v\}^{n/(n+1)},\quad v\in (0,|\clb_C(x,r_0)|,
\]
is given explicitly as a function of $n$ and $r_0/b(r_0)$ for any $x\in C$. Fix now some $0<r\le r_0$ and some $x\in C$, and take $y\in \clb_C(x,r_0)$ such that $\clb(y,b(r_0))\subset \clb_C(x,r_0)$. Take $\la\in (0,1]$ such that $r=\la r_0$. Denoting by $h_{x,\la}$ the homothety of center $x$ and ratio $\la$ we have 
\[
\clb(h_{x,\la}(y),\la b(r_0))=h_{x,\la}(\clb(y,b(r_0)))\subset h_{x,\la}(\clb_C(x,r_0))\subset\clb_C(x,r),
\]
the latter inclusion following from the concavity of $C$. We conclude that a relative isoperimetric inequality holds in $\clb_C(x,r)$ with a constant given explicitly as a function of $n$ and of $r/\la b(r_0)=\la r_0/\la b(r_0)=r_0/b(r_0)$. This means that $M>0$ can be taken uniformly for any $x\in C$ and $r\in (0,r_0]$.

We now prove (iii). Since $|B_C(x,r)|\le |B(x,r)|$, taking $\ell_2=\omega_{n+1}$ immediately gives the upper bound in \eqref{eq:isnqgdbl1a}. Then setting $\la = r/r_{0}$ and $\de = \inr(B_{C}(x,r_{0}))$ we have
\begin{align*}
|B(x,r)\cap C| &= |B(x, \la r_0)\cap C| = |B(0, \la r_0)\cap (-x+C)|\\ 
&\ge |B(0,\la r_0) \cap \la(-x+C)| = \la^{n+1}|B_{C}(x,r_0)|\\ 
&\ge \omega_{n+1}(\de\la)^{n+1} = \ell_{1} r^{n+1}\,,
\end{align*}
where $\ell_1=\omega_{n+1}(\de/r_0)^{n+1}$ only depends on $n,r_{0},b(r_{0})$. This completes the proofs of (iii) and of the lemma.
\end{proof}

An immediate consequence of Lemma~\ref{lem:inrad} and the argument leading to equation \eqref{stimaisopsmall} is the following corollary:
\begin{corollary}
\label{cor:isopinesm}
Let $C\subset \rr^{n+1}$ be a convex body of uniform geometry. Then there exist $v_0$, $c_0>0$, depending only on $n$, on the Ahlfors constant $\ell_1$ in \eqref{eq:isnqgdbl1a}, and on the Poincar\'e inequality \eqref{eq:isnqgdbl1}, such that
\begin{equation}
\label{eq:exicyl1}
I_C(v)\ge c_0\, v^{n/(n+1)}\quad\text{for any }\quad v\le v_0.
\end{equation}
\end{corollary}

\begin{remark}
An alternative proof of \eqref{eq:inrestimate} could be given using Steinhagen's Theorem \cite{48.0837.03}. The width $w$ of $\clb_C(x,r_0)$ satisfies
\[
w\le A_n \inr(\clb_C(x,r_0)),
\]
where $A_n>0$ is a constant only depending on $n$. On the other hand,
\[
|B_C(x,r_0)|\le \omega_n w r_0^n,
\]
where $\omega_n>0$ is the $\hh^n$-measure of the $n$-dimensional unit disc. Hence we obtain from \eqref{eq:mainhyp}
\[
\inr(\clb_C(x,r_0))\ge  (A_nB_n r_0^n)^{-1} b(r_{0})>0.
\]
\end{remark}

\begin{remark}\label{rmk:convinradius}
Assume that $\{C_j\}_{j\in\nn}$ converge in (global) Hausdorff distance to an unbounded convex body $C$ as $j\to\infty$. If \eqref{eq:mainhyp} holds for $C$,  then, for $j\in\nn$ large enough, one can show that $C_{j}$ satisfies the thesis of Lemma \ref{lem:inrad} with constants $M,\ell_{1},\ell_{2}$ that do not depend on $j$. Viceversa, if $\{C_j\}_{j\in\nn}$ is a sequence of convex bodies satisfying \eqref{eq:mainhyp} uniformly on $j\in \nn$, and converging locally in Hausdorff distance to a convex body $C$, then $C$ necessarily satisfies the thesis of Lemma \ref{lem:inrad} with constants $M,\ell_{1},\ell_{2}$ that only depend on $n,r_{0}$ and $b(r_{0})$. This follows from the $1$-Lipschitz continuity of the inradius as a function defined on compact convex bodies endowed with the Hausdorff distance, as shown in Lemma \ref{lem:inradius} below.
\end{remark}

Now we proceed to prove that the inradius of a bounded convex body is a $1$-Lipschitz function with respect to the Hausdorff distance. 

Let $C\subset\rr^{n+1}$ be a bounded convex body. For any $t>0$, we define the \emph{inner parallel} at distance $t$ by
\[
C_{-t}=\{p\in C: d(p,\ptl C)\ge t\}.
\]
It is well-known that $C_{-t}$ is a convex set whenever it is non-empty. 

\begin{lemma}
\label{lem:inradius}
Let $C$, $K\subset\rr^{n+1}$ be bounded convex bodies. Then
\begin{equation}
\label{eq:inr}
|\inr(C)-\inr(K)|\le \hd(C,K).
\end{equation}
This implies that $\inr$ is a $1$-Lipschitz function in the space of bounded convex bodies endowed with the Hausdorff metric.
\end{lemma}

\begin{proof}
We split the proof in two steps.

\textit{Step one.} We show that $(C_t)_{-t}=C$ for any $t>0$, whenever $C\subset\rr^{n+1}$ is a bounded convex body. Let us start proving that $C\subset (C_t)_{-t}$. Arguing by contradiction, we assume that $d(p,\ptl C_t)<t$ for some $p\in C$. Then there exists $q\in\ptl C_t$ so that $|p-q|=d(p,\ptl C_t)$. Choose $r>0$ small enough so that $|p-q|+r<t$. If $z\in B(q,r)$, then
\[
d(z,C)\le |z-p|\le |z-q|+|q-p|<|p-q|+r<t.
\]
This implies that $B(q,r)\subset C_t$, a contradiction to the fact that $q\in\ptl C_t$. So we have $C\subset (C_t)_{-t}$.

To prove the reverse inequality, we take $p\in (C_t)_{-t}$. If $p\not\in C$, then $d(p,C)=d>0$. Let $q$ be the metric projection of $p$ to $C$. Being $C$ convex, it turns out that $q$ is also the metric projection of every point in the half-line $\{q+\la\,(p-q):\la\ge 0\}$. Let $z$ be the point in this half-line at distance $t$ from $C$. Then $z\in\ptl C_t$, and $d(p,\ptl C_t)\le |p-z|=t-|p-q|<t$, a contradiction since $p$ was taken in $(C_t)_{-t}$. So we get $(C_t)_{-t}\subset C$.

\textit{Step two.} Let $\eps=\hd(C,K)$ and observe that $K\subset C_\eps$. If both $\inr(C)$, $\inr(K)$ are smaller than or equal to $\eps$, then inequality \eqref{eq:inr} is trivial. So let us assume that $\inr(K)$ is the largest inradius and that $r=\inr(K)>\eps$. Take $B(x,r)\subset K\subset C_\eps$. By Step one we find
\[
B(x,r-\eps)=B(x,r)_{-\eps}\subset (C_\eps)_{-\eps}=C.
\]
So we have $\inr(K)\ge \inr(C)\ge \inr(K)-\eps$. This implies \eqref{eq:inr}.
\end{proof}

\begin{proposition}
\label{prp:isopbound}
Let $C\subset \rr^{n+1}$ be a convex body of uniform geometry. Then any isoperimetric region in $C$ is bounded.
\end{proposition}
\begin{proof}
Let $v>0$ and $E\subset C$ be such that $|E|=v$ and $P_{C}(E) = I_{C}(v)$. Arguing as in \cite[IV.1.5]{maggi} we can find $\eps>0$ and a one-parameter family $\{\phi_{t}\}_{t\in (-\eps,\eps)}$ of diffeomorphisms, such that setting $E_{t}=\phi_{t}(E)$ one has $E_{t}\subset C$, $|E_{t}| = v+t$, and $P_{C}(E_{t}) \le P_{C}(E) + c|t|$ for all $t\in (-\eps,\eps)$ and for some constant $c>0$ depending on $E$. Let $x_{0}\in C$ be fixed, then we set $m(r) = |E\setminus B_{C}(x_{0},r)|$. If $r$ is large enough, we can entail at the same time that $m(r)<\min(\eps,v_{0})$ (where $v_{0}$ is as in Corollary \ref{cor:isopinesm}) and that the support of $\phi_{t}$ is compactly contained in $B_{C}(x_{0},r)$. Therefore we can define $F_{r} = \phi_{m(r)}(E) \setminus B_{C}(x_{0},r)$ and get for almost all $r>0$
\begin{align*}
P_{C}(F_{r}) &= P_{C}(E_{m(r)}) - 2m'(r) -P_{C}(E\setminus B_{C}(x_{0},r)) 
\\
&\le 
P_{C}(E) + c\, m(r) -2m'(r) - c_{0} m(r)^{n/(n+1)}\,,
\end{align*}
where $c_{0}$ is as in Corollary \ref{cor:isopinesm}. Since of course $|F_{r}| = |E|$, by minimality, and up to choosing $r$ large enough, we can find a constant $c_{1}>0$ such that
\[
c_{1}m(r)^{n/(n+1)} \le c_{0}m(r)^{n/(n+1)} - c\, m(r) \le -2m'(r)\,.
\]
Let us fix $r$ and take any $R>r$ such that $m(R)>0$. Then one rewrites the above inequality as 
\[
c_{2} \le -\frac{m'(r)}{m(r)^{n/(n+1)}}
\]
for some $c_{2}>0$. By integrating this last inequality between $r$ and $R$ one gets
\[
c_{3}(R-r) \le m(r)^{1/(n+1)} - m(R)^{1/(n+1)}
\]
for some $c_{3}>0$, whence the boundedness of $R$ follows. In conclusion, there exists a largest $R>0$ such that $E\subset B_{C}(x_{0},R)$, as wanted.
\end{proof}


We conclude this section with a concentration lemma, that is well-known in $\rr^{n+1}$ as well as in Carnot groups (see \cite{le-ri}). Since the ambient domain here is a convex body $C$ of uniform geometry, it seems worth giving a full proof of the result (notice the use of Tonelli's theorem instead of the covering argument used in the proof of \cite[Lemma~3.1]{le-ri}).

\begin{lemma}[Concentration]
\label{lem:lrlemme31}
Let $C\subset\rr^{n+1}$ be a convex body of uniform geometry, and $E\subset C$ a set with finite relative perimeter. Choose $0<r\le 1$ and $m\in (0,\tfrac{1}{2}]$, and assume
\begin{equation}
\label{eq:m}
|E\cap B_C(x,r)|\le m\,|B_C(x,r)|, \qquad \forall x\in C.
\end{equation}
Then there exists $\Lambda>0$, only depending on $C$, such that we have
\begin{equation}
\label{eq:lr}
\Lambda\,|E|\le m^{1/(n+1)}r\,P_C(E).
\end{equation}
\end{lemma}

\begin{remark}
The constant $\Lambda$ is defined by
\begin{equation}
\label{eq:defLambda}
\Lambda:=\frac{c_1b(1)}{\omega_{n+1}^{(n+2)/(n+1)}},
\end{equation}
where $\omega_{n+1}=|B(0,1)|$, $b(1)=\inf_{x\in C} |\clb_C(x,1)|$, and $c_1$ is the Poincar\'e constant for the relative isoperimetric inequality in balls of radius $0<r\le 1$.
\end{remark}

\begin{proof}[Proof of Lemma~\ref{lem:lrlemme31}]
Since $C$ is of uniform geometry, $|B_C(x,1)|\ge b(1)>0$ for all $x\in C$. By Lemma~\ref{eq:doubling}, inequality $|B_C(x,r)|\ge b(1)\,r^{n+1}$ holds for any $r\in (0,1]$. By \eqref{eq:isnqgdbl1}, the relative isoperimetric inequality
\[
P_C(E,B_C(x,r))\ge c_1\,\min\big\{|E\cap B_C(x,r)|,|B_C(x,r)\setminus E|\big\}^\frac{n}{n+1}
\]
holds for any $r\in (0,1]$ with a uniform Poincar\'e constant $c_1$. Tonelli's Theorem implies
\begin{align*}
\int_{C} P(E,B_{C}(x,r))\, dH^{n+1}(x) &= \int_{C}\bigg\{\int_{\intt(C)} \chi_{B(x,r)}(y)\, d|D\chi_{E}|(y)\bigg\}\, dH^{n+1}(x)\\ 
&= \int_{C}\bigg\{\int_{\intt(C)} \chi_{B(y,r)}(x)\, d|D\chi_{E}|(y)\,\bigg\} dH^{n+1}(x)\\ 
&= \int_{\intt(C)}\bigg\{ \int_{C} \chi_{B(y,r)}(x)\, dH^{n+1}(x)\bigg\}\, d|D\chi_{E}|(y)\\ 
&= \int_{\intt(C)} |B_{C}(y,r)|\, d|D\chi_{E}|(y)
\\ &\le \omega_{n+1}r^{n+1} P_{C}(E).
\end{align*}
And so we get
\begin{align*}
\omega_{n+1}r^{n+1}P_C(E)&\ge\int_C P(E,B_C(x,r))\,dH^{n+1}(x)
\\
&\ge c_1\,\int_C\frac{|E\cap B_C(x,r)|}{|E\cap B_C(x,r)|^{\frac{1}{n+1}}}\,dH^{n+1}(x)
\\
&\ge c_1\int_C\bigg(\frac{1}{m\,|B_C(x,r)|}\bigg)^\frac{1}{n+1}|E\cap B_C(x,r)|\,dH^{n+1}(x)
\\
&\ge \frac{m^{-1/(n+1)}c_1}{\omega_{n+1}^{1/(n+1)}r}\int_C |E\cap B_C(x,r)|\,dH^{n+1}(x)
\\
&=\frac{m^{-1/(n+1)}c_1}{\omega_{n+1}^{1/(n+1)}r}\int_E |B_C(y,r)|\,dH^{n+1}(y)
\\
&\ge \frac{m^{-1/(n+1)}c_1b(1)}{\omega_{n+1}^{1/(n+1)}r}\,r^{n+1}|E|,
\end{align*}
where we have used \eqref{eq:m} to obtain the inequality relating the second and third lines, and equation \eqref{eq:fubini} to get the equality in the fifth line. The above chain of inequalities, together with the definition \eqref{eq:defLambda} of $\Lambda$, imply \eqref{eq:lr}.
\end{proof}

The following two corollaries will be used in the sequel. The first one will play an important role in the proof of Theorem \ref{thm:contprof}. The second one will be used in Chapter~\ref{sec:min}.

\begin{corollary}
\label{cor:lemme1}
Let $C\subset\rr^{n+1}$ be a convex body of uniform geometry, and $E\subset C$ a set with positive relative perimeter. Then there exists $\Lambda>0$, only depending on $C$, such that, for every $r>0$ satisfying
\begin{equation*}
r<\min\bigg\{2^{1/(n+1)}\Lambda\,\frac{|E|}{P_C(E)},1\bigg\},
\end{equation*}
there exists a point $x\in C$, only depending on $r$ and $E$, with
\[
|E\cap B_C(x,r)|> \frac{|B_C(x,r)|}{2}.
\]
\end{corollary}

\begin{proof}
We simply argue by contradiction using Lemma~\ref{lem:lrlemme31} and \eqref{eq:lr} for $m=1/2$.
\end{proof}

\begin{corollary}
\label{cor:lemme2-b}
Let $C\subset\rr^{n+1}$ be a convex body of uniform geometry, $v_0>0$, and $\{E_i\}_{i\in\nn}\subset C$ a sequence such that, denoting by $H$ a generic half-space,
\[
\vol{E_i}\le v_0 \text{ for all }i\in \nn\,,\quad \lim_{i\to\infty}\vol{E_i}=v \in (0,v_{0}]\,,
\quad\liminf_{i\to\infty} P_C(E_i)\le I_H(v)\,. 
\]
Take $m_0\in (0,\tfrac{1}{2}]$ such that
\[
m_0<\min\bigg\{\frac{1}{2v_0}, \frac{\Lambda^{n+1}}{I_H(1)^{n+1}}\bigg\}\,.
\]
Then there exists a sequence $\{x_i\}_{i\in\nn}\subset C$ such that
\[
|E_i\cap B_C(x_i,1)|\ge m_0v\qquad\text{for $i$ large enough.}
\]
\end{corollary}
\begin{proof}
By contradiction, and up to subsequences, we may assume  $|E_i\cap B_C(x,1)|<m_0|E_i|$ for all $x\in C$. We apply Lemma~\ref{lem:lrlemme31}, and in particular \eqref{eq:lr} with $m=m_0 v\le m_0v_0\le\tfrac{1}{2}$, in order to get
\[
\Lambda^{n+1}\,|E_i|^n\le m_0P_C(E_i)^{n+1},
\]
From this inequality and our hypotheses, by taking limits we obtain
\[
\Lambda^{n+1}\,v^n\le m_0 I_H(v)^{n+1}
\]
and thus
\[
m_0\ge \Lambda^{n+1}\,\frac{v^n}{I_H(v)^{n+1}}=\frac{\Lambda^{n+1}}{I_H(1)^{n+1}},
\]
that is, a contradiction to the choice of $m_0$.
\end{proof}

\begin{remark}
Corollary~\ref{cor:lemme2-b} holds in particular when $\seq{E}$ is a minimizing sequence for volume $v$, since $\liminf_{i\to\infty} P_C(E_i)\le  I_C(v)\le I_H(v)$.
\end{remark}

\section{Examples}

Let us give now some examples of unbounded convex bodies of uniform geometry.

\begin{example}[Cylindrically bounded convex bodies are of uniform geometry and their asymptotic cylinders are unique up to horizontal translations]
\label{ex:cylindrically}
Assume that $C\subset\rr^{n+1}$ is a cylindrically bounded convex body as defined in \cite{MR3441524}: the set $C$ is the epigraph of a convex function defined on the interior of a bounded convex body $K\subset \rr^n\equiv \rr^n\times\{0\}\subset \rr^{n+1}$, and the intersections of $C$ with the horizontal hyperplanes $\Pi_c:=\{x_{n+1}=c\}$, for $c\in\rr$, projected to $\Pi_0$ form an increasing (w.r.t. $c$) family converging in Hausdorff distance to $K$.

Let $\{x_i\}_{i\in\nn}$ be a diverging sequence in $C$ such that $\{-x_i+C\}_{i\in\nn}$ converges locally in Hausdorff distance to an asymptotic cylinder $C_\infty$. Write $x_i=(z_i,t_i)\in\rr^n\times \rr$. The coordinates $t_i$ are unbounded and the vectors $z_i$ are uniformly bounded. This implies that the sequence $\{x_i/|x_i|\}_{i\in\nn}$ converges to the unit vector $v=(0,1)\in\rr^n\times\rr$. By construction, the half-lines $\{x+\la v:\la\ge 0\}$ are contained in $C$ for all $x\in C$. It is easy to check that
\[
(-x_i+C)\cap\Pi_0=-z_i+\pi(C\cap\Pi_{t_i}),
\]
where $\pi$ is the orthogonal projection onto the hyperplane $\Pi_0$. Since $\pi(C\cap\Pi_{t_i})$ converges to $K$ in Hausdorff distance and $\{z_i\}_{i\in\nn}$ is a bounded sequence, we immediately conclude that $(-x_i+C)\cap\Pi_0$ subconverges to a horizontal translation $K'$ of $K$. Obviously $K'\subset C_\infty$. Since vertical lines passing through a point in $C_\infty$ are contained in $C_\infty$ we immediately obtain that $K'\times\rr\subset C_\infty$.

Let us check that $C_\infty=K'\times\rr$. If $x\in C_\infty$ then $\pi(x)\in C_\infty$. By Kuratowski criterion, there exists a sequence $\{c_i\}_{i\in\nn}\subset C$ such that $\{-x_i+c_i\}_{i\in\nn}$ converges to $\pi(x)$. If we write $c_i$ as $(c_i',s_i)\in\rr^n\times\rr$, then $-z_i+c_i'$ converges to $\pi(x)$ and $-t_i+s_i$ converges to $0$. For each $i$, choose $j(i)>j(i-1)$ such that $t_{j(i)}>s_i$. Then $d_i:=(c_i',t_{j(i)})\in C$ and $-x_{j(i)}+d_i=(-z_{j(i)}+c_i',0)$ converges to $\pi(x)$. This implies that $\pi(x)\in K'\times \rr$ and so $x\in K'\times\rr$. Hence $C_\infty\subset K'\times\rr$.

The above arguments imply that any asymptotic cylinder of $C$ is a horizontal translation of $K\times\rr$. Since $K\times\rr$ has non-empty interior, Proposition~\ref{prop:maincond} implies that $C$ is of uniform geometry.

\end{example}

\begin{example}[Unbounded convex bodies of revolution are of uniform geometry and the asymptotic cylinders of non-cylindrically bounded convex bodies of revolution are either half-spaces or $\rr^{n+1}$]
\label{ex:revolution}
Let $\psi:[0,\infty)\to [0,\infty)$ be a continuous concave function satisfying $\psi(0)=0$ and $\psi(x)>0$ for all $x>0$. Consider the convex body of revolution
\[
C=C_\psi:=\{(z,t)\in\rr^n\times [0,\infty)\subset\rr^{n+1}: |z|\le \psi(t)\}.
\]
We shall assume that $\psi$ is unbounded since otherwise $C_\psi$ would be a cylindrically bounded convex body. We remark that we are not assuming any smoothness condition on $\psi$.

Take a diverging sequence $\{x_i\}_{i\in\nn}\subset C$ and assume that $\{-x_i+C\}_{i\in\nn}$ converges locally in Hausdorff distance to some asymptotic cylinder $C_\infty$.

Write $x_i=(z_i,t_i)\in\rr^n\times [0,\infty)$. The sequence $\{t_i\}_{i\in\nn}$ cannot be bounded since otherwise inequality $|z_i|\le\psi(t_i)$ would imply that $z_i$ is also uniformly bounded (and $\{x_i\}_{i\in\nn}$ would be bounded). On the other hand, since the function $\psi(t)/t$ is non-increasing by the concavity of $\psi$, the sequence $|z_i|/t_i\le\psi(t_i)/t_i$ is uniformly bounded. Hence $z_i/t_i$ subconverges to some vector $c\in\rr^n$ and so
\[
\frac{x_i}{|x_i|}=\frac{(z_i/t_i,1)}{\sqrt{(|z_i|/t_i)^2+1}}
\]
subconverges to the vector $v:=(c,1)/\sqrt{c^2+1}$, whose last coordinate is different from $0$. We conclude that any straight line parallel to $v$ containing a point in $C_\infty$ is entirely contained in $C_\infty$.

The sets $(-x_i+C)\cap\Pi_0$ are closed disks $\overline{D}(w_i,\psi(t_i))\subset\rr^n$ of center $w_i=-z_i\in\rr^n$ and radius $\psi(t_i)$. We define $r_i:=\psi(t_i)-|w_i|=\psi(t_i)-|z_i|\ge 0$.

In case $\{r_i\}_{i\in\nn}$ is an unbounded sequence the inclusion $\overline{D}(0,r_i)\subset \overline{D}(w_i,\psi(t_i))$ holds and shows that any point in $\rr^n\times\{0\}$ belongs to $C_\infty$. Hence $C_\infty=\rr^{n+1}$.

If the sequence $\{r_i\}_{i\in\nn}$ is bounded, passing to a subsequence we may assume  $\lim_{i\to\infty} r_i=c\ge 0$ and $\lim_{i\to\infty} w_i/|w_i|=e$. Let us prove first that the set
\[
K:=\{(z,0):\escpr{z,e}\ge -c\}
\]
is contained in $C_\infty$. Pick some $(z,0)\in K$ and choose $r>0$ large enough so that $z\in\intt(\overline{B}(0,r))$. In case $\escpr{z,e}>-c$ we have
\begin{align*}
\lim_{i\to\infty}\frac{|z-w_i|^2-\psi(t_i)^2}{|w_i|}&=\lim_{i\to\infty}\bigg(\frac{|z|}{|w_i|}-2\escpr{z,\frac{w_i}{|w_i|}}+\frac{|w_i|^2-\psi(t_i)^2}{|w_i|}\bigg)
\\
&=-2\escpr{z,e}-2c<0,
\end{align*}
since $|w_i|$ is unbounded, $\lim_{i\to\infty}w_i/|w_i|=e$, and
\[
\lim_{i\to\infty}\frac{\psi(t_i)^2-|w_i|^2}{|w_i|}=\lim_{i\to\infty}\frac{\psi(t_i)+|w_i|}{|w_i|}\,r_i=2c.
\]
This implies that, if $\escpr{z,e}>-c$, there exists $i_0\in\nn$ so that $z\in\overline{D}(w_i,\psi(t_i))$ for all $i\ge i_0$. Hence $(z,0)\in C_\infty$.

In case $\escpr{z,e}=-c$, we consider a sequence $\{\eps_i\}_{i\in\nn}$ of positive real numbers decreasing to $0$. For each $j$, the point $z+\eps_je$ satisfies $\escpr{z+\eps_je,e}=-c+\eps_j>-c$. Hence we can choose $i(j)$ (increasing in $j$) such that $z+\eps_je\in \overline{D}(w_i,\psi(t_j))$ for $i\ge i(j)$. We construct a sequence $\{m_k\}_{k\in\nn}$ of points in $\rr^n\times\{0\}$ so that $m_k$ is chosen arbitrarily in $\overline{D}(w_k,\psi(t_k))\cap\intt(\overline{B}(0,r))$ for $1\le k< i(1)$, and $m_k:=z+\eps_j e$ when $k\in [i(j),i(j+1))$. The point $m_k$ lies in $\overline{D}(w_k,\psi(t_k))\cap\intt(\overline{B}(0,r))$ for all $k\in\nn$, and the sequence $\{m_k\}_{k\in\nn}$ converges to $z$. By the Kuratowski criterion, $(z,0)\in C_\infty$. Hence $K\subset C_\infty$ as claimed.

Since $K\subset C_\infty$, and lines parallel to $L(v)$ intersecting $C_\infty$ are contained in $C_\infty$, we conclude that the half-space $K+L(v)$ is contained in $C_\infty$. But then $C_\infty$ itself must be either $\rr^{n+1}$ or a half-space. This is easy to prove since, in case $C\neq\rr^{n+1}$, we have $C=\bigcap_{i\in I}H_i$, where $\{H_i\}_{i\in I}$ is the family of all supporting hyperplanes to $C$. For any $i$, we would have $H\subset C\subset H_i$, and this would imply that $H_i$ is a half-space parallel to $H$ containing $H$ and so it would be $\bigcap_{i\in I} H_i$.
\end{example}

\begin{example}[An unbounded convex body of uniform geometry with $C^\infty$ boundary, all of whose asymptotic cylinders have non-smooth boundary]
\label{ex:noregbdy}
Let $g:\rr\to\rr$ be a $C^\infty$ function such that $g>0$, $g'<0$, $g''>0$, $g(0)=1/2$ and $\lim_{t\to +\infty} g(t)=0$. We consider the $C^\infty$ function in $\rr^3$ defined by
\[
f(x,y,z):=\sqrt{x^2+g(z)^2}+\sqrt{y^2+g(z)^2}.
\]

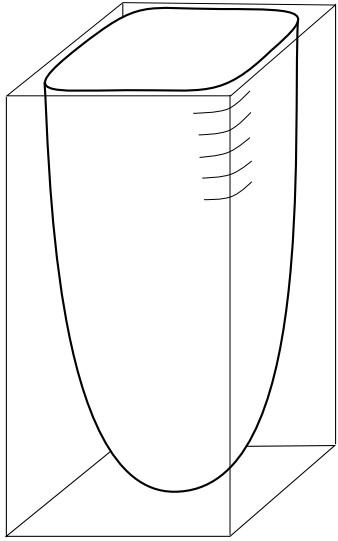
\begin{figure}[h]
\begin{tikzpicture}[y=0.80pt, x=0.8pt,yscale=-1, inner sep=0pt, outer sep=0pt]
  \path[draw=black,line join=miter,line cap=butt,miter limit=4.00,even odd
    rule,line width=0.400pt] (226.6939,70.2820) -- (226.6939,64.2862);
  \path[draw=black,line join=miter,line cap=butt,miter limit=4.00,even odd
    rule,line width=0.400pt] (221.2909,275.8951) -- (171.5698,316.3757) --
    (277.3980,316.3757) -- (326.8715,273.4169) -- (285.1463,273.7866);
  \path[draw=black,line join=miter,line cap=butt,miter limit=4.00,even odd
    rule,line width=0.400pt] (327.2966,272.7967) -- (327.2966,64.4195);
  \path[draw=black,line join=miter,line cap=butt,miter limit=4.00,even odd
    rule,line width=0.240pt] (277.7206,107.9291) --
    (327.1941,64.9703)(327.1941,64.9703) -- (226.7339,63.8341) --
    (171.8924,107.9291) -- (277.7206,107.9291);
  \path[draw=black,line join=miter,line cap=butt,miter limit=4.00,even odd
    rule,line width=0.400pt] (171.6389,316.8007) -- (171.6389,108.4234);
  \path[draw=black,line join=miter,line cap=butt,miter limit=4.00,even odd
    rule,line width=0.240pt] (277.3845,315.8329) -- (277.3845,107.4557);
  \path[draw=black,line join=miter,line cap=butt,even odd rule,line width=0.800pt]
    (273.2258,66.8783) .. controls (314.6168,66.8132) and (318.0818,68.0018) ..
    (296.7742,86.5557) .. controls (272.9790,109.4593) and (266.0945,105.0317) ..
    (228.7097,105.2654) .. controls (191.3249,105.4991) and (175.8003,108.6539) ..
    (205.4839,85.2654) .. controls (235.1675,61.8770) and (235.8818,66.6641) ..
    (273.2258,66.8783) -- cycle;
  \path[draw=black,line join=miter,line cap=butt,even odd rule,line width=0.800pt]
    (189.6774,101.3944) .. controls (191.9162,132.6462) and (192.5181,295.2847) ..
    (250.9533,295.2847) .. controls (318.3438,295.2847) and (306.8963,129.0449) ..
    (309.4678,72.3640);
  \path[draw=black,line join=miter,line cap=butt,miter limit=4.00,even odd
    rule,line width=0.160pt] (262.5532,126.4047) .. controls (262.5532,126.4047)
    and (272.3404,126.4047) .. (277.0213,124.2771) .. controls (281.7021,122.1494)
    and (287.2340,115.7665) .. (287.2340,115.7665);
  \path[draw=black,line join=miter,line cap=butt,miter limit=4.00,even odd
    rule,line width=0.240pt] (262.9787,137.0430) .. controls (262.9787,137.0430)
    and (272.3404,136.6175) .. (277.0213,134.4898) .. controls (281.7021,132.3622)
    and (286.8085,127.6813) .. (286.8085,127.6813);
  \path[draw=black,line join=miter,line cap=butt,miter limit=4.00,even odd
    rule,line width=0.240pt] (264.2553,146.8303) .. controls (264.2553,146.8303)
    and (274.0426,146.8303) .. (278.7234,144.7026) .. controls (283.4043,142.5749)
    and (287.6596,138.7452) .. (287.6596,138.7452);
  \path[draw=black,line join=miter,line cap=butt,miter limit=4.00,even odd
    rule,line width=0.240pt] (265.1064,157.0430) .. controls (265.1064,157.0430)
    and (274.0426,157.4686) .. (278.7234,155.3409) .. controls (283.4043,153.2132)
    and (287.6596,148.5324) .. (287.6596,148.5324);
  \path[draw=black,line join=miter,line cap=butt,miter limit=4.00,even odd
    rule,line width=0.160pt] (260.0000,116.1920) .. controls (260.0000,116.1920)
    and (271.9149,116.1920) .. (276.5958,114.0643) .. controls (281.2766,111.9366)
    and (286.8085,105.5537) .. (286.8085,105.5537);

\end{tikzpicture}
\caption{The convex body in Example~\ref{ex:noregbdy}}
\end{figure}

Since
\[
\frac{\ptl^2 f}{\ptl x\,\ptl y}=0,\qquad \frac{\ptl^2 f}{\ptl x^2}=\frac{g(z)}{\sqrt{x^2+g(z)^2}},\quad \frac{\ptl^2 f}{\ptl y^2}=\frac{g(z)}{\sqrt{y^2+g(z)^2}},
\]
and
\[
\det(\text{Hess}(f))=\frac{g(z)^5\big(\sqrt{x^2+g(z)^2}+\sqrt{y^2+g(z)^2}\,\big)g''(z)}{(x^2+g(z)^2)^2(y^2+g(z)^2)^2},
\]
the function $f$ is strictly convex. Hence the sublevel set $C:=\{(x,y,z):f(x,y,z)\le 1\}$ is a convex set. As
\[
\frac{\ptl f}{\ptl z}=\bigg(\frac{1}{\sqrt{x^2+g(z)^2}}+\frac{1}{\sqrt{y^2+g(z)^2}}\bigg)\,g(z)g'(z)\neq 0,
\]
the boundary $\ptl C$ is locally the graph of a $C^\infty$ function defined in the $xy$-plane by the Implicit Function Theorem. Since $g$ is a decreasing function, $\ptl C$ is a global graph. Observe that $C\cap\{z=z_0\}$ is empty for $z_0<1/2$ and that $C\cap\{z=0\}=\{(0,0,0)\}$. For $z_0>0$ the convex bodies $C_{z_0}:=\{(x,y,0):(x^2+g(z_0)^2)^{1/2}+(y^2+g(z_0)^2)^{1/2}\le 1\}$ satisfy $C_{z_0}\subset C_{z_1}$ when $z_0<z_1$ and the family $\{C_{z_0}\}_{z_0>0}$ converges in Hausdorff distance to $\{(x,y,0):|x|+|y|\le 1\}$ when $z_0\to +\infty$.

We conclude that $C$ is a cylindrically bounded convex body with $C^\infty$ boundary and, by Example~\ref{ex:cylindrically}, all its asymptotic cylinders are horizontal translations of $\{(x,y,z):|x|+|y|\le 1\}$.
\end{example}

\begin{example}[Closed $(n+1)$-dimensional convex cones are of uniform geometry]
\label{ex:convexcones}
Let $L\subset\rr^{n+1}$ be a closed convex cone with non-empty interior and vertex 0. Since $p+L\subset L$ for any $p\in L$, we have $p+\clb_L(0,r)\subset \clb_L(p,r)$ for any $r>0$. This implies
\[
\vol{\clb_L(p,r)}\ge\vol{\clb_L(0,r)}=\vol{\clb_L(0,1)}\,r^{n+1},
\]
for any $p\in L$ and $r>0$. Hence $L$ is of uniform geometry.
\end{example}

For convex cones we can prove that tangent cones out of a vertex are always asymptotic cylinders. From Lemma~\ref{lem:ciave}(ii) we may conclude that the limits of these tangent cones are also asymptotic cylinders. In general, we are able to prove that any asymptotic cylinder contains a tangent cone.

\begin{proposition}
\label{prop:convexcone}
Let $C\subset\rr^{n+1}$ be a closed $(n+1)$-dimensional convex cone such that $0\in \ptl C$ is a vertex of $C$. Then
\begin{enum}
\item For any $x\in C\setminus\{0\}$ and $\mu>0$, we have $-x+C_x=-\mu x+C_{\mu x}$. Moreover, the set $-x+C_x$ is a closed convex cylinder.
\item For any $x\in C\setminus\{0\}$, $-x+C_x$ is an asymptotic cylinder of $C$.
\item Let $\{x_i\}_{i\in\nn}$ be a divergent sequence in $\ptl C$ so that $-x_i+C$ converges locally in Hausdorff distance to $K\in\mathcal{K}(C)$. Assume that $z_i:=|x_i|^{-1}x_i$ converges to $z$ and that $-x_i+C_{x_i}$ converges to $K'$. Then $-z+C_z\subset K\subset K'$.
\end{enum}
\end{proposition}

\begin{proof}
To prove (i) we fix $\la>0$ and $c\in C$, so that
\[
-x+h_{x,\la}(c)=-\mu x+h_{\mu x,\mu^{-1}\la}(\mu c)\in -\mu x+C_{\mu x},
\]
since $\mu c\in C$. This implies that $-x+h_{x,\la}(C)\subset -\mu x+C_{\mu x}$. As $\la>0$ is arbitrary and $-\mu x+C_{\mu x}$ is closed we get $-x+C_x\subset -\mu x+C_{\mu x}$. The reverse inclusion is obtained the same way. The set $-x+C_x$ is trivially closed and convex. It is a cylinder since it contains the line $L(x)=\{tx:t\in\rr\}$.

For the proof of (ii) take $x\in C\setminus\{0\}$ and an increasing diverging sequence of positive real numbers $\la_i$. Let $x_i:=\la_i x$. Taking a subsequence if needed we may assume that $-x_i+C$ converges locally in Hausdorff distance to an asymptotic cylinder $K\in\mathcal{K}(C)$. We shall check that $K=-x+C_x$. Take first $z\in K$ so that there exists a sequence of points $c_i\in C$ such that $z=\lim_{i\to\infty} -x_i+c_i$. Since $C\subset C_{x}$, we get $z\in -x+C_x$. This implies $K\subset -x+C_x$. For the reverse inclusion fix some $\la>0$ and take $z\in -x+h_{x,\la}(C)$. Then there exists $c\in C$ such that $z=\la(c-x)$ and so
\[
z=-\la_i x+\la_i\big(x+\frac{\la}{\la_i}(c-x)\big).
\]
This implies that $z\in -x_i+C$ for $i$ large since $x+\tfrac{\la}{\la_i}(c-x)$ belongs to $C$ for $i$ large enough so that $\la/\la_i\le 1$. In particular, $z\in K$ and so we obtain $-x+h_{x,\la}(C)\subset K$. As $\la>0$ is arbitrary and $K$ is closed, from the definition of the tangent cone $C_x$ we obtain $-x+C_x\subset K$.

Let us check that (iii) holds. Take $\mu>0$ and a point $x\in -z+h_{z,\mu}(C)$. Then there exists $c\in C$ such that $x=\mu(c-z)$ and, setting $\la_{i}=|x_{i}|$, we have 
\begin{equation*}
x=\mu(c-z)=-x_i+\big(x_i+\frac{\mu}{\la_i}(\la_i(c-z+z_i)-x_i)\big).
\end{equation*}
We define
\[
d_i:=x_i+\frac{\mu}{\la_i}(\la_i c - x_i)
\]
and observe that $d_i\in C$ for large $i$ since $\la_i c\in C$ and $\mu/\la_i\le 1$. We have
\[
x-(-x_i+d_i)=\mu(-z+z_i)\to 0
\]
when $i\to\infty$. This implies that $x\in K$ and so we get $-z+h_{z,\mu}(C)\subset K$ for all $\mu>0$. As $K$ is closed we obtain $-z+C_z\subset K$.

Finally, take $x\in K$ and choose a sequence of points $c_i\in C$ such that $x=\lim_{i\to\infty} -x_i+c_i$. Since $C\subset C_{x_i}$, it follows that $x$ can be written as the limit of a sequence of points in $-x_i+C_{x_i}$. This implies that $x\in K'$.
\end{proof}

Let $C\subset\rr^{n+1}$ be a closed $(n+1)$-dimensional convex cone so that $0\in\ptl C$ is a vertex of $C$, and assume that $\ptl C\setminus\{0\}$ is of class $C^1$. Then any asymptotic cylinder of $C$ is either $\rr^{n+1}$ or a half-space. This is easy to check since, by Remark~\ref{rem:bdycylinder} we may obtain $\rr^{n+1}$ as an asymptotic cone by taking $x\in\intt(C)$ and a diverging sequence $\la_ix$, with $\lim_{i\to\infty}\la_i=+\infty$. On the other hand, if we take a diverging sequence in the boundary of $C$, Proposition~\ref{prop:convexcone}(iii) implies that the limit asymptotic cylinder $K$ contains a set of the form $-z+C_z$, with $z\in\ptl C\setminus\{0\}$. Since $\ptl C\setminus\{0\}$ is of class $C^1$, the set $-z+C_z$ is a half-space. Hence $K$ is a half-space.

In Proposition \ref{prop:NDACsmooth} below we show that the smoothness of the asymptotic cone $C_{\infty}$ implies that all asymptotic cylinders of $C$ are either $\rr^{n+1}$ or half-spaces. This fact will be of use in Section \ref{sec:isoprigid}: indeed we shall prove in Corollary \ref{cor:existforNDACsmooth} that unbounded convex bodies with a non-degenerate asymptotic cone, that is smooth except that in a vertex, admit isoperimetric solutions for any prescribed volume. The idea of the proof is that all asymptotic cylinders of $C_{\infty}$ are either $\rr^{n+1}$ or half-spaces, , as we have proved above, thus one has to show that this property can be transferred to $C$. 

\begin{proposition}\label{prop:NDACsmooth}
Let $C$ be an unbounded convex body with non degenerate asymptotic cone $C_{\infty}$. Assume that $0$ is a vertex of $C_{\infty}$ and that $\ptl C\setminus \{0\}$ is of class $C^{1}$. Then any asymptotic cylinder $K\in {\mathcal K}(C)$ is either $\rr^{n+1}$ or a half-space.
\end{proposition}

\begin{proof}
We assume $C\neq \rr^{n+1}$. By Remark~\ref{rem:bdycylinder} it is enough to consider asymptotic cylinders obtained from diverging sequences contained in the boundary of $C$.

We first prove the following fact: for every $\delta>0$, for any divergent sequence of points $x_{i}\in \ptl C$, and for any sequence of unit outer normal vectors $u_i\in N(C,x_{i})$, whenever $y_{i}\in \ptl C$ satisfies $|y_{i}-x_{i}|\le \delta$ then for every sequence of unit vectors $v_i\in N(C,y_i)$ one has
\begin{equation}\label{eq:closeplanes}
\lim_{i\to\infty} |u_i-v_i| = 0.
\end{equation}
To see this we argue by contradiction assuming the existence of positive constants $\delta,\eps>0$ and of sequences $y_{i}\in \ptl C$ and $v_i\in N(C,y_i)$, $|v_i|=1$, such that 
\[
|y_{i}-x_{i}|\le \delta\qquad \text{and}\qquad |u_i-v_i|\ge \eps,\qquad \text{for all }i\,.
\]
Let $t_{i} = |x_{i}|$ and define $z_{i} = t_{i}^{-1}x_{i}$, $w_{i} = t_{i}^{-1}y_{i}$. Clearly, as $i\to\infty$ we have $t_{i}\to +\infty$, $|z_{i}|=1$, $|w_{i}-z_{i}|\le t_{i}^{-1}\delta\to 0$, and $t_{i}^{-1}C\to C_{\infty}$ locally in Hausdorff distance, owing to the properties of the asymptotic cone $C_{\infty}$. At the same time, $u_i\in N(t_i^{-1}C,z_i)$ and $v_i\in N(t_i^{-1}C,w_i)$. By compactness, up to extracting a subsequence, we have that $z_{i}\to z\in \ptl C_{\infty}\cap \ptl B(0,1)$, $w_{i}\to z$, $u_i\to u_\infty$ and $v_i\to v_\infty$ with $u_\infty$, $v_\infty$ unit vectors in $N(C_\infty,z)$. However we have $|u_\infty-v_\infty| = \lim_{i\to\infty}|u_i-v_i|\ge \eps$, which contradicts the regularity of $\ptl C_{\infty}$ at $z$.

Now we prove that any asymptotic cylinder of $C$ is a half-space, so that we can conclude by Theorem \ref{thm:existence}. Let $K\in \mathcal K(C)$ and let $x_{i}\in \ptl C$ be a diverging sequence such that, setting $C_{i}=-x_{i}+C$, we have $C_{i}\to K$ locally in Hausdorff distance. Fix $\delta>0$ and choose $z\in \ptl K$ such that $|z|\le \delta/2$. Choose $u_\infty\in N(K,0)$ with $|u|=1$ such that $u_\infty$ is the limit of unit outer normal vectors $u_i\in N(K,x_i)$; then choose $v_\infty\in N(K,z)$ such that $|v|=1$.

Up to an isometry, and for $i$ large enough, the convex bodies $K$ and $C_{i}$ locally coincide with the epigraphs of convex functions defined on a relative neighborhood of $z$ in some supporting hyperplane for $K$ at $z$. Hence by Lemma \ref{lem:maybeattouch} we infer that there exist sequences $y_{i} \in \ptl C_{i}$, $v_i\in N(K,y_i)$ with $|v_i|=1$, such that $y_{i}-x_{i}\to z$, $|y_{i}-x_{i}|\le \delta$ and $\lim_{i\to\infty} v_i=v_\infty$. Now we recall \eqref{eq:closeplanes} and obtain 
\[
|u_\infty-v_\infty| = \lim_{i\to\infty} |u_i-v_i| = 0,
\]
so that $u_\infty=v_\infty$. This means that $K$ is necessarily a half-space and  the proof is completed.
\end{proof}

\chapter{A generalized existence result}
\label{sec:min}

\section{Preliminary results}

Let $C\subset \rr^{n+1} $ be an unbounded convex body of uniform geometry. The main result of this section is Theorem \ref{thm:genexist}, which shows existence of minimizers of the relative isoperimetric problem in $C$ in a generalized sense, for any prescribed volume $v>0$. More precisely we will show that there exists a finite family of sets $(E^{0},E^{1},\dots,E^{\ell})$, which satisfy  $E^{0}\subset C$, $E^j\subset K^j$, where $K^j$ is an asymptotic cylinder of $C$ for all $j\ge 1$, and
\begin{equation*}
\sum_{j=0}^{\ell}|E^{j}| = v,\qquad 
P_{C}(E^{0}) + \sum_{j=1}^{\ell} P_{K^{j}}(E^{j}) = I_{C}(v)\,.
\end{equation*}
Before stating and proving Theorem \ref{thm:genexist}, we need some preparatory results.

\begin{lemma}
\mbox{}
\label{lem:antier}
Let $\{C_i\}_{i\in\nn}$ be a sequence of bounded convex bodies in $\rr^{n+1}$ that converges in Hausdorff distance to a bounded convex body $C$. Let $E_i\subset C_i$ be finite perimeter sets so that $\{\pp_{C_i}(E_i)\}_{i\in\nn}$ is bounded. Then, possibly passing to a subsequence, $E_i\to E$ in $L^1(\rr^{n+1})$, where $E$ is a finite perimeter set in $C$, and
\begin{equation}
\label{eq:limper}
\pp_C(E)\le \liminf_{i\to\infty} \pp_{C_i}(E_i).
\end{equation}
\end{lemma}

\begin{proof}
Since $C_i$ converges to $C$ in Hausdorff distance and $C$ is bounded, there exist a Euclidean ball $B\subset\rr^{n+1}$ so that $C_i\subset B$ for all $i\in\nn$. By \cite[Cor.~1.3.6]{sch}, $\h^n(\ptl C_i)\le \h^n(\ptl B)$ for all $i\in\nn$. Hence
\[
\pp_{\rr^{n+1}}(E_i)\le \pp_{C_i}(E_i)+\h^n(\ptl C_i)\le \pp_{C_i}(E_i)+\h^n(\ptl B).
\]
Since the sequence $\{P_{C_i}(E_i)\}_{i\in\nn}$ is bounded by hypotheses, the previous inequality shows that the sequence $\{\pp_{\rr^{n+1}}(E_i)\}_{i\in\nn}$ is bounded. As $\{E_i\}_{i\in\nn}$ is uniformly bounded, a standard compactness result for finite perimeter sets, \cite[Thm.~12.26]{maggi}, implies the existence of a set $E\subset \rr^{n+1}$ and a non-relabeled subsequence $E_i$ such that $E_i\to E$ in $L^1(\rr^{n+1})$ and $\chi_{E_i}$ converges almost everywhere to $\chi_E$. By Kuratowski criterion, \cite[Thm.~1.8.7]{sch}, we get $E\subset C$.

Fix an open set $A\subset\subset \intt(C)$, then $A\subset\subset \intt(C_{i})$ for $i$ large enough and thus by the lower-semicontinuity of the perimeter we infer
\[
P(E,A) \le \liminf_{i\to \infty} P(E_{i},A) \le \liminf_{i\to\infty} P_{C_{i}}(E_{i})\,.
\]
Then, by recalling that $P(E,\cdot)$ is a Radon measure and by taking the supremum over $A\subset\subset \intt(C)$ we obtain
\[
P_{C}(E) =\sup_{A\subset\subset \intt(C)} P(E,A)\le \liminf_{i\to\infty} P_{C_{i}}(E_{i})\,,
\]
which proves \eqref{eq:limper}.
\end{proof}

\begin{lemma}
\mbox{}
\label{lem:niceapprox}
Let $\{C_i\}_{i\in\nn}$ be a sequence of convex bodies in $\rr^{n+1}$  locally converging in Hausdorff distance to a convex body $C$. Let $E\subset C$ be a bounded set of finite perimeter and volume $v$. Assume $v_i\to v$. Then there exists a sequence of bounded sets $E_i\subset C_i$ of finite perimeter such that $E_i\to E$ in $L^1(\rr^{n+1})$ and
\begin{enumerate}
\item[(i)] $E_i\to E$  in $L^1(\rr^{n+1})$,
\item[(ii)] $\vol{E_i} =v_i$ for all $i\in\nn$, and
\item[(iii)] $\pp_C(E)=\lim_{i\to\infty} \pp_{ C_i}(E_i)$.
\end{enumerate}
\end{lemma}

\begin{proof}
Let $B\subset \rr^{n+1}$ be a closed Euclidean ball so that $E\subset \intt(B)$. Since $C_i\cap B\to C\cap B$ in Hausdorff distance, then arguing as in the proof of \cite[Thm.~1.8.16]{sch} we can find $\la_i\to 1$ so that $\la_i (C_i\cap B)\subset C\cap B$ and $\la_i (C_i\cap B)\to C\cap B$ in Hausdorff distance.

Arguing as in \cite[IV.1.5]{maggi} we can find $\eps>0$ and a one-parameter family $\{\phi_{t}\}_{t\in (-\eps,\eps)}$ of diffeomorphisms with compact support in the interior of $C\cap B$ such that, setting $E(t)=\phi_{t}(E)$ one has $E(t)\subset C$, $|E(t)| = v+t$, and $P_{C}(E(t)) \le P_{C}(E) + c|t|$ for all $t\in (-\eps,\eps)$ and for some constant $c>0$ only depending on $E$. 

For fixed $t\in (-\eps,\eps)$,
\[
\lim_{i\to\infty}\vol{E(t)\cap \la_i(C_i\cap B)}=\vol{E(t)\cap (C\cap B)},
\]
since
\begin{multline*}
0<\vol{E(t)\cap (C\cap B)}-\vol{E(t)\cap \la_i(C_i\cap B)}
\\
\le \vol{E\cap (C\cap B\setminus \la_i(C_i\cap B))}
\le \vol{C\cap B\setminus \la_i(C_i\cap B)}\to 0,
\end{multline*}
as $\la_i(C_i\cap B)$ converges to $C\cap B$ in Hausdorff distance. Since $\vol{E(-\eps/2)}<v< \vol{E(\eps/2)}$, the above argument implies that, for large $i$, the inequalities
\[
 \vol{E(-\eps/2)\cap \la_i C_i}<\la_i^{n+1} v_i< \vol{E(\eps/2)\cap \la_i C_i}
\]
hold because $\la_i\to 1$. The continuity of the functions $t\mapsto \vol{E(t)\cap\la_i(C_i\cap B)}$ implies, for large $i$, the existence of $t(i)$ such that $\vol{E(t(i))\cap \la_i(C_i\cap B)}=\la_i^{n+1}v_i$.

Set $E_i=\la_i^{-1}(E(t(i))\cap\la_i(C_i\cap B))=\la_i^{-1}E(t(i))\cap (C_i\cap B)$. Clearly $\vol{E_i}=v_i$, hence (ii) is verified.
Moreover
\[
\pp_{\la_iC_i}(E(t(i))\cap \la_i(C_i\cap B))\le \pp_C(E(t(i)))\le\pp _C(E)+c\,t(i)
\]
and taking $i\to\infty$ we get
\begin{align*}
\limsup_{i\to\infty}\pp_{C_i}(E_i)&=\limsup_{i\to\infty}\pp_{\la_iC_i}(E(t(i))\cap \la_i (C_i\cap B))
\\
&\le \limsup_{i\to\infty}\pp_C(E(t(i)))\le\pp _C(E).
\end{align*}

On the other hand, by construction $E_i\to E$  in $L^1(\rr^{n+1})$ (which proves (i)) and $\{\pp_{C_i}(E_i)\}_{i\in\nn}$ is bounded. Therefore by Lemma~\ref{lem:antier} we get 
\[
\pp_{C}(E)\le \liminf_{i\to\infty} \pp_{ C_i}(E_i).
\]
Finally, combining the two above inequalities, we obtain (iii).
\end{proof}

The following two propositions are easily proved by means of Lemma \ref{lem:niceapprox}.

\begin{proposition}
\label{cor:niceapproxK}
Let $C\subset\rr^{n+1}$ be  an unbounded convex body, let $K\in \mathcal{K}(C)$ be an asymptotic cylinder of $C$, and let $E\subset K$ be a bounded set of finite perimeter and volume $v\ge 0$. Let $\{v_i\}_{i\in\nn}$ be any sequence of positive numbers converging to $v$. Then there exists a sequence $E_i\subset C$ of bounded sets of finite perimeter with $\vol{E_i}=v_i$ and $\lim_{i\to \infty} \pp_{C}(E_i)=\pp_{K}(E)$. In particular, this implies that $I_{C}(v)\le I_{K}(v)$.
\end{proposition}

\begin{proof}
By hypothesis, there exists a divergent sequence $\{x_i\}_{i\in\nn}\subset \rr^{n+1}$ so that $C_{i} = -x_i+C$ converges locally in Hausdorff distance to $K$, which is a convex set by Proposition~\ref{prop:maincond}. If $K$ has empty interior, then the claim trivially follows by setting $E_{i}:= \emptyset$. Otherwise if $K$ is an unbounded convex body, we obtain from Lemma~\ref{lem:niceapprox} the existence of a sequence $F_i\subset C_{i}$ of sets of finite perimeter and $\vol{F_i}=v_i$ satisfying $\lim_{i\to \infty} \pp_{C_{i}}(F_i)=\pp_{K}(E)$. Setting $E_i:=x_i+F_i\subset C$ we immediately prove the claim.
\end{proof}

\begin{proposition}
\label{prp:niceapproxI}
Let $\{C_i\}_{i\in\nn}$ be a sequence of unbounded convex bodies locally converging in Hausdorff distance to a convex body $C$. Let $v, v_{i}>0$ be such that $v_i\to v$ as $i\to\infty$.  Then
\begin{equation}
\label{eq:niceapproxI0}
\limsup_{i\to\infty}I_{C_i}(v_i)\le I_C(v).
\end{equation}
\end{proposition}

\begin{proof}
Owing to Proposition~\ref{prp:seqbound}, for a given $\eps>0$ we can find a bounded set $E\subset C$, of finite perimeter and volume $v$, such that
\begin{equation}
\label{eq:niceapproxI1}
\pp_C(E)\le I_C(v)+\eps.
\end{equation}
Lemma \ref{lem:niceapprox} implies the existence of a sequence of sets of finite perimeter $E_i\subset C_i$ with $\vol{E_i}=v_i$ and
\begin{equation}
\label{eq:niceapproxI2}
\lim_{i\to\infty}\pp_{C_i}(E_i)=\pp_C(E)
\end{equation}
Combining \eqref{eq:niceapproxI1} and \eqref{eq:niceapproxI2} we get $\limsup_{i\to\infty}I_{C_i}(v_i)\le I_C(v)+\eps$. Since $\eps>0$ is arbitrary, we obtain \eqref{eq:niceapproxI0}.
\end{proof}

The next lemma is a key tool for the proof of Theorem \ref{thm:genexist}.
\begin{lemma}
\label{lemma:splitseq}
Let $\{C_i\}_{i\in\nn}$ a sequence of unbounded convex bodies converging locally in Hausdorff distance to an unbounded convex body $C$ containing the origin. Let $E_i\subset C_i$ be a sequence of measurable sets with volumes $v_i$ converging to $v>0$ and uniformly bounded~perimeter. Assume $E\subset C$ is the $L_{loc}^1(\rr^{n+1})$ limit of $\{E_i\}_{i\in\nn}$. Then, passing to a (non relabeled) subsequence, there exist diverging radii $r_i>0$ such that
\[
E_i^d:=E_i\setminus B(0,r_i)
\] 
satisfies
\begin{enumerate}
\item[(i)] $|E|+\lim_{i\to\infty} |E_i^d|=v$.
\item[(ii)] $P_C(E)+\liminf_{i\to\infty} P_{C_i}(E_i^d)\le\liminf_{i\to\infty} P_{C_i}(E_i)$.
\end{enumerate}
\end{lemma}

\begin{proof}
The set $E$ clearly has finite volume less than or equal to $v$. Take a sequence of increasing radii $\{s_i\}_i$ such that $s_{i+1}-s_i\ge i$ for all $i\in\nn$. Since $E_i$ converges to $E$ in $L^1_{loc}(\rr^{n+1})$, taking a non relabeled subsequence of $E_i$ we may assume
\[
\int_{B(0,s_{i+1})} |\chi_{E_i}-\chi_{E}|\le\frac{1}{i}
\]
for all $i$. By the coarea formula
\[
\int_{s_i}^{s_{i+1}}{\h}^{n}(E_i \cap \ptl B(0,t))\, dt\le \int_{\rr}{\h}^{n}(E_i \cap \ptl B(0,t))\,dt= \vol{E_i}=v_i.
\]
Hence the set of $t\in [s_i,s_{i+1}]$ such that $H^n(E_i\cap\ptl B(0,t))\le \frac{2v_i}{i}$ has positive measure. By \cite[Chap.~28, p.~216]{maggi}, we can choose $ r_i\in [s_i,s_{i+1}]$ in this set so that
\begin{align*}
P_{C_i}(E_i\cap B(0,r_i))&=P(E_i,\intt{(C_i\cap B(0,r_i))})+H^{n}(E_i\cap \ptl B(0,r_i)),
\\
P_{C_i}(E_i\setminus B(0,r_i))&=P(E_i,\intt{(C_i\setminus B(0,r_i))})+H^{n}(E_i\cap \ptl B(0,r_i)).
\end{align*}
We define
\[
E_i^d:=E_i\setminus B(0,r_{i}).
\]
Then $\vol{E_i}=\vol{E_i\cap B(0,r_i)}+\vol{E_i^d}$. Since $E_i$ converges to $E$ in $L^1_{loc}(\rr^{n+1})$ and $E$ has finite volume, we have $\lim_{i\to\infty}\vol{E_i\cap B(0,r_i)}=\vol{E}$. This proves (i).

On the other hand,
\begin{align*}
P_{C_i}(E_i)&\ge P(E_i,\intt(C_i\cap B(0,r_i)))+P(E_i,\intt(C_i\setminus B(0,r_i)))
\\
&=P_{C_i}(E_i\cap B(0,r_i))+P_{C_i}(E_i^d)-\frac{4v_i}{i}.
\end{align*}
Taking inferior limits we have
\[
\liminf_{i\to\infty} P_{C_i}(E_i)\ge \liminf_{i\to\infty} P_{C_i}(E_i\cap B(0,r_i))+\liminf_{i\to\infty} P_{C_i}(E_i^d).
\]
By Lemma \ref{lem:antier} we finally get (ii).
\end{proof}

\section{Proof of the main result}

Before stating (and proving) the main result of this section, i.e. Theorem \ref{thm:genexist} below, we introduce some extra notation. 

We say that a finite family $E^{0},E^{1},\ldots,E^{k}$ of sets of finite perimeter is a \emph{generalized isoperimetric region in $C$ if} $E^0\subset C=K^0$, $E^i\subset K^i\in\mathcal{K}(C)$ for $i\ge 2$ and, for any family of sets $F^0,F^1,\ldots,F^k$ such that $\sum_{i=0}^k \vol{E^i}=\sum_{i=0}^k\vol{F^i}$, we have
\[
\sum_{i=0}^{k}P_{K^{i}}(E^{i}) \le \sum_{i=0}^{k}P_{K^{i}}(F^{i}).
\]
Obviously, each $E^i$ is an isoperimetric region in $K^i$ with volume $\vol{E^i}$.

\begin{theorem}\label{thm:genexist}
Let $C\subset \rr^{n+1}$ be a convex body of uniform geometry and fix $v_0>0$. Then there exists $\ell\in\nn$ with the following property: for any minimizing sequence $\{F_{i}\}_{i}$ for volume $v_0$, one can find a (not relabeled) subsequence $\{F_{i}\}_{i}$ such that, for every $j\in \{0,\dots,\ell\}$, there exist
\begin{itemize}
\item a divergent sequence ${\{x_i^j\}}_i$,
\item a sequence of sets ${\{F_i^j\}}_i$,
\item an asymptotic cylinder $K^j\in\mathcal{K}(C)$,
\item an isoperimetric region $E^j\subset K^j$ (possibly empty),
\end{itemize}
with in particular $x_i^0=0$ for all $i\in\nn$ and $K^0=C$, such that 
\begin{enumerate}
\item[(i)] $F_i^{j+1}\subset F_i^j\subset F_i$ for all $i\in \nn$ and $j\in \{0,\dots,\ell-1\}$;
\item[(ii)] $-x_i^j+C$ converges to $K^j$ locally in Hausdorff distance for all $j\in\{1,\ldots,\ell\}$;
\item[(iii)] $-x_i^j+F_i^j$ converges to $E^j\subset K^j$ in $L^1_{loc}(\rr^{n+1})$ for all $j\in\{0,\ldots,\ell\}$;
\item[(iv)] for any $0\le q\le\ell$, $E^0,E^1,\ldots,E^q$ is a generalized isoperimetric region of volume $\sum_{j=0}^q\vol{E^j}$.
\item[(v)] $\sum_{j=0}^\ell |E^j|=v_0$;
\item[(vi)] $I_C(v_0)=\sum_{j=0}^\ell I_{K^j}(|E^j|)$.
\end{enumerate}
\end{theorem}

\begin{proof}
We shall split the proof into several steps.

\emph{Step one}. Here we define the set $E^{0}$ as the $L^{1}_{loc}$ limit of ${\{F_{i}\}}_{i}$ up to subsequences, then show that it is isoperimetric for its volume. We henceforth assume that $C$ contains the origin. Let $E^0\subset C$ be the (possibly empty) limit in $L^1_{loc}(\rr^{n+1})$ of the sequence $\seq{F}$. By Lemma~\ref{lemma:splitseq} there exists a sequence of diverging radii $r_i^0>0$ so that the set $F_i^1:=F_i\setminus B(0,r_i^0)$ satisfies
\begin{equation}
\label{eq:volE^0}
\vol{E^0}+\lim_{i\to\infty}\vol{F_i^1}=v_0,
\end{equation}
and
\begin{equation*}
P_C(E^0)+\liminf_{i\to\infty} P_C(F_i^1)\le\liminf_{i\to\infty} P_C(F_i).
\end{equation*}
In case $\vol{E^0}>0$ the set $E^0$ is isoperimetric for its volume: otherwise there would exist a bounded measurable set $G^0\subset C$ satisfying $\vol{G^0}=\vol{E^0}$ and $P_C(G^0)<P_C(E^0)$. This set can be approximated by a sequence $\seq{G}$ of uniformly bounded sets of finite perimeter satisfying $\vol{G_i}+\vol{F_i^1}=v_0$ and $\lim_{i\to\infty} P_C(G_i)=P_C(G^0)$. For large $i$, $G_i$ and $F_i^1$ are disjoint, $|G_i\cup F_i^1|=v_0$ and it holds that
\begin{align*}
I_{C}(v_{0}) &\le \liminf_{i\to\infty} P_C(G_i\cup F_i^1)
\\
& =\liminf_{i\to\infty} (P_C(G_i)+P_C(F_i^1))=P_C(G^0)+\liminf_{i\to\infty} P_C(F_i^1)
\\
&<P_C(E^0)+\liminf_{i\to\infty} P_C(F_i^1)\le\liminf_{i\to\infty} P_C(F_i)=I_C(v_0),
\end{align*}
yielding a contradiction. A similar argument proves the equality
\begin{equation}
\label{eq:perE^0}
P_C(E^0)+\liminf_{i\to\infty} P_C(F_i^1)=I_C(v_0).
\end{equation}
In particular, (iv) is trivially satisfied for $q=0$.

\emph{Step two}. We have the following alternative: either $\vol{E^0}=v_0$ (which means that $E^0$ is an isoperimetric region of volume $v_0$ and thus the theorem is verified for $\ell = 1$, $x^{1}_{i}$ diverging such that $-x^{1}_{i}+C \to K^{1}$ in local Hausdorff distance, $E^{1} = \emptyset$, and $F^{1}_{i}$ defined as in step one) or $\vol{E^0}<v_0$. The latter case corresponds to a ``volume loss at infinity''. We observe that in this case the sequence ${\{F_i^1\}}_{i\in\nn}$ defined in step one satisfies the hypotheses of Corollary~\ref{cor:lemme2-b} since $\lim_{i\to\infty}|F_i^1|=v_0-|E^0|<v_0$ and $\liminf_{i\to\infty} P_C(F_i^1)\le I_H(v_0-|E^0|)$, where $I_H$ is the isoperimetric profile of a half-space in $\rr^{n+1}$. The last inequality follows by contradiction: in case $\liminf_{i\to\infty} P_C(F_i^1)> I_H(v_0-|E^0|)$, we could consider the union of the bounded isoperimetric set $E^0$ with a disjoint ball $B_C(x,r)$ centered at a boundary point $x\in\ptl C$ with volume $|B_C(x,r)|=v_0-|E^0|$. Since $P_C(B_C(x,r))\le I_H(|B_C(x,r)|)$, we would obtain
\begin{align*}
P_C(E^0\cup B_C(x,r))&\le P_C(E^0)+I_H(|B_C(x,r)|)\\ &< P_C(E^0)+\liminf_{i\to\infty} P_C(F_i^1)\le\liminf_{i\to\infty} P_C(F_i)=I_C(v_0),
\end{align*}
yielding again a contradiction. Since $\{F_i^1\}_{i\in\nn}$ satisfies the hypotheses of Corollary~\ref{cor:lemme2-b} we can find a sequence of points $x_i^1\in C$ so that
\[
|F_i^1\cap B_C(x_i^1,1)|\ge m_0|F_i^1|
\]
for all $i$. The sequence $\{x_i^1\}_{i\in\nn}$ is divergent since the sequence $\{F_i^1\}_{i\in\nn}$ is divergent. Then, possibly passing again to a subsequence, the convex sets $-x_i^1+C$ converge to an asymptotic cylinder $K^1$, and the sets $-x_i^1+F_i^1$ converge in $L^1_{loc}(\rr^{n+1})$ to a set $E^1\subset K^1$ of volume $v_0-|E^0|\ge |E^1|\ge m_0(v_0-|E^0|)$. We can apply Lemma~\ref{lemma:splitseq} to find a sequence of diverging radii $r_i^1$ so that the set $F_i^2\subset F_i^1\subset C$ defined by
\[
-x^{1}_{i} + F_i^2 =(-x^{1}_{i}+F_i^1)\setminus B(0,r_i^1)
\]
satisfies
\begin{align*}
|E^1|+\lim_{i\to\infty} |F_i^2|&=\lim_{i\to\infty}|F_i^1|,
\\
P_{K^1}(E^1)+\liminf_{i\to\infty} P_{C}(F_i^2)&\le\liminf_{i\to\infty} P_C(F_i^1),
\end{align*}
where in the last inequality we have used $P_{-x_i^1+C}(-x_i^1+F_i^2)=P_C(F_i^2)$. Equation \eqref{eq:volE^0} then implies
\begin{equation*}
|E^0|+|E^1|+\lim_{i\to\infty}|F_i^2|=v_0,
\end{equation*}
and \eqref{eq:perE^0} yields
\[
P_C(E^0)+P_{K^1}(E^1)+\liminf_{i\to\infty} P_C(F_i^2)\le I_C(v_0).
\]
Arguing in a similar way as in step one, we show that $E^1$ is isoperimetric for its volume in $K^1$. We argue by contradiction assuming the existence of a bounded measurable set $G^1\subset K^1$ with $P_{K^1}(G^1)<P_{K^1}(E^1)$. By Proposition~\ref{cor:niceapproxK} we would find a sequence of uniformly bounded subsets $G_i^1\subset -x_i^1+C$ with volumes $v_0-|E^0|-|F_i^2|$ such that $\lim_{i\to\infty} P_{-x_i^1+C}(G_i^1)=P_{K^1}(G^1)$. Since the sets $G_i^1$ are uniformly bounded and the sequence $F_i^2$ is divergent, we have $G_i^1\cap F_i^1=\emptyset$ for large $i$. Being the sequence $x_i^1$ divergent, the sets $x_i^1+G_i^1$, $x_i^1+F_i^2$ are disjoint from $E^0$. So the sets $E^0\cup(x_i^1+G_i^1)\cup (x_i^1+F_i^2)$ have volume $v_0$ and their perimeters have inferior limit less than or equal to
\[
P_C(E_0)+P_{K^1}(G^1)+\liminf_{i\to\infty} P_C(x_i^1+F_i^2),
\]
that, by the choice of $G^{1}$, is strictly less than
\[
P_C(E_0)+P_{K^1}(E^1)+\liminf_{i\to\infty} P_C(x_i^1+F_i^2)\le I_C(v_0),
\]
again yielding a contradiction. The same argument shows that
\begin{equation*}
P_C(E_0)+P_{K^1}(E^1)+\liminf_{i\to\infty} P_C(x_i^1+F_i^2)=I_C(v_0)\,.
\end{equation*}
Arguing similarly, we can then show that (iv) is satisfied for $q=1$.

\emph{Step three (induction)}. Assume that after repeating step two $q$ times we have found $q$ asymptotic cylinders $K^j$, $q$ isoperimetric regions $E^j\subset K^j$ of positive volume, $j=1,\ldots,q$, and a chain of (non relabeled) subsequences
\[
F_i^{q+1}\subset F_i^{q}\subset\cdots\subset F_i^1\subset F^1\,,
\]
so that
\begin{align}
\label{eq:induction-1}
\vol{E^0}+\vol{E^1}+\cdots+\vol{E^q}+\lim_{i\to\infty}\vol{F_i^{q+1}}&=v_0
\\
\label{eq:induction-2}
P_C(E^0)+P_{K^1}(E^1)+\cdots +P_{K^q}(E^q)+\liminf_{i\to\infty} P_C(F_i^{q+1})&=I_C(v_0).
\end{align}
If $\lim_{i\to\infty}\vol{F_i^{q+1}}=0$ we are done. Otherwise, if $\lim_{i\to\infty} |F_i^{q+1}|>0$ then we claim that the inequality
\begin{equation}
\label{eq:claim2}
\liminf_{i\to\infty} P_C(F_i^{q+1})\le I_H(v_0-\sum_{j=0}^q\vol{E^j})
\end{equation}
must be satisfied. In order to prove \eqref{eq:claim2} we reason by contradiction assuming that
\[
\liminf_{i\to\infty} P_C(F_i^{q+1})>I_H(v_0-\sum_{j=1}^q\vol{E^j})+\eps
\]
holds for some positive $\eps>0$. Recall that each isoperimetric set $E^j$ is bounded and that each asymptotic cylinder $K^j$ is the local limit in Hausdorff distance of a sequence $-x_i^j+C$, where $\{x_i^j\}_i$ is a diverging sequence in $C$. For each  $j\in 1,\ldots,q$, consider a sequence of uniformly bounded sets $G_i^j\subset -x_i^j+C$ of finite perimeter such that
\[
\lim_{i\to\infty} P_{-x_i^j+C}(G_i^j)=P_{K^j}(E^j),\text{ and }\vol{G_i^j}=\vol{E^j}\ \text{ for all } i.
\]
Since the sequences $\{x_i^j+G_i^j\}_i$ of subsets of $C$ are divergent, for every $j\in 1,\ldots,q$, we can find $i(j)$ so that the sets $G_j:=x_{i(j)}^{j}+G_{i(j)}^j$ are pairwise disjoint, do not intersect $E^0$, and
\[
P_C(G_j)<P_{K^j}(E^j)+\frac{\eps}{q}.
\]
Now choose some $x\in\ptl C$ so that the intrinsic ball $B$ centered at $x$ of volume $v_0-\sum_{j=0}^q\vol{E^j}$ is disjoint from $E^0\cup\bigcup_{j=1}^q G_j$. We know that $P_C(B)\le I_H(|B|)=I_H(v_0-\sum_{j=0}^{q}\vol{E^j})$. So we finally obtain that $E^0\cup B\cup \bigcup_{j=1}^q G_j$ has volume $v_0$, and
\begin{align*}
P_C(E^0\cup B\cup \bigcup_{j=1}^q G_j)&\le P_C(E^0)+\sum_{j=1}^q P_{K^j}(E^j)+I_H(|B|)+\eps.
\\
&<P_C(E^0)+\sum_{j=1}^q P_{K^j}(E^j)+\liminf_{i\to\infty} P(F_i^{q+1})=I_C(v_0),
\end{align*}
providing a contradiction and thus proving \eqref{eq:claim2}.

We can then apply Corollary~\ref{cor:lemme2-b} to obtain a divergent sequence of points $x_i^{q+1}$ so that
\[
|F_i^{q+1}\cap B_C(x_i^{q+1},1)|\ge m_0|F_i^{q+1}|
\]
for all $i$. Possibly passing to a subsequence, the sets $-x_i^{q+1}+C$ converge to an asymptotic cylinder $K^{q+1}$ and the sets $x_i^{q+1}+F_i^{q+1}$ to a set $E^{q+1}\subset K^{q+1}$ satisfying $v_0-\sum_{j=0}^q\vol{E^j}\ge \vol{E^{q+1}}\ge m_0\,(v_0-\sum_{j=0}^q\vol{E^j})$. We can use again Lemma~\ref{lemma:splitseq} to obtain a sequence of diverging radii $r_i^{q+1}$ such that
\[
-x_i^{q+1}+F_i^{q+2}:=(-x_i^{q+1}+F_i^{q+1})\setminus B(0,r_i^{q+1})
\]
satisfies 
\begin{align*}
|E^{q+1}|+\liminf_{i\to\infty} |F_i^{q+2}|&=|F_i^{q+1}|,
\\
P_{K^{q+1}}(E^{q+1})+\liminf_{i\to\infty} P_C(F_i^{q+2})&\le \liminf_{i\to\infty} |F_i^{q+1}|.
\end{align*}
From \eqref{eq:induction-1} we get
\[
\sum_{j=0}^{q+1}\vol{E^j}+\liminf_{i\to\infty} |F_i^{q+2}|=v_0,
\]
and, from \eqref{eq:induction-2} we obtain
\[
P_C(E^0)+\sum_{j=1}^{q+1} P_{K^j}(E^j)+\liminf_{i\to\infty} P_C(F_i^{q+2})\le I_C(v_0).
\]
Reasoning as above we conclude that $E^{q+1}$ is isoperimetric in $K^{q+1}$ and that equality holds in the above inequality, thus yielding
\[
P_C(E^0)+\sum_{j=1}^{q+1} P_{K^j}(E^j)+\liminf_{i\to\infty} P_C(F_i^{q+2})= I_C(v_0).
\]
Arguing as in step two, we obtain (iv) with $q+1$ in place of $q$. Moreover it is clear from the procedure illustrated above that (i)--(iii) will be granted at the end of the inductive process.

\emph{Step four (finiteness)}. 
Let us finally prove that the induction step needs to be repeated only a finite number $\ell-2$ of times. The key observation leading to this conclusion is the existence of a constant $\beta>0$, ultimately depending only on the domain $C$ and on an upper bound for the prescribed volume $v$, such that any $E^{j}$ with $j\ge 2$ obtained as in step three (with $|E^{j}|>0$) necessarily satisfies 
\begin{equation}\label{Ejbeta}
|E^{j}|\ge \beta\,.
\end{equation}
As an immediate consequence of \eqref{Ejbeta}, one obtains
\begin{equation}\label{ellbound}
\ell \le 2+\lfloor v/\beta\rfloor\,,
\end{equation}
where $\lfloor x\rfloor$ denotes the largest integer $\le x$. 
The property expressed by \eqref{ellbound} is actually stronger than a generic finiteness of $\ell$, as the right-hand side of \eqref{ellbound} does not depend upon the specific choices made during each application of step three.

In order to prove \eqref{Ejbeta} we first notice that
\begin{equation}\label{max01}
\max \{|E^{0}|,|E^{1}|\} \ge m_{0}v/2\,.
\end{equation}
Indeed, \eqref{max01} follows by arguing as in step two with the additional observation that, owing to Corollary \ref{cor:lemme2-b}, the following more precise alternative holds: either $|E^{0}|\ge v/2$, or $\lim_{i\to\infty}|F^{1}_{i}| =v-|E^{0}| \ge v/2$ and therefore $|E^{1}| \ge m_{0}v/2$. Thus, \eqref{max01} is proved by recalling that $m_{0}\le 1$. Let now assume without loss of generality that $|E^{1}|\ge |E^{0}|$ (indeed, the argument is even simpler in the opposite case). 

By \cite[Lemma II.6.21]{maggi} we can find a deformation $E^1_t$, parameterized by $t\in [-\delta,\delta]$ and obtained as a one-parameter flow associated with a vector field with compact support in the interior of $K^{1}$, such that $E^1_0=E^1$, $\vol{E^1_t}=\vol{E^1}+t$ and
\begin{equation}\label{gentlypush}
P_{K^1}(E^1_t)\le P_{K^1}(E^1)+M|t|,\quad t\in [-\delta,\delta]
\end{equation}
for some positive constant $M$ depending on $E^1$. 

Let now $\{E^{j}\subset K^{j}\}$ for $2\le j\le q$ be the sets obtained after applying step three $q-1$ times. Assuming that $v>\sum_{j=0}^{q-1}|E^{j}|$ (which is clearly the case in order to justify the application of step three $q-1$ times) is equivalent to require that $|E^{j}|>0$ for all $j=2,\dots,q$. Let now $j\in \{2,\dots,q\}$ be fixed and assume by contradiction that 
\begin{equation}\label{pushti}
t_{j}:=|E^{j}|< \min\{v_{0},\delta,\frac{c_{0}^{n+1}}{M^{n+1}}\}\,,
\end{equation}
where $c_{0}$ and $v_{0}$ are as in Corollary \ref{cor:isopinesm}. Owing to  Remark \ref{rmk:convinradius} one easily deduces that the isoperimetric inequality for small volumes stated in Corollary \ref{cor:isopinesm} is also valid for any $K\in {\mathcal K}(C)$ with the same constants $c_{0}$ and $v_{0}$, thus by \eqref{pushti} we get
\begin{equation}\label{isoKunif}
P_{K^{j}}(E^{j})\ge c_{0}|E^{j}|^{n/(n+1)}\,.
\end{equation}
Next we set $F^1 = E^{1}_{t_{j}}$, $F^{j}=\emptyset$ and $F^{i} = E^{i}$ for all $i\in\{0,\dots,q\}\setminus \{1,j\}$. Then we observe that $\sum_{i=0}^{q}|F^{i}| = \sum_{i=0}^{q}|E^{i}|$. On the other hand, by step three we know that 
$E^{0},\ldots,E^{q}$ is a generalized isoperimetric region and, at the same time, by \eqref{gentlypush}, \eqref{pushti} and \eqref{isoKunif}, we have that 
\begin{align*}
\sum_{i} P_{K^{i}}(F^{i}) = \sum_{i\neq j} P_{K^{i}}(E^{i}) + Mt_{i} < \sum_{i\neq j} P_{K^{i}}(E^{i}) + c_{0}t_{j}^{n/(n+1)} \le  \sum_{i} P_{K^{i}}(E^{i})\,,
\end{align*}
that is, a contradiction. Setting $\beta = \min\{\delta,\frac{c_{0}^{n+1}}{M^{n+1}}\}$, we have thus proved \eqref{Ejbeta}. Consequently, (v) and (vi) are now satisfied together with (i)--(iv), which concludes the proof of the theorem.
\end{proof}

We finally state the following result for future reference.

\begin{corollary}
\label{cor:antepol}
Let $C=K\times\rr^{k}$, where $K$ is an $(n+1-k)$-dimensional bounded convex body. Then isoperimetric regions exist in $C$  for all volumes.
\end{corollary} 

\begin{proof}
The proof follows from Theorem \ref{thm:genexist} and the trivial fact that any asymptotic cylinder of $C$ coincides with $C$ up to translation.
 \end{proof}

\chapter{Concavity of the isoperimetric profile}
\label{sec:cont}
Here we prove the continuity and then the concavity of the isoperimetric profile $I_{C}$ of an unbounded convex body of uniform geometry. In Theorem \ref{thm:conY}, the main result of the section, we also prove the connectedness of isoperimetric regions.

\section{Continuity of the isoperimetric profile}

\begin{theorem}[Continuity of the isoperimetric profile]
\label{thm:contprof}
Let $C\subset \rr^{n+1}$ be an unbounded  convex body. Then its isoperimetric profile $I_C$ is a continuous function.
\end{theorem}
\begin{proof}
We only need to consider the case of uniform geometry, that is, we can assume \eqref{eq:mainhyp} since otherwise $I_C\equiv 0$ (see Proposition \ref{prop:maincond}). We closely follow Gallot's proof \cite[Lemme~6.2]{gallot}. 

We choose two volumes $v<w$ in an open interval $J$ whose closure does not contain $0$. 
For any $\eps>0$ we consider a bounded set $E\subset C$ of volume $v$ (depending also on $\eps$) such that $P_C(E)\le I_C(v)+\eps$. Let $B\subset C$ be a closed intrinsic ball at positive distance from $E$ such that $|B|=w-v$. Then
\[
I_C(w)\le P_C(E\cup B)=P_C(E)+P_C(B)\le I_C(v)+\mu\,(w-v)^{n/(n+1)}+\eps,
\]
where $\mu>0$ is a constant depending on the geometry of $C$. Since $\eps>0$ is arbitrary, we get
\begin{equation}
\label{eq:gallot1}
I_C(w)\le I_C(v)+\mu\,(w-v)^{n/(n+1)}.
\end{equation}

Assume now that the closure of $J$ is the interval $[a,b]$, with $a>0$, and take $r>0$ so that
\[
r<\min\bigg\{\Lambda\,\frac{a}{I_H(b)+1},1\bigg\},
\]
where $H$ denotes a generic half-space in $\rr^{n+1}$ and $\Lambda$ is defined as in \eqref{eq:defLambda}.
With this choice it follows that, for any set $E$ of volume $|E|\in [a,b]$ such that $P_C(E)\le I_C(|E|)+\eps$, and $0<\eps\le 1$, we have
\[
r<\Lambda\,\frac{a}{I_H(b)+1}\le \Lambda\,\frac{|E|}{P_C(E)}\quad\text{and}\quad r<1\,.
\]
Hence we can apply Lemma~\ref{lem:lrlemme31} to conclude that, for every $E\subset C$ with $|E|\in [a,b]$ and $P_C(E)\le I_C(|E|)+\eps$, for $0<\eps\le 1$, there exists a point $x\in C$, depending on $E$, such that
\[
|E\cap B_C(x,r)|\ge \frac{b(1)}{2}\,r^{n+1}.
\]
If necessary, we reduce the size of the interval $J$ so that $|J|\le \tfrac{b(1)}{2}\,r^{n+1}$. Consider now a set $E\subset C$ with $|E|=w$ and $P_C(E)\le I_C(w)+\eps$ with $0<\eps<1$. Pick a point $x\in C$ so that $|E\cap B_C(x,r)|\ge |J|\ge w-v$. By the continuity of the function $\rho\mapsto |E\cap B_C(x,\rho)|$, we find some $s\in (0,r)$ such that $|E\cap B_C(x,s)|=w-v$. Then we have
\begin{align*}
I_C(v)&\le P_C(E\setminus (E\cap B_C(x,s)))\le P_C(E)+P_C(B_C(x,s))
\\
&\le I_C(w)+\mu\,(w-v)^{n/(n+1)}+\eps.
\end{align*}
As $\eps>0$ is arbitrary we get
\begin{equation}
\label{eq:gallot2}
I_C(v)\le I_C(w)+\mu\,(w-v)^{n/(n+1)}.
\end{equation}

Finally, from \eqref{eq:gallot1} and \eqref{eq:gallot2}, the continuity of $I_C$ in the interval $J$ follows. As $J$ is arbitrary, we conclude that $I_C$ is a continuous function.
\end{proof}

\section{Approximation by smooth sets}

Let $\rho:\rr^{n+1}\to\rr$ be the standard symmetric mollifier: the function $\rho$ is radial, has compact support in $\overline{B}(0,1)$, its integral over $\rr^{n+1}$ equals $1$, and the functions
\[
\rho_\eps(x):=\frac{1}{\eps^{n+1}}\,\rho\big(\frac{x}{\eps}\big)
\]
converge to Dirac's delta when $\eps\to 0$ in the sense of distributions.

Let $C\subset\rr^{n+1}$ be an unbounded convex body and let $d_C(x)$ denote the distance from $x\in\rr^{n+1}$ to $C$. It is well-known (cf. \cite[\S~1.2]{sch}) that $d_C$ is a 1-Lipschitz convex function, differentiable at any $x\in \rr^{n+1}\setminus C$, and such that $\nabla d_{C}(x)$ is a unit vector in $\rr^{n+1}\setminus C$. For every $\eps>0$ we define the smooth, convex, non-negative function
\begin{equation*}
g_{C,\eps}(x):=\int_{\rr^{n+1}}\rho_{\eps}(x-y)\,d_C(y)dy,
\end{equation*}
then we set
\begin{equation*}
C^{\eps}:=g_{C,\eps}^{-1}([0,\eps]). 
\end{equation*}

The following lemma allows us to approximate in local Hausdorff distance an unbounded convex body $C$ by unbounded convex bodies with $C^\infty$ boundaries. The approximation is strong in the sense that any asymptotic cylinder of $C$ is also approximated by asymptotic cylinders of the approaching convex bodies.
\begin{lemma}
\label{lem:fund}
 Let $C\subset \rr^{n+1}$ be an unbounded convex body, then
\mbox{}
\begin{enum}
\item $C\subset C^\eps \subset C+2\eps B$, where $B=\clb(0,1)$.
\item $C^{\eps}$ is an unbounded convex body with $C^\infty$ boundary.
\item For each $w\in\rr^{n+1}$, we get $w+C^{\eps}=(w+C)^{\eps}$.
\item Let $\{C_i\}_{i\in\nn}$ be a sequence of unbounded convex bodies that converges to an unbounded convex body $C$ locally in Hausdorff distance. Then also $(C_i)^{\eps}$ converges to $C^{\eps}$ locally in Hausdorff distance, as $i\to\infty$.
\item $\mathcal{K}(C^{\eps}) = (\mathcal{K}(C))^{\eps}:=\{K^\eps:K\in\mathcal{K}(C)\}$. In particular, any cylinder in $\mathcal{K}(C^{\eps})$ has $C^\infty$ boundary.
\end{enum}
\end{lemma}

\begin{proof}
Let $x\in C$, then $d_C(x)=0$ and since $d_C$ is 1-Lipschitz we get 
\begin{equation*}
\begin{split}
g_{C,\eps}(x)&=\int_{\rr^{n+1}}\rho_{\eps}(x-y)\,d_C(y)\,dy \le\int_{\rr^{n+1}}\rho_{\eps}(x-y)\,|d_C(y)-d_C(x)|\,dy 
\\
&\le\int_{B(x,\eps)}\rho_{\eps}(x-y)\,|y-x|\,dy\le\eps\int_{B(x,\eps)}\rho_{\eps}(x-y)\,dy=\eps.
\end{split}
\end{equation*}
Consequently, $x\in C^\eps$ and so $C\subset C^\eps$. Assume now that $x\in C^\eps$. Then $g_{C,\eps}(x)\le \eps$ and we get
\begin{equation*}
d_C(x)-g_{C,\eps}(x)\le\int_{B(x,\eps)}\rho_{\eps}(x-y)\,|d_C(x)-d_C(y)|\,dy\le\eps.
\end{equation*}
Thus, $d_C(x)\le2\eps$. Consequently, $x\in C+2\eps B$ and so $C^\eps \subset C+2\eps B$. This proves (i).

We now prove (ii). By (i) we get that $C^\eps$ is convex since it is the sublevel set of a convex function. As it contains $C$, it is necessarily unbounded. If $x\in C$, then
\[
g_{C,\eps}(x)=\int_{B(x,\eps)}\rho_\eps(x-y)\,d_C(y)\,dy\le \int_{B(x,\eps)}\rho_\eps(x-y)\,|x-y|\,dy<\eps,
\]
thus $\eps$ is not the minimum value of $g_{C,\eps}$. Consequently, $\nabla g_{C,\eps}(z)\neq 0$ for every $z\in\ptl C^\eps$, hence $\ptl C^\eps$ is smooth.

Item (iii) follows easily from the equalities $d_{w+C}(x)=d_C(x-w)$ and $g_{C,\eps}(x-w)=g_{w+C,\eps}(x)$.

We now prove (iv). We fix $R>0$ and check that $(C_i)^\eps\cap\clb(0,R)$ converges in Hausdorff distance to $C^\eps\cap\clb(0,R)$. To this aim, we exploit Kuratowski criterion \cite[Thm.~1.8.7]{sch}. First, let $x\in C^\eps\cap\clb(0,R)$. We need to check that $x$ is the limit of a sequence of points in $(C_i)^\eps\cap\clb(0,R)$. If $x\in\intt(C^\eps)\cap\clb(0,R)$ then $g_{C,\eps}(x)<\eps$ and thus $g_{C_i,\eps}(x)<\eps$ (henceforth $x\in (C_i)^\eps$) for $i$ large enough. Otherwise we approximate $x$ by a sequence $\{y_{j}\}_{j\in \nn}\subset \intt(C^\eps)\cap\clb(0,R)$, then, arguing as above, for any $j\in \nn$ we select $i_{j}\in \nn$ with the property that $i_{j+1}>i_{j}$ for all $j$ and $y_{j}\in \intt(C_{i}^{\eps})$ for all $i\ge i_{j}$. In order to build a sequence of points $x_{i}\in C_{i}$ that converge to $x$, we proceed as follows. First we arbitrarily choose $x_{i}\in C_{i}$ for $i=1,\dots,i_{1}-1$. Then we set $x_{i} = y_{j}$ for all $i_{j}\le i<i_{j+1}$ (notice that the definition is well-posed, thanks to the fact that $\{i_{j}\}_{j\in \nn}$ is strictly increasing). It is then easy to check that the sequence $\{x_{i}\}_{i\in \nn}$ has the required properties.

Second, let $x_{i_k}\in C_{i_k}^\eps\cap\clb(0,R)$ converge to some point $x$. Since $g_{C_{i_k},\eps}(x_{i_k})\le\eps$ and $g_{C_{i_k},\eps}$ uniformly converges in compact sets to $g_{C,\eps}$, we have $g_{C,\eps}(x)\le\eps$ and so $x\in C^\eps$.

Now we prove (v). Let $K\in\mathcal{K}(C^{\eps})$. Then there exists a divergent sequence $\{x_i\}_{i\in\nn}\subset C^\eps$ so that, by (iii), $-x_i+C^{\eps}=(-x_i+C)^\eps\to K$ in local Hausdorff distance. Let $y_i$ be the metric projection of $x_i$ onto $C$. We have $|x_i-y_i|\le 2\eps$ and so the sequence $\{y_i\}_{i\in\nn}\subset C$ is divergent. Hence $(-y_i+C)^\eps\to K$ in local Hausdorff distance. Since $-y_i+C$ subconverges to some $K'\in \K(C)$, (iv) implies $(K')^\eps=K$ and so $K\in (\K(C))^\eps$. The proof of the reverse inclusion is similar.
\end{proof}

\section{Concavity of the isoperimetric profile}

Now we proceed to show in Lemma~\ref{lem:imp} that the function $I_C^{(n+1)/n}$ is concave when $C$ is a convex body of uniform geometry with $C^\infty$ boundary, such that all its asymptotic cylinders have also $C^\infty$ boundary. The general case of $C$ convex but not necessarily smooth will then follow by approximation. Lemma~\ref{lem:profilescomparison} is a technical result that will be needed in the proof of Lemma~\ref{lem:imp}.

\begin{lemma}
\label{lem:profilescomparison}
Let $C\subset\rr^{n+1}$ be a convex body of uniform geometry. Let $K^j\in\mathcal{K}(C)\cup\{C\}$ for $j=0,\ldots,m$. Then for all $v_{0},\ldots,v_{m}\ge 0$ we have
\[
I_C(v_0+\ldots+v_m)\le\sum_{j=0}^m I_{K^j}(v_j).
\]
\end{lemma}

\begin{proof}
For every $j$, consider a bounded set $E^j\subset K^j$ of volume $v_j$ such that $P_{K^j}(E^j)<I_{K^j}(v_j)+\eps$. Using Proposition~\ref{cor:niceapproxK}, we get a set $F^j\subset C$ of volume $v_j$ such that $P_C(F^j)<P_{K^j}(E^j)+\eps$. The sets $F^j$ can be taken disjoint. Then we have
\begin{align*}
I_C(v_0+\ldots +v_m)&\le P_C(F^0\cup\ldots\cup F^m)=P_C(F^0)+\ldots +P_C(F^m)
\\
&\le P_{K^0}(E^0)+\ldots +P_{K^m}(E^m)+(m+1)\,\eps
\\
&\le I_{K^0}(v_0)+\ldots +I_{K^m}(v_m)+2\,(m+1)\,\eps.
\end{align*}
Letting $\eps\to 0$ we obtain the result.
\end{proof}

\begin{lemma}
\label{lem:imp}
Let $C\subset\rr^{n+1}$ be a convex body of uniform geometry with $C^\infty$ boundary. Assume that all its asymptotic cylinders have also $C^\infty$ boundary. Then $I_C^{(n+1)/n}$ is a concave function, hence in particular $I_{C}$ is strictly concave.
\end{lemma}

\begin{proof}
Fix some positive volume $v_0>0$. By Theorem \ref{thm:genexist}, there  exist $m\in\nn$, $K^j \in\mathcal{K}(C)$, $1\le j\le m$, and isoperimetric regions $E^j\subset K^j$, such that 
\begin{equation}
\label{eq:imp1}
\sum_{j=0}^m\vol{E^j}=v_0\quad\text{and}\quad  I_C(v_0)=\sum_{j=0}^m\,I_{K^j}(\vol{E^j}),
\end{equation}
where $K^0=C$. Denote by $S^j$ the regular part of $\ptl_{K^j} {E^j}$ and assume that $S^j$ is nonempty for all $j=0,\dots,m$ (otherwise we may restrict the summation to those indices $j$ such that this property holds true). Since $E^j\subset K^j$ are isoperimetric sets in $K^j$, a standard first variation argument implies that $S^j$ has constant mean curvature and that $S^j$ intersects orthogonally $\ptl {K^j}$.

Let us check that all $S^j$ have the same constant mean curvature. Otherwise, there exist $S^{j_1}, S^{j_2}$, $j_1$, $j_2\in\{0,\ldots, m\}$, with different mean curvatures. A standard first variation argument allows us to deform $E^{j_1}$ and $E^{j_2}$ to get $F^{j_1}\subset K^{j_1}$ and $F^{j_2}\subset K^{j_2}$ satisfying
\begin{equation}
\label{eq:imp1a}
\begin{split}
|F^{j_1}|+|F^{j_2}|&=|E^{j_1}|+|E^{j_2}|,
\\
P_{K^{j_1}}(F^{j_1})+P_{K^{j_2}}(F^{j_2})&<P_{K^{j_1}}(E^{j_1})+P_{K^{j_2}}(E^{j_2}).
\end{split}
\end{equation}
Moreover the sets $F^{j_1}$, $F^{j_2}$ are bounded since they are nice deformations of isoperimetric regions, which are bounded by Proposition \ref{prp:isopbound}. Letting $F_i=E_i$ when $i\neq j_1,j_2$, we get from \eqref{eq:imp1} and \eqref{eq:imp1a}
\begin{equation*}
\sum_{j=0}^m\vol{F^j}=v_0\quad\text{and}\quad  I_C(v_0)>\sum_{j=0}^m\,I_{K^j}(\vol{F^j}),
\end{equation*}
Using Proposition~\ref{cor:niceapproxK}, we can approximate the sets $F^j\subset K^j$ by sets in $C$ of volumes $|F^j|$ and relative perimeters in $C$ as close as we wish to $P_{K^j}(F^j)$. This way we get a finite perimeter set $\Om\subset C$ so that
\begin{equation*}
|\Om|=v_0\quad\text{and}\quad I_C(v_0)>P_C(\Om),
\end{equation*}
yielding a contradiction.

Let us now prove that $I_C^{(n+1)/n}$ is a concave function. We closely follow the proof of Theorem~3.2 in \cite{bay-rosal}. We choose a family of functions $\{\varphi_\eps^j\}_{\eps>0}$ defined on $S^j\subset \ptl_{K^j}E^j$ as in \cite[Lemma~3.1]{bay-rosal}. These functions satisfy $0\le\varphi_\eps^j\le 1$, and $\varphi_{\eps}^{j}$ converges to the constant function $1$ on $S^j$ both pointwise and in the Sobolev norm when $\eps\to 0$.

Fix $\eps>0$ and consider a  $C^\infty$ vector field $X_\eps^j$ in $K^j$ whose associated flow $\{\psi^j_{\eps,t}\}_{t\in\rr}$ preserves the boundary of $K^j$ and such that $X_\eps^j=\varphi_\eps^j N^j$ on $S^j$, where $N^j$ is the outer unit normal to $E^j$ on $S^j$. The vector field $X_\eps^j$ is obtained by extending the vector field $\varphi_\eps^jN^j$, defined on the regular part $S^j$ of $\ptl_{K^j}E^j$. The derivative of the volume for this variation is equal to
\[
\int_{S^j}\varphi_\eps^j\,d\h^n>0,
\]
and so there exists a function $P_\eps^j(v)$ assigning to $v$ close to $|E^j|$ the perimeter of the set $\psi_{\eps,t(v)}^j(E^j)$ of volume $v$. Trivially we have
\begin{equation}
\label{eq:ikjcomp}
I_{K^j}(v)\le P_\eps^j(v).
\end{equation}
For $v$ close to $v_0$, we define the function
\[
P_\eps(v)=\sum_{j=0}^m P_\eps^j(|E^j|+\la_j(v-v_0)),
\]
where

\[
\la_j=\frac{\h^n(S^j)}{\h^n(S^0)+\ldots +\h^n(S^m)}=\frac{\h^n(S^j)}{I_C(v_0)}.
\]
Observe that $\sum_{j=0}^m\la_j=1$. Using Lemma~\ref{lem:profilescomparison} and \eqref{eq:ikjcomp} we have
\begin{align*}
I_C(v)^{(n+1)/n}&\le \bigg(\sum_{j=0}^m I_{K^j}(|E^j|+\la_j(v-v_0))\bigg)^{(n+1)/n}
\\
&\le \bigg(\sum_{j=0}^m P_\eps^j(|E^j|+\la_j(v-v_0))\bigg)^{(n+1)/n}=P_\eps(v)^{(n+1)/n}.
\end{align*}
By the arguments in the proof of Theorem~3.2 in \cite{bay-rosal}, it is enough to show
\begin{equation*}
\limsup_{\eps\to 0}\bigg(\frac{d^2}{dv^2}\bigg|_{v=v_0} P_\eps(v)^{(n+1)/n}\bigg)\le 0
\end{equation*}
to prove the concavity of $I_C^{(n+1)/n}$. Observe that
\[
\frac{d^2}{dv^2}\bigg|_{v=v_0} P_\eps(v)^{(n+1)/n}=\bigg(\frac{n+1}{n}\bigg)\,P^{1/n}_\eps(v_0)\,\bigg\{\frac{1}{n}\,\frac{P_\eps'(v_0)^2}{P_\eps(v_0)}+P''_\eps(v_0)\bigg\}.
\] 
Note that
\[
P_\eps'(v_0)=\sum_{j=0}^m \la_j(P_\eps^j)'(|E^j|)=\bigg(\sum_{j=0}^m\la_j\bigg)\,H=H,
\]
where $H$ is the common constant mean curvature of $S^j$ for all $j=0,\dots,m$, and that
\begin{align*}
P''_\eps(v_0)&=\sum_{j=0}^m \la_j^2(P^j_\eps)''(|E^j|)
\\
&=\sum_{j=0}^m\la_j^2\,\bigg(\int_{S^j}\varphi_\eps^j\bigg)^{-2}\!\bigg\{\!
\int_{S^j} \!|\nabla_{S^j}\varphi_\eps^j|^2-|\sg^j|^2(\varphi_\eps^j)^2
-\int_{\ptl S^j}\text{II}^j(N^j,N^j)(\varphi_\eps^j)^2\bigg\},
\end{align*}
where $\nabla_{S^j}$ is the gradient in $S^j$, $|\sg^j|^2$ is the squared norm of the second fundamental form $\sg^j$ of $S^j$, and $\text{II}^j$ is the second fundamental form of $\ptl K^j$. Taking limits when $\eps\to 0$ we get as in \cite[(3.7)]{bay-rosal}
\begin{align*}
\limsup_{\eps\to 0} P_\eps''(v_0)&\le -\sum_{j=0}^m\frac{\la_j^2}{\h^n(S^j)^2}\,\bigg\{\int_{S^j}|\sg^j|^2+\int_{\ptl S^j}\text{II}^j(N^j,N^j)\bigg\}
\\
&\le -\sum_{j=0}^m\frac{\la_j^2}{\h^n(S^j)^2}\int_{S^j}|\sg^j|^2.
\end{align*}
Hence we have
\begin{align*}
\limsup_{\eps\to 0} \bigg\{\frac{1}{n}\,\frac{P_\eps'(v_0)^2}{P_\eps(v_0)}+P''_\eps(v_0)\bigg\}&\le \frac{1}{n}\frac{H^2}{I_C(v_0)}-\sum_{j=0}^m\frac{\la_j^2}{\h^n(S^j)^2}\int_{S^j}|\sg^j|^2
\\
&=\sum_{j=0}^m\frac{1}{I_C(v_0)^2}\int_{S^j} \frac{H^2}{n}-\sum_{j=0}^m\frac{\la_j^2}{\h^n(S^j)^2}\int_{S^j}|\sg^j|^2
\\
&=\frac{1}{I_C(v_0)^2}\sum_{j=0}^m\int_{S^j}\bigg(\frac{H^2}{n}-|\sg^j|^2\bigg)\le 0,
\end{align*}
from the definition of $\la_j$. Since the second lower derivative of $I_C^{(n+1)/n}$ is non-negative and $I_C^{(n+1)/n}$ is continuous according to Theorem \ref{thm:contprof}, then Lemma~3.2 in \cite{MR1803220} implies that $I_C^{(n+1)/n}$ is concave and hence non-decreasing. Then $I_C$ is strictly concave, being the composition of $I_C^{(n+1)/n}$ with the strictly concave non-increasing function $x\mapsto x^{n/(n+1)}$.
\end{proof}

\begin{lemma}
\label{lem:one-diam}
Let $C\subset\rr^{n+1}$ be an unbounded convex body of bounded geometry such that $I_K^{(n+1)/n}$ is concave for any $K\in\{C\}\cup\mathcal{K}(C)$. Then, for any prescribed volume $v>0$, any generalized isoperimetric region in $C$ of volume $v$ consists of a single, connected set $E$ contained in $K\in\{C\}\cup\mathcal{K}(C)$. 
Moreover, the diameter of $E$ is bounded above by a constant only depending on $v$ and on the constants $n,r_{0},b(r_{0})$ appearing in \eqref{eq:mainhyp}.
\end{lemma}

\begin{proof}%
\textit{Step one}. We assume by contradiction that for some positive prescribed volume we can find a generalized isoperimetric region $E^0,E^1,\ldots,E^q$ such that $|E^{j}|>0$ at least for two distinct indices $j=j_{1},j_{2}$, $j_{1}\neq j_{2}\in \{0,\dots,q\}$. 
Let $v^{j}=\vol{E^j}$ and set $v = v^{j_{1}}+v^{j_{2}}$. Then combining Proposition \ref{cor:niceapproxK} with Theorem \ref{thm:genexist} (iv) we get
\begin{equation}
\label{eq:one-diam2}
I_C(v) = I_{K^{j_{1}}}(v^{j_{1}}) + I_{K^{j_{2}}}(v^{j_{2}}) \ge I_{C}(v^{j_{1}}) + I_{C}(v^{j_{2}}).
\end{equation}
On the other hand, the strict concavity of $I_{C}(v)$ implies strict subadditivity, hence we find
\[
I_C(v)< I_{C}(v^{j_{1}}) + I_{C}(v^{j_{2}})\,,
\]
which is in contradiction with \eqref{eq:one-diam2}.

\textit{Step two}. In order to prove the last part of the statement, we observe that any minimizer $E\subset C$ (or, respectively, $E\subset K$ for some $K\in \K(C)$) for a prescribed volume $v>0$ must satisfy a uniform density estimate depending only on the dimension $n$ and on the ratio $I_{C}(v)/v$. Owing to Remark \ref{rmk:convinradius}, we can consider without loss of generality the case $E\subset C$, fix $x\in E$ and, for any $r>0$, consider the set $E_{r} = E\setminus B_{C}(x,r)$. Set $m(r) = |E\cap B_{C}(x,r)|$ and notice that, by concavity of $I_{C}$, one has
\begin{equation}\label{stimetta1}
P_{C}(E) = I_{C}(v) \le I_{C}(v-m(r)) +\frac{I_{C}(v)}{v} m(r) \le P_{C}(E_{r}) +\frac{I_{C}(v)}{v} m(r)\,. 
\end{equation}
On the other hand, for almost all $0<r<r_{0}$ such that $m(r) \le |B_{C}(x,r)|/2$, owing to the relative isoperimetric inequality \eqref{eq:isnqgdbl1} one has
\begin{equation}
\begin{split}
\label{stimetta2}
P_{C}(E_{r}) &= P_{C}(E) - P_{C}(E\cap B_{C}(x,r)) + 2m'(r) 
\\
&\le P_{C}(E) - M\, m(r)^{\frac{n}{n+1}} + 2m'(r).
\end{split}
\end{equation}
Hence combining \eqref{stimetta1} and \eqref{stimetta2} we get
\begin{equation}\label{diffineqbis}
M\, m(r)^{\frac{n}{n+1}} - \frac{I_{C}(v)}{v} m(r) \le 2m'(r)\,.
\end{equation}
Now, assume that for some $0<r_{1}<r_{0}$ we have $m(r_{1}/2)>0$ and $m(r_{1}) \le |B_{C}(x,r_{1}/2)|/2$, then of course we have $m(r)\le m(r_{1})\le |B_{C}(x,r)|/2$ for all $r_{1}/2<r<r_{1}$, hence \eqref{diffineqbis} holds true for almost all $r\in (r_{1}/2,r_{1})$. Moreover, up to choosing $r_{1}$ small enough depending on the ratio $\frac{I_{C}(v)}{v}$ and on the isoperimetric constant $M>0$ appearing in \eqref{diffineqbis} (we recall that, by Lemma \ref{lem:inrad}, $M$ depends only on $n,r_{0},b(r_{0})$), we can entail that
\begin{equation}\label{diffineqter}
\frac{M}{4}\le \frac{m'(r)}{m(r)^{\frac{n}{n+1}}}\qquad\text{for almost all }r\in (r_{1}/2,r_{1})\,.
\end{equation}
Arguing as in the proof of Proposition \ref{prp:isopbound}, i.e. integrating \eqref{diffineqter} between $r_{1}/2$ and $r_{1}$, we conclude that 
\begin{equation}\label{uniflowerbound}
m(r_{1}) \ge \left(\frac{M}{8(n+1)}\right)^{n+1} r_{1}^{n+1}\,,
\end{equation}
i.e., that a uniform lower bound for the volume of $E$ holds in balls of radius smaller than $r_{0}$ centered at Lebesgue points of the characteristic function of $E$. Note that this lower bound is uniform, while the one given by Proposition \ref{prp:isopbound} is a-priori dependent on the set $E$.

Finally, by combining \eqref{uniflowerbound} with the connectedness of $E$ proved in step one, we eventually obtain a uniform lower bound on the diameter of $E$. To prove this we fix a maximal family $\mathcal B$ of disjoint balls of radius $r = r_{0}/2$ centered at Lebesgue points of the characteristic function of $E$, so that the union of the concentric balls with radius $2r$ covers $E$. By \eqref{uniflowerbound} we obtain $|E\cap B| > c_0$ for any $B\in \mathcal B$ and for a constant $c_{0}$ only depending on $n,r_{0},b(r_{0})$ and $v$. Consequently the cardinality of $\{\mathcal B\}$ cannot exceed $v/c_0$. On the other hand the union of the concentric balls with radius $2r$ must be connected (otherwise $E$ would be disconnected) and thus the diameter of $E$ is necessarily bounded by the sum of the diameters of these balls, i.e., by $4r_{0} v/c_0$.
\end{proof}

\begin{theorem}
\label{thm:conY}
Let $C\subset \rr^{n+1}$ be an unbounded convex body of uniform geometry. Then $I_{C}^{(n+1)/n}$ is concave. Moreover, any generalized isoperimetric region for $I_{C}(v)$ is associated with a connected set $E\subset K$ with $|E|=v$, for $K\in \{C\}\cup \K(C)$ suitably chosen, so that it holds $P_{K}(E) = I_{K}(v) = I_{C}(v)$.
\end{theorem}
\begin{proof}
We notice that, by the assumption on $C$, any $K\in {\mathcal K}(C)$ is an unbounded convex cylinder of uniform geometry. According to Lemma \ref{lem:fund} we approximate $C$ by a sequence of smooth unbounded convex bodies $C^{i}$ converging to $C$ in global Hausdorff distance as $i\to \infty$. It is then immediate to prove that $C^{i}$ is of uniform geometry. Moreover, thanks to Lemma \ref{lem:fund}, any $K\in \K(C^{i})$ has smooth boundary and is of uniform geometry. Therefore we deduce by Lemma \ref{lem:imp} that $I_{K}^{\frac{n+1}{n}}$ is concave for all $K\in \{C^{i}\}\cup \K(C^{i})$ and for all $i$. 

To deduce the concavity of $I_{C}^{(n+1)/n}$, it is enough to show that $\lim_{i\to\infty} I_{C^{i}}(v)=I_{C}(v)$ for all $v>0$. Owing to  Proposition \ref{prp:niceapproxI}, it remains to prove the lower semicontinuity of the isoperimetric profile, i.e. that
\begin{equation}\label{lscprofile}
\liminf_{i\to\infty} I_{C^{i}}(v) \ge I_{C}(v)\,.
\end{equation}
To this aim, by Lemma \ref{lem:one-diam} we find $K^{i}\in C^{i}\cup \K(C^{i})$ and $E^{i}\subset K^{i}$ with $|E^{i}|=v$, such that $P_{K^{i}}(E^{i}) = I_{C^{i}}(v)$ for all $i$. By Remark \ref{rmk:convinradius} and Lemma \ref{lem:one-diam}, the diameter of $E^{i}$ is uniformly bounded by some uniform constant $d>0$, hence we can assume that, up to translations, $E^{i}\subset B(0,d)$ for all $i$. Up to subsequences, the corresponding translates of $K^{i}$ converge in local Hausdorff sense to a limit convex set $K$ that necessarily belongs to $\{C\}\cup \K(C)$ up to a translation, see Lemmata \ref{lem:ciave} and \ref{lem:zK}. By the uniform boundedness of the perimeter of $E^{i}$ (see Remark \ref{rem:half-plane}), up to a further extraction of a subsequence we can assume that $E^{i}$ converges to a limit set $E\subset K\cap B(0,d)$ in $L^{1}$, whence $|E| = v$. By Lemma \ref{lem:antier}, we deduce that $P_{K}(E)\le \liminf_{i\to\infty}P_{K^{i}}(E^{i})$, hence by the inequality $I_{C}(v)\le I_{K}(v)$ (see Proposition \ref{cor:niceapproxK}) we finally deduce \eqref{lscprofile}. 
This concludes the proof of the theorem.
\end{proof}

\begin{corollary}
\label{cor:novo}
Let $C\subset\rr^{n+1}$ be a convex body and $\la<1$. Then $I_C(v) \ge I_{ \la C}(v)$
\end{corollary} 

\begin{proof}
According to \eqref{eq:proflac}  we get 
\[
Y_{\lambda C}(v)=\la^{n+1}Y_{C}\bigg(\frac{v}{\la^{n+1}}\bigg).
\] 
Then, as $Y_C$ is concave with $Y_C(0)=0$ and $\la<1$, the proof follows.
\end{proof}

\chapter{Sharp isoperimetric inequalities and isoperimetric rigidity}
\label{sec:rigid}

In section \ref{sec:unbounded} isoperimetric inequalities like $I_C\le I_{C_{\min}}$, where $C_{\min}$ is the tangent cone with the smallest aperture in $\{C\}\cup\mathcal{K}(C)$, and its immediate consequence $I_C\le I_H$, where $H$ is a closed half-space, have been mentioned (see Proposition \ref{prp:ICleICp} and Remark \ref{rem:half-plane}). In this section, we shall prove some new isoperimetric inequalities, like $I_C\ge I_{C_\infty}$(see Theorem~\ref{thm:optnondeg}) for a convex body $C$ with non-degenerate asymptotic cone $C_\infty$, recall known ones, and prove rigidity results for the equality cases.

\section{Convex bodies with non-degenerate asymptotic cone}

We begin by studying the relative isoperimetric problem in convex bodies with non-degenerate asymptotic cone (see Chapter \ref{sec:preliminaries} for the corresponding definition).  We shall first need some notation.

We denote by $\mathcal{C}_0^{n+1}$ the set of convex bodies in $\rr^{n+1}$ that contain the origin and have non-degenerate asymptotic cone. Of course, convex bodies in $\mathcal{C}_0^{n+1}$ are unbounded. We regard  this space equipped with the topology of local convergence in Hausdorff distance. 
Let $\Gamma_{k}^{n+1}, 0\le k \le n+1$, be the set of convex bodies $C\in\mathcal{C}_0^{n+1}$ so that $C=\tilde{C}\times \rr^{n+1-k}$ up to an isometry, where $\tilde{C}\subset \rr^{k}$ is a line-free convex body, or just the origin in the case $k=0$. Observe that $\Gamma_{0}^{n+1} = \{ \rr^{n+1}\}$ and that $\Gamma_{1}^{n+1} = \{H:\ \text{$H$ is a half-space}\}$. The latter identity is easy to prove since $\tilde{C}$ in the above decomposition is $1$-dimensional and hence a line, a half-line or a segment. However the case of the segment is excluded, as slabs in $\rr^{n+1}$ have degenerate asymptotic cones. We also notice that $\Gamma_{n+1}^{n+1}$ is the set of line-free convex bodies of $\rr^{n+1}$ with non-degenerate asymptotic cone.

Let $C\subset\rr^{n+1}$ be an unbounded convex body. Note that if $C$ contains a line $L$ through the origin, then $C = \tilde C \oplus L$. Therefore a sequence of translations of $C$ by a divergent sequence of points belonging to $L$ converges to $C$, hence $C$ belongs to the space $\mathcal{K}(C)$ of its asymptotic cylinders. This implies that for all $k=0,\dots,n$ and $C\in \Gamma_{k}^{n+1}$ one has $C\in \mathcal{K}(C)$.

Given $C\in \mathcal{C}_{0}^{n+1}$ there exist two orthogonal projections $\pi_{C}$, $\pi^{\perp}_{C}$ such that $V_{C} = \pi_{C}(C)$ is an $(n+1-k)$-dimensional linear space, $\pi_{C}^{\perp}(C)$ is either $\{0\}$ when $k=0$, or a line-free $k$-dimensional convex body contained in the orthogonal complement $V^{\perp}$ of $V_{C}$, and $C = V_{C} \oplus \pi_{C}^{\perp}(C)$. If $C$ is line-free then there holds $V_{C}=\{0\}$, thus $\pi_{C}(x)=0$ and $\pi_{C}^{\perp}(x) = x$ for all $x\in C$.

Now we set for $0\le m\le n+1$
\begin{equation}
\label{eq:bonita}
\mathcal{C}_{0,m}^{n+1}=\bigcup_{k=0}^{m}\Gamma^{n+1}_{k}
\end{equation}
and note that $\mathcal{C}_{0,n+1}^{n+1} = \mathcal{C}_{0}^{n+1}$. The family $\mathcal{C}_{0,m}^{n+1}$ consists of those sets in $C_0^{n+1}$ possessing a Euclidean factor of dimension larger than or equal to $(n+1-m)$. We shall prove in Theorem~\ref{thm:optnondeg} that the isoperimetric profile $I_C$ of a convex body $C$ with non-degenerate asymptotic cone $C_\infty$ is bounded below by $I_{C_\infty}$. Moreover, in Theorem \ref{thm:novo} we show a rigidity result for the equality case in $I_{C}\ge I_{C_{\infty}}$, which states that if $I_{C}(v_{0}) = I_{C_{\infty}}(v_{0})$ for some $v_{0}>0$, then $C$ and $C_{\infty}$ are isometric. We also prove that $I_{C}(v)$ and $I_{C_{\infty}}(v)$ are asymptotic for $v\to+\infty$, and that isoperimetric regions exist for large volumes. First we need the following two lemmas.

\begin{lemma}
\label{lem:r1}
\mbox{}
\begin{enum}
\item
The set $\mathcal{C}_{0,m}^{n+1}$ is closed under local convergence in Hausdorff distance for every $0\le m \le n+1$.
\item
For every $k=0,\dots,n+1$, if $C\in \Gamma^{n+1}_{k}$ then $C_\infty\in \Gamma^{n+1}_{k}$, where $C_\infty$ is the asymptotic cone of $C$.
\item
Let $C\in\Gamma^{n+1}_{m}$ for some $m\in \{1,\dots,n+1\}$, and let $x_i\in C$, $i\in\nn$. Suppose that $-x_i+C\to K$ locally in Hausdorff distance. If $\pi_{C}^{\perp}(x_i)\to x\in C$ as $i\to\infty$ then $K=-x+C$. If $\pi_{C}^{\perp}(x_i)$ diverges, then $K\in\mathcal{C}^{n+1}_{0,m-1}$.
\end{enum} 
\end{lemma}

\begin{proof}
To show (i), let $K_i\in \mathcal{C}^{n+1}_{0,m}$ and assume that $K_i\to K$ locally in Hausdorff distance. Since the family of linear spaces $V_{i}=V_{K_{i}}$ is relatively compact in the local Hausdorff topology, and since $\dim(V_{i})\ge n+1-m$ for all $i$, we conclude that $K$ must contain a local Hausdorff limit of (a subsequence of) $\{V_{i}\}_{i\in \nn}$, which is necessarily a linear space of dimension at least $n+1-m$. This shows that $K\in \mathcal{C}^{n+1}_{0,m}$, as wanted.

We now prove (ii). Up to a rigid motion of $\rr^{n+1}$, we may assume that $C=\tilde{C}\times \rr^{n+1-k}$, where $\tilde{C}$ contains no lines. It is easy to check that $C_\infty=\tilde{C}_\infty\times \rr^{n+1-k}$, where $\tilde{C}_\infty$ is the asymptotic cone of $\tilde{C}$. Since $\tilde{C}$ is line-free and $\tilde{C}_\infty\subset \tilde{C}$, then $\tilde{C}_\infty$ is line-free. Thus $C_\infty\in \Gamma^{n+1}_{k}$.

Finally we prove (iii).  Since $-\pi_C(x_i) + C=C$ for all $i\in\nn$, we get that $-\pi_C^{\perp}(x_i) + C\to K$. Thus, if $\{\pi_C^{\perp}(x_i)\}_{i\in\nn}$ converges to $x\in C$, then $K=-x+C$. Assume now that  $\{\pi_C^{\perp}(x_i)\}_{i\in\nn}$ diverges and that $\pi_C^{\perp}(x_i)/|\pi_C^{\perp}(x_i)|$ converges to some vector $v\in\esf^n$. Since $\{\pi_C^{\perp}(x_i)\}_{i\in\nn}\subset C \cap V_{C}^{\perp}$,  the argument in the proof of Lemma \ref{lem:cyl} yields that $\{tv:t\in\rr\}\subset K\cap V_{C}^{\perp}$. Since $C\in\Gamma^{n+1}_m $ we get $-\pi_C^{\perp}(x_i)+ C \in \Gamma^{n+1}_m $. Thus by (ii) we get $K\in \mathcal{C}^{n+1}_{0,m}$. Since $K\cap V_{C}^{\perp}$ contains a line we finally get $K\in \mathcal{C}^{n+1}_{0,m-1}$.
\end{proof}

\begin{lemma}
\label{lem:r2}
Let $C\in\Gamma_m^{n+1}$, $m\in\{1,\ldots,n+1\}$, with asymptotic cone $C_\infty$. Let $\{\la_i\}_{i\in\nn}$ be a sequence of positive numbers such that $\lambda_i\downarrow 0$ and take $y_i\in\la_i C$ for all $i\in\nn$.
\mbox{}
\begin{enum}
\item If $-y_i+\la_i C\to M$ locally in Hausdorff distance, then $C_\infty\subset M _\infty$ and, furthermore, $-\pi_C^\perp(y_i)+\la_i C\to M$ locally in Hausdorff distance. In case $\{\pi_C^\perp(y_i)\}_{i\in\nn}$ subconverges to $y\in C_\infty$, then $M=-y+C_\infty$. If $\{\pi_C^\perp(y_i)\}_{i\in\nn}$ diverges, then $M\in\mathcal{C}_{0,m-1}^{n+1}$ and one has the strict inequality
\begin{equation}
\label{eq:r21}
I_{M _\infty}(v) >I_{C _\infty}(v), \quad\text{ for all } v>0. 
\end{equation}
\item
Assume that $K_i\in\mathcal{K}(\la_i C)$, that $K_i\in\mathcal{C}_{0,m-1}^{n+1}$, and that
$K_i \to K$ locally in Hausdorff distance. Then $K\in\mathcal{C}_{0,m-1}^{n+1}$, $C_\infty\subset K_\infty$, and
\begin{equation}
\label{eq:r22}
I_{K _\infty}(v) >I_{C _\infty}(v),\quad\text{ for all } v>0. 
\end{equation}
\end{enum} 
\end{lemma}

\begin{proof}
Let  $C=\tilde{C}\times V$, where $V\subset\rr^{n+1}$ is a linear subspace of dimension $(n+1-m)$ and $\tilde{C}\subset V^\perp$ is a line-free convex set.

We first  prove (i). Since $C_\infty$ is the asymptotic cone of $\la_i C $ for all $i$, and $y_i+C_\infty\subset \la_i C$ for all $i\in\nn$, we have $C_\infty\subset -y_i+ \la_i C$, and so $C_\infty\subset M $. As $M_\infty$ is the largest convex cone with vertex $0$ included in $M$, we have $C_\infty\subset M_\infty$.

Setting $y_i=\pi_C(y_i)+\pi_C^\perp(y_i)$, and noticing that $-\pi_C(y_i)+\la_i C=\la_i C$ for all $i\in\nn$, we have that $-\pi_C^\perp(y_i)+\la_iC$ converges to $M$ locally in Hausdorff distance, as $i\to\infty$. If $\pi_C^\perp(y_i)$ subconverges to $y$ then $M=-y+C$ and, as $\pi_C^\perp(y_i)\in\la_i C\cap V^\perp$, we have $y\in C_\infty\cap  V^{\perp}\subset C_\infty$. Assume now that $\pi_C^\perp(y_i)$ diverges. Eventually passing to a subsequence, we can assume that $\frac{\pi_C^\perp(y_i)}{|\pi_C^\perp(y_i)|}\to v$. By Lemma~\ref{lem:gencyl} the line $L(v)$ is contained in $M\cap V^\perp$. Since $C\in\Gamma_m^{n+1}$ then $-\pi_C^\perp(y_i)+\la_i C\in\Gamma_m^{n+1}$. Thus by Lemma~\ref{lem:r1}(i) we get $M\in\mathcal{C}_{0,m}^{n+1}$ and, since $M\cap V^{\perp}$ contains a line, we obtain that $M\in\mathcal{C}_{0,m-1}^{n+1}$. Lemma \ref{lem:r1}(ii) then implies that $M_\infty\in\mathcal{C}_{0,m-1}^{n+1}$. Since $C\in\Gamma_{m}^{n+1}$ by hypothesis, Lemma~\ref{lem:r1}(ii) implies that $C_\infty\in\Gamma_m^{n+1}$. Hence the inclusion $C_\infty\subset M _\infty$ is strict, consequently $\alpha(M)>\alpha(C_\infty)$. Thus by \eqref{eq:isopsolang} we get $I_{M_\infty}(v)>I_{C_\infty}(v)$ for every $v>0$.

We now prove (ii).  Note that each $K_i$ is a local limit in Hausdorff distance of translations of $\la_i C$ and  $C_\infty\subset\la_i C$  for all $i\in\nn$. Then  $C_\infty\subset K_i$ for all $i\in\nn$. Since $K_i \to K$ locally in Hausdorff distance, we get  that $C_\infty\subset K$. As $K _\infty$ is the largest convex cone included in $K$ we have $C_\infty\subset K_\infty$. Since $K_i\in\mathcal{C}_{0,m-1}^{n+1}$ and $K_i \to K$ locally in Hausdorff distance, Lemma~\ref{lem:r1}(i) implies that $K\in\mathcal{C}_{0,m-1}^{n+1}$ and, by Lemma~\ref{lem:r1}(ii), we have $K\in\mathcal{C}_{0,m-1}^{n+1}$. Arguing as in the proof of (i), we get $C_\infty\in\Gamma_m^{n+1}$ and thus obtain the strict inclusion $C_\infty\subset K _\infty$, which gives \eqref{eq:r22} at once.
\end{proof}

\begin{theorem}
\label{thm:optnondeg}
Let $C$ be a convex body with non-degenerate asymptotic cone $C_\infty$. Then
\begin{equation}
\label{eq:icgeicinfty}
I_C(v)\ge I_{C_{\infty}}(v)\qquad \text{for all }v>0
\end{equation}
and
\begin{equation}
\label{eq:optimanondegen}
\lim_{v\to\infty}\frac{I_C(v)}{I_{C_{\infty}}(v)}=1.
\end{equation}
Moreover, isoperimetric regions exist in $C$ for sufficiently large volumes, and any sequence of isoperimetric regions with volumes tending to infinity converges up to a rescaling to a geodesic ball centered at a vertex in the asymptotic cone $C_{\infty}$. 
\end{theorem}

\begin{proof}
\textit{Step one}. We first show some inequalities involving the isoperimetric profiles of $C$, $C_{\infty}$, and of rescalings of $C$, as well as a uniform bound on the diameter of generalized isoperimetric regions.
We fix $v>0$ and consider a sequence $v_i\uparrow\infty$. Then we define the sequence $\la_i\downarrow 0$ by $\la_i^{n+1}v_i=v$. 
Since $\la_i C\to C_\infty$ locally in Hausdorff distance, Proposition \ref{prp:niceapproxI} implies
\begin{equation}
\label{eq:upnden}
\limsup_{i\to\infty} I_{\la_i C}(v)\le I_{C_\infty}(v)
\end{equation}
for any $v>0$. Owing to Corollary \ref{cor:novo}, inequalities $I_{\la_i C}\le I_C$ hold for large $i$, and taking limits we obtain
\begin{equation}
\label{eq:lC}
\limsup_{i\to\infty}I_{\la_i C}(v)\le I_C(v)
\end{equation}
for all $v>0$.

Let $E_i$ be generalized isoperimetric regions in $\la_i C$ of volume $v$.  Recall that the sets $E_i$ are connected by Theorem~\ref{thm:conY}.  We shall prove that $E_i$ have uniformly bounded diameter.  Fix $r_0>0$ and set
\begin{equation*}
\mathcal{L}=\bigcup_{0<\la\le1}\{\la C\} \cup \mathcal{K}(\la C).
\end{equation*}
By Lemma \ref{lem:one-diam} and Theorem \ref{thm:conY}, it suffices to show that 
\begin{equation}
\label{eq:br0}
\inf_{L\in\mathcal{L}, x\in L}|B_L(x,r_0)|>0.
\end{equation}
Fix $\la>0$ and let $L\in \{\la C\} \cup \mathcal{K}(\la C)$. Recall that $C_\infty$ is the asymptotic cone of $\la C$, and so $C_\infty\subset \la C$. The set $L$ is a local Hausdorff limit of translations of $\la C$. Hence there is a (possibly diverging) sequence $\{z_i\}_{i\in\nn}\subset \la C$ so that $-z_i+\la C\to L$ locally in Hausdorff distance. Since $C_\infty\subset-z_i+ \la_i C$ we get $C_\infty\subset L$. Now let $x\in L$.  As $x+C_\infty\subset L$, we get $\clb_{x+C_\infty}(x,r) \subset \clb_L(x,r)$. Since $C_{\infty}$ is non-degenerate,  we can  pick $\de>0$ and $y\in C_\infty$ so that $B(y,\delta)\subset \clb_{C_\infty}(0,r_0)$. Hence $B(x+y,\delta)\subset \clb_{x+C_\infty}(x,r_0)$, which gives \eqref{eq:br0} at once.

\textit{Step two.} We prove \eqref{eq:icgeicinfty} by an induction argument. Note that \eqref{eq:icgeicinfty} holds trivially in $\mathcal{C}_{0,1}^{n+1}$. Fix $2\le m\le n+1$. Assuming the validity of \eqref{eq:icgeicinfty} for every set in $\mathcal{C}_{0,m-1}^{n+1}$ we shall prove that it holds for any set in $\Gamma_{m}^{n+1}$. Then the proof would follow from \eqref{eq:bonita}. So we assume that
\begin{equation}
\label{eq:induc}
I_{K}\ge I_{K_\infty}\quad\text{for every }K\in \mathcal{C}_{0,m-1}^{n+1}.
\end{equation}
Take $C\in\Gamma_m^{n+1}$ and consider a sequence of generalized isoperimetric regions $E_i$ in $\la_iC$ of volume $v>0$, as in Step one. The sets $E_i$ are connected with uniformly bounded diameter. We shall distinguish three cases.

Case 1: Assume that $E_i\subset\la_i C$ and that $\{E_i\}_{i\in\nn}$ is uniformly bounded.  By Lemma \ref{lem:antier}, there exists a finite perimeter set $E\subset C_\infty$ of volume $v$ so that
\begin{equation}
\label{eq:liminfc1}
 I_{C_\infty}(v)\le\pp_{C_\infty}(E)\le\liminf_{i\to\infty}\pp_{\la_iC}(E_i)=\liminf_{i\to\infty}I_{\la_i C}(v).
\end{equation}
Owing to \eqref{eq:lC} and \eqref{eq:liminfc1} we get $I_{C_\infty}(v)\le I_C(v)$ and this concludes the proof of \eqref{eq:icgeicinfty} in this case.

Case 2: Assume that $E_i\subset\la_iC$ and that $\{E_i\}_{i\in\nn}$ diverges. Let $x_i\in E_i$. If $\pi_C^\perp(x_i)$ converges, then  we get that $-x_i+\la_i C\to -x+C_\infty$ for some $x\in C_\infty$ and this sub-case is reduced to the previous one. So we assume that $\pi_C^\perp(x_i)$ diverges. Passing to a subsequence we may assume that $-x_i+\la_iC\to M$ locally in Hausdorff distance. By Lemma~\ref{lem:r2}(ii) we have $M\in\mathcal{C}_{0,m-1}^{n+1}$. Then
\eqref{eq:induc}  implies
\begin{equation}
\label{eq:c21}
I_M\ge I_{M_\infty}.
\end{equation}
Now by Lemma \ref{lem:antier} there exists a finite perimeter set $E\subset M$ so that
\begin{equation*}
\begin{split}
\label{eq:liminfc2}
 I_M(v)\le\pp_M(E)&\le\liminf_{i\to\infty}\pp_{(-x_i+\la_iC)}(-x_i+E_i)
 \\
 &=\liminf_{i\to\infty}\pp_{\la_iC}(E_i)=\liminf_{i\to\infty}I_{\la_i C}(v).
\end{split}
\end{equation*}
Then \eqref{eq:c21}, \eqref{eq:liminfc2}  and \eqref{eq:r21}  imply                    
\begin{equation*}
  I_{C_\infty}(v)<\liminf_{i\to\infty}I_{\la_i C}(v)
\end{equation*}
which gives a contradiction with \eqref{eq:upnden}. So Case 2 cannot hold.

Case 3. We assume that $E_i\subset K_i$, where $K_i$ is an asymptotic cylinder of $\la_i C$. From Lemma~\ref{lem:r1}(iii) we have the following alternative for each $i$: either $K_i$ is a translation of $\la_i C$, or $K_i\in\mathcal{C}_{0,m-1}^{n+1}$.  If the first possibility holds for infinite $i$, we can treat this case as in the previous ones. If the second holds for infinite $i$, passing to a subsequence we may assume that $K$ is the local limit in Hausdorff distance of the sequence $K_i$. Then Lemma~\ref{lem:r1}(ii) implies that $K\in\mathcal{C}_{0,m-1}^{n+1}$ and \eqref{eq:induc} implies
\begin{equation}
\label{eq:c222}
I_{K_\infty}\le I_{K}.
\end{equation}
Since the sets $\{E_i\}_{i\in\nn}$ have uniformly bounded diameter, Lemma \ref{lem:antier} implies the existence of a finite perimeter set $E\subset K$ with volume $v$ so that
\begin{equation}
\label{eq:liminfc3}
I_K(v)\le \pp_K(E)\le \liminf_{i\to\infty} \pp_{K_i}(E_i)=\liminf_{i\to\infty} I_{K_i}(v)= \liminf_{i\to\infty} I_{\la_i C}(v).
\end{equation}
Then,  \eqref{eq:c222}, \eqref{eq:liminfc3} and  \eqref{eq:r22} imply %
\begin{equation*}
  I_{C_\infty}(v)<\liminf_{i\to\infty}I_{\la_i C}(v),
\end{equation*}
which gives a contradiction with \eqref{eq:upnden} showing that Case 3 cannot hold.

We now prove \eqref{eq:optimanondegen}. Since $C_{\infty}$ is the asymptotic cone of each $\la_i C$ then \eqref{eq:icgeicinfty} holds for every  $\la_i C, i\in\nn$. Taking limits we conclude
\[
I_{C_\infty}(v)\le\liminf_{i\to\infty}I_{\la_i C}(v).
\]
Thus, by \eqref{eq:upnden}, we get
\begin{equation}
\label{eq:optimanondegen7}
I_{C_\infty}(v)=\lim_{i\to\infty}I_{\la_i C}(v).
\end{equation}
From \eqref{eq:optimanondegen7}, Lemma \ref{lem:link I laC I C} and the fact that $C_{\infty}$ is a cone we deduce
\[
1=\lim_{\la\to 0}\frac{I_{\la C}(1)}{I_{C_{\infty}}(1)}=
\lim_{{\la}\to
0}\frac{{\la}^nI_C(1/{\la}^{n+1})}{{\la}^n I_{C_{\infty}}(1/{\la}^{n+1})}=\lim_{v\to\infty}\frac{I_C(v)}{I_{C_{\infty}}(v)},
\]
which shows \eqref{eq:optimanondegen}.

\textit{Step three.} We prove the existence of isoperimetric regions for large volumes. We argue by contradiction assuming that there exists a sequence $v_i\uparrow \infty$ such that no generalized isoperimetric region of volume $v_i$ is realized in $C$. If $v>0$ is fixed and $\la_i\downarrow 0$ is as in Step one, then no generalized isoperimetric regions for volume $v$ are realized in $\la_iC$. Arguing exactly as in Case 3 of Step two, we get a contradiction.

\textit{Step four.} We show the last part of the statement, i.e., that suitable rescalings of sequences of isoperimetric regions with diverging volumes must converge to isoperimetric regions in $C_{\infty}$. To this end we argue again by contradiction. We assume there exists $v>0$ so that a sequence $E_i\subset \la_i C$ of  isoperimetric regions of volume  $v$  diverges. Then arguing  exactly as in Case 2 of Step two we get a contradiction. As a consequence, only Case 1 in Step two holds. Then the sets $E_i$ are uniformly bounded and thus every subsequence of $E_i$ converges to some $E$, where $E\subset C_\infty$ is an isoperimetric region by \eqref{eq:upnden} and \eqref{eq:lC}, hence it must be a geodesic ball centered at some vertex of $C_\infty$, owing to the results in \cite{FI}.%
\end{proof}

From \eqref{eq:optimanondegen} and \eqref{eq:isopsolang} we easily get

\begin{corollary}
\label{cor:nondegen}
Let $C, C'\subset \rr^{n+1}$ be unbounded convex bodies with non-degenerate asymptotic cone  satisfying $\alpha(C_\infty)>\alpha({C'}_\infty)$. Then for $v>0$ sufficiently large we have $I_C(v)>I_{C'}(v)$.
\end{corollary} 

\begin{remark}
\label{rem:nondeg}
Let $C\subset \rr^{n+1}$ be an unbounded convex body. By \eqref{eq:icgeicinfty},  \eqref{eq:solangle}, and the fact that $C_\infty$ is the largest cone included in $C$, we get that if $C$ contains a solid convex cone $K$, then $I_C\ge I_K$.
\end{remark}

\begin{remark}
\label{rem:optnondeg}
Let $C\subset \rr^{n+1}$ be an unbounded convex body with non-degenerate asymptotic cone $C_\infty$. Theorem~\ref{thm:optnondeg} and  \eqref{eq:isopsolang} imply that the isoperimetric dimension of $C$ is $n$. Furthermore 
\[
I_{C_\infty}(1)=\sup\{a>0:\ I_C(v) \ge av^{n/(n+1)} \text{ for every}\,\, v>0 \}
\]
or, equivalently,
\[
I_{C_\infty}(1)=\inf_{v>0}\frac{I_C(v)}{v^{n/(n+1)}}.
\]
\end{remark}

\section{The isoperimetric profile for small volumes}

Theorem \ref{thm:limI} below states that the isoperimetric profile of an  unbounded convex body of uniform geometry is asymptotic to $I_{C_{\min}}$ for small volumes. In the case of a bounded convex body this result was stated and proved in \cite[Thm.~6.6]{MR3335407}. To accomplish this we shall need the following Lemma.
\begin{lemma}
\label{lem:C0<K}
Let $\{L_i\}_{i\in\nn}$ be a sequence of convex bodies converging locally in Hausdorff distance to a  convex body $L$. Assume that $0\in L_i$ for all $i\in\nn$. Let $\la_i$ be a sequence of positive real numbers converging to $+\infty$ and assume that $\la_iL_i$ converges locally in Hausdorff distance to an unbounded convex body $M$. Then $L_0\subset M$.
\end{lemma}

\begin{proof}
As $L_0=\cl{\cup_{\la > 0} \la L}$ and $M$ is a closed set, it is enough to prove that $ \cup_{\la > 0} \la L\subset M$. Take a point $z\in \cup_{\la > 0} \la L$, and also $\la>0,\, z'\in L$ so that  $z=\la z'$. Since $L_i\to L$ locally in Hausdorff distance, there exists a sequence $z_i\in L_i$ so that $z_i\to z'$. Consequently
\begin{equation}
\label{eq:C0<K}
z=\la z'=\lim_{i\to\infty}\la z_i =\lim_{i\to\infty}\la_i \big(\frac{\la}{\la_i} z_i\big). 
\end{equation}
Since $\la_i$ is a diverging sequence, inequality $\frac{\la}{\la_i}<1$ holds for $i$ large enough. Since $z_i\in L_i$ and the sets $L_i$ are convex and contain the origin, we have $\frac{\la}{\la_i} z_i\in L_i$. By \eqref{eq:C0<K} and the local convergence in Hausdorff distance of $\la_iL_i$ to $M$ we conclude that $z\in M$.
\end{proof}

\begin{example}
In general the set $M$ is not a cone and can be different from $L_0$. Take $L:=[0,1]^3\subset\rr^3$, $x_i:=(i^{-1},0,0)$ and $L_i:=-x_i+L$. Then $0\in L_i$ for all $i\in\nn$ and $L_i\to L$ in Hausdorff distance. However, if we take $\la_i:=i$, then $\la_iL_i$ converges locally in Hausdorff distance to the set $[-1,+\infty)\times [0,+\infty)\times [0,+\infty)$, different from $L_0=[0,+\infty)^3$.
\end{example}

\begin{theorem}
\label{thm:limI}
Let $C$ be a convex body (if unbounded, we further assume that it is of uniform geometry). Then $I_{C_{\min}}(v)>0$ for all $v\in (0,|C|)$ and 
\begin{equation}\label{IcIcminvzero}
\lim_{v\to 0}\frac{I_{C}(v)}{I_{C_{\min}}(v)}=1.
\end{equation}
Moreover, any sequence of generalized isoperimetric regions with volumes tending to zero subconverges to a point either in $C$ or in some $K\in \mathcal{K}(C)$, where the minimum of the solid angle function is attained.
\end{theorem}

\begin{proof}
First observe that $I_{C_{\min}}(v)>0$ for all $v\in (0,|C|)$. This is trivial in the bounded case and, in the unbounded uniform geometry case, it follows from Proposition~\ref{prop:maincond} and \eqref{eq:isopsolang}.

Let $v_i\downarrow 0$ and $E_i$ be generalized isoperimetric regions of volumes $v_i$ in $C$, for $i\in\nn$.
 Let $\la_i\uparrow\infty$ so that, for all $i\in\nn$, $\vol{\la_i E_i}=1$.  %
Then $\la_iE_i$ are generalized isoperimetric regions of volume $1$ in $\la_iC$.  Recall that the sets $E_i$ are connected by Theorem \ref{thm:conY}.

First we prove that the sets $\la_i E_i$ have uniformly bounded diameter.  Fix $r_0>$ and set
\begin{equation*}
\mathcal{M}=\bigcup_{\la\ge 1}\{\la C\} \cup \mathcal{K}(\la C).
\end{equation*}
By Lemma \ref{lem:one-diam} and Theorem \ref{thm:conY}, it suffices to prove that 
\begin{equation}
\label{eq:bra}
\inf_{M\in\mathcal{M}, x\in M}|B_M(x,r_0)|>0.
\end{equation}
Fix  $\la>1$ and let $M\in \{\la C\} \cup \mathcal{K}(\la C)$. Since $C$ is of bounded geometry, Proposition \ref{prop:maincond} yields 
\begin{equation*}
b(r_0)=\inf_{x\in C}|B_C(x,r_0)|>0.
\end{equation*}
If $K\in\mathcal{K}( C)$  then it is a local limit in Hausdorff distance of translations of $ C$. Hence 
\begin{equation}
\label{eq:brc}
\inf_{x\in K}|B_K(x,r_0)|\ge b(r_0).
\end{equation}
If $M\in\mathcal{K}(\la C)$, with $\la>1$, then $M=\la K$ for some $K\in\mathcal{K}( C)$.  Since $M$ is convex then $h_{x,\la^{-1}}(M)\subset M$, for every $x\in M$.  Consequently 
\begin{equation}
\label{eq:brcb}
|B_M(x,r_0)|\ge |B(x,r_0)\cap h_{x,\la^{-1}}(M)|
\end{equation}
As  $h_{x,\la^{-1}}(M)$ is isometric to $K$ then \eqref{eq:brc} and \eqref{eq:brcb}
imply
\begin{equation*}
\inf_{x\in M}|B_M(x,r_0)|\ge b(r_0).
\end{equation*}
This concludes the proof of $\eqref{eq:bra}$ 

Since $\{\la_i E_i\}_{i\in\nn}$ has uniformly bounded diameter we shall distinguish two cases.

Case 1. Assume that $E_i$ is contained in $K_i\in \mathcal{K}(C)$ for infinitely many indices $i$. Possibly passing to a subsequence we may assume that $K_i\to K$  and $\la_i K_i\to K'$ in local Hausdorff distance. Applying Lemma \ref{lem:ciave}(ii) for the particular case $C_i=C$ we get that $ K\in \mathcal{K}(C)$. %
By Lemma~\ref{lem:zK} we may assume that $0\in E_i$ for all $i$. As $\{\diam{\la_i E_i}\}_{i\in\nn}$ is uniformly bounded, Lemma \ref{lem:antier} yields  a finite perimeter set $E\subset K'$, with $\vol{E}=1$,  such that
\begin{equation}
\label{eq:clca13a}
I_{K'}(1)\le\pp_{K'}(E)\le\liminf_{i\to\infty}\pp_{\la_i K_i}(\la_i E_i).
\end{equation}
Now by Lemma \ref{lem:C0<K} and Remark \ref{rem:nondeg}  we get 
\begin{equation}
\label{eq:clca13b}
I_{K_0}\le I_{K'},
\end{equation}
where $K_0$ is the tangent cone of $K$ at $0$. Since $\la_i E_i\subset \la_i K_i$ are isoperimetric regions of volume 1, we get by  \eqref{eq:clca13a} and \eqref{eq:clca13b},
\begin{equation}
\label{eq:clca13c}
I_{K_0}(1)\le \liminf_{i\to\infty}I_{\la_i K_i}(1).
\end{equation}
Owing to \eqref{eq:clca13c}, \eqref{eq:isopsolang}, Lemma \ref{lem:link I laC I C}, the fact that $I_{K_i}(v_i)=I_C(v_i)$, and that $\la K_{0}=K_0$ we obtain
\begin{align*}
\liminf_{i\to\infty}\frac{I_C(v_i)}{ I_{ K_0}(v_i)}&=\liminf_{i\to\infty}\frac{{\la}_i^nI_C(1/{\la}_i^{n+1})}{{\la}_i^n I_{K_{0}}(1/{\la}_i^{n+1})}
\\
&=\liminf_{i\to\infty}\frac{I_{\la_i C}(1)}{I_{K_{0}}(1)}=\liminf_{i\to\infty}\frac{I_{\la_i K_i}(1)}{I_{K_{0}}(1)}\ge 1.
\end{align*}
Owing to Remark \ref{rem:ICmin} we have
\begin{equation*}
I_{C_{\min}}\le I_{K_0}.
\end{equation*}
Thus
\begin{equation*}
\limsup_{i\to\infty}\frac{I_C(v_i)}{ I_{ K_0}(v_i)}\le
\limsup_{i\to\infty}\frac{I_C(v_i)}{I_{C_{\min}}(v_i)}\le 
1\le \liminf_{i\to\infty}\frac{I_C(v_i)}{ I_{ K_0}(v_i)},
\end{equation*}
consequently
\begin{equation*}
\lim_{i\to\infty}\frac{I_C(v_i)}{ I_{C_{\min}}(v_i)}=1.
\end{equation*}

Case 2. Assume that $E_i$ is contained in $C$ for infinitely many indices $i$. Let $x_i\in E_i$ be such that $(-x_i+ C)\to K$ locally in Hausdorff distance, up to a subsequence. If $\{x_i\}_{i\in\nn}$ subconverges to $x\in C$, then $K=-x+C$. Otherwise $\{x_i\}_{i\in\nn}$ is unbounded and, by the definition of asymptotic cylinder, we get that $ K\in \mathcal{K}(C)$. Possibly passing to a subsequence, $\la_i(-x_i+ C)\to K'$ locally in Hausdorff distance. Now by Lemma \ref{lem:C0<K} and  Remark \ref{rem:nondeg} we get 
\begin{equation*}
I_{K_0}\le I_{K'}.
\end{equation*}
 Arguing as in the previous case we get a finite perimeter set $E\subset K'$ with $\vol{E}=1$, such that 
\begin{equation*}
\begin{split}
\label{eq:clca12a}
I_{K_0}(1)&\le I_{K'}(1)\le\pp_{K'}(E)
\\
&\le \lim_{i\to\infty}\pp_{\la_i(-x_i+ C)}(\la_i(-x_i+  E_i))=\lim_{i\to\infty}I_{\la_i C}(1).
\end{split}
\end{equation*}
Now we continue as in the final part of the proof of step one to conclude the proof of \eqref{IcIcminvzero}. The proof of the last part of the statement is a direct consequence of the previous arguments and from the fact that that $\diam(E_i)\to 0$ since $\{\diam(\la_i E_i)\}_{i\in\nn}$ is bounded. 
\end{proof}

From the Theorem~\ref{thm:limI} and \eqref{eq:isopsolang} we easily get

\begin{corollary}
\label{cor:smalvo}
Let $C, C'\subset \rr^{n+1}$ be unbounded convex bodies  satisfying
\[
\alpha(C_{\min})<\alpha(C'_{\min}). 
\]
Then for sufficient small volumes we have 
\[
I_C<I_{C'}.
\]
\end{corollary} 

Note that in a polytope or a prism, i.e the product of a polytope with a Euclidean space, isoperimetric regions exist, for all volumes, by compactness and Corollary \ref{cor:antepol} respectively. %
\begin{corollary}[{\cite[Theorem 6.8]{MR3335407}, \cite[Theorem 3.8]{MR3385175}}]
\label{cor:pol}
Let $C\subset\rr^{n+1}$ be a polytope or a prism. Then for sufficient small volumes isoperimetric regions are geodesic balls centered at vertices of the tangent cone with minimum solid angle.
\end{corollary} 
\begin{proof}
According to Theorem \ref{thm:limI}, a sequence $\{E_i\}_{i\in\nn}$ of isoperimetric regions of volumes going to zero collapses to $p$, where $C_p$ attains the minimum of the solid angle function. Since $C$ is a polytope or a prism  then, for sufficient  large $i\in\nn$, $E_i\subset C_p$. Then the proof follows by the fact that the only isoperimetric regions in $C_p$ are geodesic balls centered at $p$, see \cite{FI}.
 \end{proof}

\begin{corollary}[{\cite{pedri}}]
\label{cor:slab}
Let $C\subset\rr^{n+1}$ be a slab. Then for sufficiently small volumes isoperimetric regions are half-balls.
\end{corollary} 

\begin{proof}
Since $C$ is the product of a segment with a Euclidean space then Corollary \ref{cor:antepol} implies that isoperimetric regions exist for all volumes. Since $C$ is a slab then all points on the boundary of $C$ attain the minimum of the solid angle function. As shown in the proof of Theorem \ref{thm:limI}, the diameter of a sequence $\{E_i\}_{i\in\nn} $ of isoperimetric regions of volumes going to zero, also goes to zero. Consequently,  for sufficient  large $i\in\nn$, $E_i$ belong to a half-space and since in a half-space the only isoperimetric regions are half-balls, the proof follows.
 \end{proof}

\begin{remark}
\label{rem:slab}
In \cite{pedri} Pedrosa and Ritoré completely solved the isoperimetric problem in a slab of $\rr^{n+1}$ by means of Alexandrov reflection and the characterization of stable free boundary hypersurfaces of revolution connecting two parallel hyperplanes. They showed that up to dimension $n+1=8$ the only isoperimetric regions are half-balls and tubes. The case $n+1=9$ is still undecided, while for $n+1\ge 10$ isoperimetric regions of undouloid type may appear. See the remarks after Proposition~5.3 in \cite{pedri}.
\end{remark} 

\section{Isoperimetric rigidity}
\label{sec:isoprigid}

We consider the following, general question: assuming that a relative isoperimetric inequality holds for a convex body $C$, does equality for some prescribed volume imply some geometric characterization of $C$? Whenever this happens, we will say that the isoperimetric inequality is \emph{rigid}. In the following we will provide answers to this question in various cases of interest, see Theorems \ref{thm:novo}, \ref{thm:rigid1-a}, and \ref{thm:rigid1-c}, and Corollary~\ref{cor:rigid1-b}. A key tool for proving these rigidity results is Theorem \ref{thm:limI}. The first rigidity result we present is Theorem \ref{thm:novo}, which can be seen as a refinement of Theorem \ref{thm:optnondeg}.

\begin{theorem}
\label{thm:novo}
Let $C$ be a convex body  with non-degenerate asymptotic cone $C_\infty$. If equality holds in the isoperimetric inequality \eqref{eq:icgeicinfty} for some volume $v_{0}>0$, that is $I_{C_{\infty}}(v_{0})=I_{C}(v_{0}) $, then $C$ is isometric to $C_{\infty}$.
\end{theorem}

\begin{proof}
Assume that $I_C(v_0)=I_{C_{\infty}}(v_0)$ for some $v_0>0$. Since $Y_C$ is concave by Theorem \ref{thm:conY}, $Y_{C_{\infty}}$ is linear by \eqref{eq:isopsolang}, and $Y_C\ge Y_{C_{\infty}}$ by \eqref{eq:icgeicinfty}, the function  $Y_C- Y_{C_{\infty}}:\rr^+\to\rr^+$  is concave and non-negative, thus non-decreasing. Hence $Y_C(v_0)=Y_{C_{\infty}}(v_0)$ implies $Y_C(v)-Y_{C_\infty}(v)=0$ for all $v\le v_0$. Thus
\begin{equation}
\label{Iciinf}
I_C(v)=I_{C_{\infty}}(v),\qquad\text{for all}\ v\le v_0.
\end{equation}
By Lemma \ref{lem:losemcontangcon}  there exist $K\in\{C\}\cup\mathcal{K}(C)$ and $p\in K$ so that $I_{K_p}=I_{C_{\min}}$. 
Assume first that $K$ is an asymptotic cylinder of $C$ not isometric to $C$. As $ K_\infty\subset K\subset K_p$  we obtain $I_{K_\infty}\le I_{K_p}=I_{C_{\min}}$ by \eqref{eq:isopsolang}. By Lemma~\ref{lem:r2}(ii) we get $I_{C_\infty}<I_{K_\infty}$.
Combining the two last inequalities we obtain ${I_{C_{\infty}}}<{I_{C_{\min}}}$, and since $I_{C_{\infty}}^{(n+1)/n}$ and $I_{C_{\min}}^{(n+1)/n}$  are linear functions we get%
\begin{equation}
\label{eq:limcincmin}
\lim_{v\to 0}\frac{I_{C_{\infty}}(v)}{I_{C_{\min}}(v)}< 1.
\end{equation}
Now by \eqref{Iciinf}, \eqref{eq:limcincmin} and  Theorem \ref{thm:limI}
\begin{equation*}
1=\lim_{v\to 0}\frac{I_C(v)}{I_{C_{\min}}(v)}=\lim_{v\to 0}\frac{I_{C_{\infty}}(v)}{I_{C_{\min}}(v)}<1,
\end{equation*}
yielding a contradiction. Consequently $K=C$. Then by \eqref{Iciinf}, Theorem \ref{thm:limI}  and  \eqref{eq:isopsolang} we get
\begin{equation*}
1=\lim_{v\to 0}\frac{I_C(v)}{I_{C_{\min}}(v)}=\lim_{v\to 0}\frac{I_{C_{\infty}}(v)}{I_{C_p}(v)}=\frac{\alpha({C_{\infty}})}{\alpha({C_p})}.
\end{equation*}
and since 
\[
p+C_\infty\subset C\subset C_p,
\]
we conclude that $C=p+C_\infty$.
\end{proof}

As a consequence of the previous results, we are able to show the following asymptotic property of isoperimetric regions of small volume.

\begin{corollary}
Let $C\subset\rr^{n+1}$ be a convex body. Take any sequence of isoperimetric regions with volumes converging to zero and rescale the isoperimetric sets to have volume one. Then a subsequence converges to a geodesic ball centered at a vertex in a tangent cone of minimum solid angle.
\end{corollary} 

\begin{proof}
Take a sequence of isoperimetric regions $E_i$ in $C$ with $|E_i|\to 0$, and let $\la_i>0$ such that $|\la_iE_i|=1$. Consider a sequence $x_i\in E_i$ and 
assume that, up to passing to subsequences, both sequences $-x_{i}+C$ and $\la_i(-x_i+C)$ converge in local Hausdorff sense to $K\in \{C\}\cup {\mathcal K}(C)$ and $K'$, respectively. Let $K_0$ be the tangent cone at $0$ of $K$.\begin{equation}
\label{eq:comparison}
I_{K_0}(1)\le I_{K'}(1)=\liminf_{i\to\infty} I_{\la_iC}(1).
\end{equation}
On the other hand, by \eqref{eq:ICleICp} we have 
\[
\frac{I_{\la_iC}(1)}{I_{K_0}(1)}=\frac{I_C(|E_i|)}{I_{K_0}(|E_i|)}\le \frac{I_C(|E_i|)}{I_{C_{\min}}(|E_i|)}\le 1,
\]
so that
\[
\limsup_{i\to\infty} I_{\la_iC}(1)\le I_{C_{\min}}(1)\le I_{K_0}(1).
\]
Comparing this equation with \eqref{eq:comparison} we get
\[
I_{K_0}(1)=I_{C_{\min}}(1)=I_{K'}(1)=\lim_{i\to\infty}I_{\la_iC}(1).
\]
In particular, $K_0$ is a tangent cone of $-x+C$ with the smallest possible solid angle. Let $K'_\infty$ be the asymptotic cone of $K'$. Since it is the largest cone included in $K'$ we get $K_0\subset K'_\infty$ and so $I_{K_0}\le I_{K'_\infty}$. Due to Theorem \ref{thm:optnondeg} there holds $I_{K'}\ge I_{K'_\infty}$, and consequently $I_{K'}(1)= I_{K'_\infty}(1)$. Then Theorem~\ref{thm:novo} implies that $K'=K_0$ and this concludes the proof.
\end{proof}

We now prove three rigidity results describing the equality cases in the isoperimetric inequalities \eqref{eq:3isop}.

\begin{theorem}
\label{thm:rigid1-a}
Let $C\subset \rr^{n+1}$ be a  convex body. Then, for every $v\in (0,|C|)$, $\la<1$ and $w\in [v,|C|)$ we have%
\begin{equation}
\label{eq:3isop}
I_C(v) \le{I_{C_{\min}}(v)},
\quad
I_C(v) \ge I_{ \la C}(v),
\quad
I_C(v) \ge  \frac{I_C(w)}{w^{n/(n+1)}}\,v^{n/(n+1)}.
\end{equation}
If equality holds in any of the inequalities in \eqref{eq:3isop} for some $v_0>0$, then $I_C(v)=I_{C_{\min}}(v)$ for every $v\le v_0$. Moreover, for every $p\in K$, $K\in \{ C\} \cup \mathcal{K}(C)$, where the infimum of the solid angle function is attained, there holds $K_p\cap B(p,r_0)=K\cap B(p,r_0)$, where $r_0$ is defined by $\vol{B_K(p,r_0)}=v_0$ and geodesic  balls $B_K(p,r)$, with $r\le r_0$, realize $I_C$ for $v\le v_0$.
\end{theorem}

\begin{proof}
We start with the first inequality in \eqref{eq:3isop}. By Lemma \ref{lem:losemcontangcon} and \eqref{eq:isopsolang} there exists $K\in\{C\}\cup\mathcal{K}(C)$ and $p\in C$ so that $I_{K_p}=I_{C_{\min}}$. By Propositions~\ref{cor:niceapproxK} and \ref{prp:ICleICp} we have
\begin{equation}
\label{eq:asymI3}
I_C\le I_K\le I_{K_p}= I_{C_{\min}}.
\end{equation}
Assume now there exists $v_0>0$ such that $I_C(v_0)=I_{C_{\min}}(v_0)$. Recall that $Y_C=I_C^{(n+1)/n}$. From \eqref{eq:asymI3} we get $Y_C\le Y_{C_{\min}}$ and, since $Y_C$ is concave by Theorem \ref{thm:conY} and $Y_{C_{\min}}$ is linear by \eqref{eq:isopsolang}, the non-negative function $Y_{C_{\min}}-Y_C$ is convex and so it is non-decreasing. Hence $Y_C(v_0)=Y_{C_{\min}}(v_0)$ implies $Y_C(v)=Y_{C_{\min}}(v)$  for all $v\le v_0$. Consequently $I_C(v)=I_{C_{\min}}(v)$ for all $v\le v_0$.

Choose $r_{0}>0$ such that $\vol{ B(p,r_{0})\cap K}=v_0$ and let $L_p$ be the closed cone centered at $p$ subtended by $\ptl B(p,r_{0})\cap K$. By Proposition \ref{cor:niceapproxK}, the fact that $I_C(v_0)=I_{C_{\min}}(v_0)=I_{K_p}(v_0)$, and the inequality \eqref{eq:ICleICp} in the proof of Proposition~\ref{prp:ICleICp} applied to $K_p$, we get $\alpha(L_p)=\alpha(K_p)$. Since $L_p\subset K_p$, we have
\[
 B(p,r_{0})\cap L_p= B(p,r_{0})\cap K_p,
\]
and since
\[
 B(p,r_{0})\cap L_p\subset  B(p,r_{0})\cap K\subset B(p,r_{0})\cap K_p,
\]
we deduce 
\[
B(p,r_{0})\cap K=B(p,r_{0})\cap K_p.
\]
Moreover, since $I_C(v)=I_{K_p}(v)$ for all $v\le v_0$ then by \eqref{eq:isopsolang} we get 
\[
I_C(\vol{B_{K_p}(p,r)})=\pp_{K_p}(B_{K_p}(p,r)),\quad\text{for every}\quad r \le r_0.
\]
This concludes the proof of the equality case in the first inequality of \eqref{eq:3isop}.

Note that the second inequality in \eqref{eq:3isop} has already been proved in Corollary \ref{cor:novo}. We shall caracterize the equality case. If there exists $v_0>0$ such that $I_C(v_0)=I_{\lambda C}(v_0)$, then  $Y_C(v_0)=Y_{\lambda C}(v_0)$. Hence $Y_C$ is linear for $v\le v_0$. Since $Y_{C_{\min}}$  is linear, then by Theorem \ref{thm:limI} we have $Y_{C_{\min}}=Y_C$ for every $v\le v_0$ and we proceed as above to conclude the proof.

We now prove the third inequality in \eqref{eq:3isop}. Since $Y_C$ is concave we have
\[
Y_C(v) \ge  \frac{Y_C(w)}{w}\,v
\] 
for every $0<v\le w$. Raising to the power $n/(n+1)$ we get the desired inequality. Now if equality holds for some $0<v_0<w$ then $Y_C$ is linear for $0<v< v_0$ %
and we proceed as before to conclude the proof.
\end{proof}

\begin{corollary}
\label{cor:rigid1-b}
Let $C\subset \rr^{n+1}$ be a  convex body. Then, for every $v\ge 0$,
\begin{equation}
\label{eq:half-plane2}
I_C(v)\le I_H(v),%
\end{equation}
where $H\subset\rr^{n+1}$ is a closed half-space. If equality holds in the above inequality for some $v_0>0$ then $C$ is a closed half-space or a slab and isoperimetric regions for volumes $v\le v_0$ are half-balls.
\end{corollary}

\begin{proof}
Note that inequality \eqref{eq:half-plane2} has already been proved in Remark \ref{rem:half-plane}. Therefore it only remains to prove the rigidity property. Assume
\[
I_C(v_0)=I_H(v_0),\quad\text{for some}\ v_0>0.
\]
If $I_H(v_0)>0$ then, owing to Proposition \ref{prop:maincond}, $C$ is of uniform geometry in case it is unbounded. Owing to the first inequality in \eqref{eq:3isop} and the previous equality we get
\[
I_H(v_0)=I_C(v_0)\le I_{C_{\min}}(v_0)\le I_{C_p}(v_0)\le I_H(v_0),
\]
for any $p\in\ptl C$ since $C_p$ is a convex cone. Hence $I_{C}(v_0)=I_{C_p}(v_0)=I_{C_{\min}}(v_0)$ for all $p\in\ptl C$. Theorem~\ref{thm:rigid1-a} then implies that every point $p\in\ptl C$ has a neighborhood in $\ptl C$ which is a part of a hyperplane. It turns out that each connected component of $\ptl C$ is a hyperplane of $\rr^{n+1}$ and so  $C$ is a closed half-space or a slab.
\end{proof}

In \cite{MR2329803} Choe, Ghomi and Ritor\'e investigated the isoperimetric profile outside a convex body $L$ with smooth boundary showing that
\begin{equation*}
I_{\rr^{n+1}\setminus L}(v)\ge I_H(v),\quad\text{for all}\quad v>0.
\end{equation*}
In the following theorem we first show that the above inequality holds for a convex body $C$ without any regularity assumption, and then characterize the case of equality holding for some $v_{0}>0$.
\begin{theorem}
\label{thm:rigid1-c}
Let $C\subsetneq\rr^{n+1}$ be a convex body, and let $I_{\rr^{n+1}\setminus C}$ denote the isoperimetric profile of $\rr^{n+1}\setminus C$. Then we have
\begin{equation}
\label{eq:outin}
I_{\rr^{n+1}\setminus C}(v) \ge I_C(v),%
\end{equation}
for every $0<v<|C|$.  If equality holds in \eqref{eq:outin} for some $v_0>0$ then $C$ is a closed half-space or a slab.
\end{theorem}

\begin{proof}
By the proof of Lemma \ref{lem:fund} we can find a sequence $\{C_i\}_{i\in\nn}$ of convex bodies so that $C_{i+1}\subset C_i$, for all $i\in\nn$, and $C_i\to C$ in  Hausdorff distance.

Note that Proposition~\ref{prp:seqbound}, which is proved in \cite[Remark 3.2]{MR3385175}, holds also if the ambient space is the exterior of a convex body. Consequently for a given $\eps>0$ we can find $r=r(\eps)>0$ and a finite perimeter set $E\subset(\rr^{n+1}\setminus  C)\cap\clb(0,r)$, of  volume $v$, such that
\begin{equation}
\label{eq:q1}
\pp_{\rr^{n+1}\setminus  C}(E)\le I_{\rr^{n+1}\setminus  C}(v)+\eps.
\end{equation}
Let $\Om_i=E\cap (\rr^{n+1}\setminus  C_i)$ and let $B_i$ be Euclidean geodesic balls outside $\clb(0,r)$  having volumes $|E|-|\Om_i|$. Set $E_i=\Om_i\cup B_i$. Then
\begin{equation}
\label{eq:q2}
\vol{E_i}=v\quad\text{and}\quad\lim_{i\to\infty}\pp_{\rr^{n+1}\setminus  C_i}(E_i)=\pp_{\rr^{n+1}\setminus  C}(E)
\end{equation}
Combining \eqref{eq:q1} and \eqref{eq:q2} we get
\begin{equation}
\label{eq:q3}
\limsup_{i\to\infty}I_{\rr^{n+1}\setminus  C_i}(v)\le \lim_{i\to\infty}\pp_{\rr^{n+1}\setminus  C_i}(E_i)\le I_{(\rr^{n+1}\setminus  C)}(v)+\eps.
\end{equation}
As the set $C_i$ have smooth boundary for all $i$, inequality \eqref{eq:outin} holds for all $C_i$ by the result of \cite{MR2329803} and, since $\eps>0$ is arbitrary, we get by \eqref{eq:q3}
\begin{equation*}
I_{\rr^{n+1}\setminus C}(v) \ge I_H(v).
\end{equation*}
Combining this with \eqref{eq:half-plane2} we get 
\begin{equation*}
I_{\rr^{n+1}\setminus C}(v) \ge I_C(v),\quad\text{for all}\quad 0< v<\vol{C}.
\end{equation*}
Suppose now that  equality holds for some $v_0>0$ in the above inequality then 
\begin{equation*}
I_{\rr^{n+1}\setminus C}(v_0)=I_H(v_0)=I_{C}(v_0).
\end{equation*}
Consequently Corollary~\ref{cor:rigid1-b} implies that $C$ is a closed half-space or a slab.
\end{proof}

\begin{remark}
\label{rem:smalvol}
Let $C\subset \rr^{n+1}$ be  a convex body. Then the first inequality in \eqref{eq:3isop} combined with Theorem \ref{thm:limI} and  \eqref{eq:isopsolang} imply
\[
I_{C_{\min}}(1)=\inf\{a>0:\ I_C(v) \le a\,v^{n/(n+1)},\text{ for every } v>0 \}.
\]
Equivalently
\[
I_{C_{\min}}(1)=\sup_{v>0}\frac{I_C(v)}{v^{n/(n+1)}}.
\]
\end{remark}

We recall that an unbounded convex body $C$ is \emph{cylindrically bounded} if it is contained in a right circular cylinder (the tubular neighborhood of a straight line, the axis, in $\rr^{n+1}$). Up to rigid motions, we may assume that the axis is the vertical coordinate axis. Let us denote by $\pi$ the orthogonal projection onto the hyperplane $\{x_{n+1}=0\}$. The closure of the projection $\pi(C)$ is a convex body $K\subset \{x_{n+1}=0\}$. By Example~\ref{ex:cylindrically}, the cylinder $C_\infty = K\times \rr$ is, up to horizontal translations, the only asymptotic cylinder of $C$. In case $C$ contains a line, it is a cylinder and coincides with $C_\infty$. Otherwise, we may assume, eventually composing with a reflection with respect to $\{x_{n+1}=0\}$ and a vertical translation, that $C$ is contained in the half-space $x_{n+1}\ge 0$. 
Before going on we introduce some further notation. First, we set $C_{\infty}^{+} = K\times [0,+\infty)$. Then for any $v>0$ we let $\tau_{C}(v)$ be the unique real number such that the Lebesgue measure of the set $\{x\in C:\ x_{n+1}\le \tau_{C}(v)\}$ is equal to $v$, and we denote such a set by $\Omega(v)$.

Theorem \ref{thm:exiscyl} (i) has been proved in \cite{rv2} under the additional hypotheses that the boundaries of both the convex body and its asymptotic cylinder are $C^{2,\alpha}$. The rigidity in Theorem  \ref{thm:exiscyl} (iii) is a new result.

\begin{theorem}
\label{thm:exiscyl}
Let $C\subset \rr^{n+1}$ be a cylindrically bounded convex body, and assume that $C$ is not a cylinder. Then 
\begin{enum}
\item
Isoperimetric regions exist in $C$ for sufficiently large volumes. 
\item 
There exists $v_0>0$ so that $I_C(v)\le I_{C_\infty^+}(v)$ for every $v\ge v_0$.
\item 
If equality $I_C(v_1)=I_{C_\infty^+}(v_1)$ holds for some $v_1\ge v_0$, then $I_C(v)=I_{C_\infty^+}(v)$ for every $v\ge v_1$. Moreover $C\setminus\Om(v_1)=C_\infty^+\setminus\Om(v_1)$ and so $\Om(v)$ are isoperimetric regions in $C$ for $v\ge v_1$.
\end{enum}
\end{theorem}
\begin{proof}
Let us prove (i) first. By \cite[Theorem 3.9]{rv2} there exists $v_0>0$ so that the slabs $K\times I$, where $I\subset\rr$ is a compact interval, are the only isoperimetric regions of volume larger than or equal to $v_0$ in $C_\infty$. So, for $v\ge v_0$, we get
\begin{equation}
\label{eq:cyl1}
I_C(v)\le \pp_C(\Om (v))\le \h^n(K)<2\h^{n}(K)=I_{C_\infty}(v).
\end{equation}
Thus, by Theorem~\ref{thm:conY}, for every $v\ge v_0$ there exists an isoperimetric region of volume $v$ in $C$. 

We now prove (ii). By \cite[Corollary 3.10]{rv2} there exists $v_0>0$ so that the half-slabs $K\times [0,b]$ are the only isoperimetric regions in $C_{\infty}^{+}$ of volume larger than or equal to $v_0$. Then we obtain
\begin{equation}
\label{eq:cyl4}
I_C(v)\le \pp_C(\Om (v))\le \h^n(K)=I_{C_\infty^+}(v),\quad\text{for every}\,\,v\ge v_0.
\end{equation}

We now prove (iii). We know that $I_C$ is non-decreasing by Theorem \ref{thm:conY}, and that $I_{C_\infty^+}(v)=\h^n(K)$  for $v\ge v_0$. Then we get $I_C(v)=I_{C_\infty^+}(v)$ for every $v\ge v_1$. Furthermore \eqref{eq:cyl4} provides $\pp_C(\Om (v))= \h^n(K)$ for every $v\ge v_1$, yielding \[
C\cap (\rr^n\times \{t(v)\})=C_\infty^+\cap (\rr^n\times \{t(v)\}
\]
for every $v\ge v_1$. Hence $C\setminus\Om(v_1)=C_\infty^+\setminus\Om(v_1)$.
\end{proof}

We now conclude the section with two applications of the rigidity results shown before.

\begin{theorem}
\label{thm:existence}
Let $C\subset\rr^{n+1}$ be an unbounded convex body, different from a half-space, such that any asymptotic cylinder of $C$ is either a half-space or $\rr^{n+1}$. Then $C$ is of uniform geometry and any generalized isoperimetric region must lie in $C$ for any given volume.
\end{theorem}

\begin{proof}
The unbounded convex body $C$ is of uniform geometry by Proposition~\ref{prop:maincond}(iii) since any asymptotic cylinder has non-empty interior. Now Theorem~\ref{thm:conY} implies that, for any given $v_0>0$, there exists a generalized isoperimetric region either in $C$ or in an asymptotic cylinder $K$. Assume the latter case holds. Since any asymptotic cylinder of $C$ is either a half-space $H$ or $\rr^{n+1}$ we have $I_C(v_0)\ge I_H(v_0)$. By Theorem~\ref{thm:rigid1-c}, $I_C(v_0)\le I_H(v_0)$ and so equality $I_C(v_0)=I_H(v_0)$ holds. By the rigidity result of Theorem~\ref{thm:rigid1-c}, $C$ is a half-space or a slab. The second case cannot hold since any asymptotic cylinder of a slab is again a slab. This contradiction shows that any isoperimetric region of volume $v_0$ must be contained in $C$.
\end{proof}

\begin{corollary}
\label{cor:existforNDACsmooth}
Let $C\subset \rr^{n+1}$ be an unbounded convex body satisfying at least one of the following properties:
\begin{itemize}
\item $C$ is a non-cylindrically bounded convex body of revolution;
\item $C$ has a non-degenerate asymptotic cone $C_{\infty}$ such that $\ptl C_{\infty}$ is of class $C^{1}$ with the only exception of a vertex. 
\end{itemize}
Then any generalized isoperimetric region must lie in $C$ for any given volume.
\end{corollary}
\begin{proof}
We observe that in both cases any asymptotic cylinder of $C$ is either a half-space or $\rr^{n+1}$, see Example~\ref{ex:revolution} and Proposition~\ref{prop:NDACsmooth}. The conclusion is thus achieved by applying Theorem~\ref{thm:existence}. 
\end{proof}

\chapter{The isoperimetric dimension of an unbounded convex body}
\label{sec:isopdim}

\section{An asymptotic isoperimetric inequality}

Given a convex body $C\subset\rr^{n+1}$ of uniform geometry, we shall prove in  Theorem~\ref{thm:isopphi} a relative isoperimetric inequality on $C$ depending on the growth rate of the volume of geodesics balls in $C$. We shall follow the arguments by Coulhon and Saloff-Coste \cite{MR1232845}, who established similar inequalities for graphs, groups and manifolds. Their approach makes use of a non-decreasing function $V:\rr^+\to \rr^+$ satisfying
\begin{enumerate}
\item[(i)] $\vol{B_C(x,r)}\ge V(r)$ for all $x\in C$ and $r>0$, and
\item[(ii)] $\lim_{r\to \infty} V(r)=\infty$.
\end{enumerate}
The \emph{reciprocal function} of $V$, $\phi_V:\rr^+\to\rr^+$, is defined by
\[
\phi_V(v):=\inf\{r\in\rr^+:V(r)\ge v\}.
\]
It is immediate to check that $\phi$ is a non-decreasing function. Moreover, if $V_1\ge V_2$ then, for any $v>0$, $\{r\in\rr^+:V_2(r)\ge v\}\subset\{r\in\rr^+:V_1(r)\ge v\}$, and so $\phi_{V_1}\le\phi_{V_2}$.

When $C$ is a convex body of uniform geometry, we know that the quantity $b(r)=\inf_{x\in C}\vol{B_C(x,r)}$ is positive for all $r>0$ by Proposition~\ref{eq:mainhyp}. Let us check that the function $b(r)$ is non-decreasing and satisfies (i) and (ii).

When $0<r<s$, it follows that $b(r)\le\vol{B_C(x,r)}\le\vol{B_C(x,s)}$ for all $x\in C$. This implies $b(r)\le\inf_{x\in C}\vol{B_C(x,s)}=b(s)$ and so the function $b(s)$ is non-decreasing. Property (i) is immediate from the definition of $b(r)$. It remains to show that $\lim_{r\to \infty} b(r)=\infty$. To prove this, consider a vector $v$ with $|v|=1$ so that the half-line $L_{x,v}:=\{x+\la v:\la\ge 0\}$ is contained in $C$ for all $x\in C$. Fix now $x_0\in C$ and $r_0>0$. Then the family  $2kr_0v+B_C(x_0,r_0)$, $k\in \nn\cup\{0\}$, is composed of disjoint sets. Moreover
\[
\bigcup_{k=0}^m\big(2kr_0v+B_C(x_0,r_0)\big)\subset B_C(x_0,(2m+1)r_0),
\]
and so
\[
\vol{B_C(x_0,(2m+1)r_0)}\ge \sum_{k=0}^m\vol{\big(2kr_0v+B_C(x_0,r_0)\big)}\ge (m+1)b(r_0).
\]
Taking infimum over $x_0\in C$ we have 
\[
b((2m+1)r_0)\ge (m+1)b(r_0).
\]
Since $b(r)$ is increasing, this inequality implies $\lim_{r\to \infty} b(r)=\infty$.

\begin{remark}
It is worth noting that, when $C$ is an arbitrary convex body, the asymptotic behaviour of the volume of balls centered at a given point is independent of the point. More precisely we have
\begin{equation*}
\lim_{r\to\infty}\frac{\vol{B_C(x,r)}}{\vol{B_C(y,r)}}=1,
\end{equation*}
for any pair of points $x,y\in C$. To prove this, fix two points $x,y\in C$ and let $d$ be the Euclidean distance between $x$ and $y$. Observe first that Lemma~\ref{lem:doubling} implies, for any $z\in C$,
\[
1\le \frac{\vol{B_C(z,r+d)}}{\vol{B_C(z,r)}}\le\frac{(r+d)^{n+1}}{r^{n+1}}
\]
Taking limits when $r\to\infty$ we get
\[
\lim_{r\to\infty} \frac{\vol{B_C(z,r+d)}}{\vol{B_C(z,r)}}=1.
\]
So we have
\[
\frac{\vol{B_C(x,r)}}{\vol{B_C(y,r)}}=\frac{\vol{B_C(x,r)}}{\vol{B_C(y,r+d)}}
\frac{\vol{B_C(y,r+d)}}{\vol{B_C(y,r)}}\le \frac{\vol{B_C(y,r+d)}}{\vol{B_C(y,r)}},
\]
where the inequality holds since $B_C(x,r)\subset B_C(y,r+d)$. Reversing the roles of $x$ and $y$ we get
\[
\frac{\vol{B_C(x,r)}}{\vol{B_C(x,r+d)}}\le\frac{\vol{B_C(x,r)}}{\vol{B_C(y,r)}}.
\]
Taking limits when $r\to\infty$ we obtain
\[
1\le\liminf_{r\to\infty}\frac{\vol{B_C(x,r)}}{\vol{B_C(y,r)}} \le\limsup_{r\to\infty} \frac{\vol{B_C(x,r)}}{\vol{B_C(y,r)}}\le 1.
\]
\end{remark}

In addition to the existence of the lower bound $V(r)$ for the volume of metric balls in $C$, essential ingredients in the proof of the isoperimetric inequality are the existence of a doubling constant, given by Lemma~\ref{lem:doubling}, and the following uniform Poincar\'e inequality for convex sets, proven by Acosta and Dur\'an \cite{MR2021262} using the ``reduction to one-dimensional problem technique'' introduced by Payne and Weinberger \cite{MR0117419}.

\begin{theorem}[{\cite[Thm.~3.2]{MR2021262}}]
Let $\Om\subset\rr^n$ be a convex domain with diameter $d$ and let $u\in W^{1,1}(\Om)$ with $\int_\Om u=0$. Then
\[
\|u\|_{L^1(\Om)}\le \frac{d}{2}\,\|\nabla u\|_{L^1(\Om)}.
\]
Moreover the constant $1/2$ is optimal.
\end{theorem}

In geometric form, this inequality reads (\cite[\S~5.6]{ev-ga})

\begin{lemma}
\label{lem:isopacosta}
Let $K\subset\rr^{n+1}$ be a bounded convex body with diameter $d$ and let $E\subset K$ be a set of finite perimeter. Then
\[
\frac{d}{2}\, P_K(E)\ge \min\{|E|,|K\setminus E|\}.
\]
In particular, if $C\subset\rr^{n+1}$ is an unbounded convex body, $E\subset C$ has locally finite perimeter in $C$ and $r>0$, then
\[
rP(E,\intt(B_C(x,r)))\ge \min\{|E\cap \clb_C(x,r)|,|\clb_C(x,r)\setminus E|\}.
\]
\end{lemma}

Using Lemma~\ref{lem:isopacosta} we can prove the following isoperimetric inequality on a convex body of uniform geometry.
\begin{theorem}
\label{thm:isopphi}
Let $C\subset\rr^{n+1}$ be a convex body of uniform geometry. Let $V:\rr^+\to\rr^+$ be a non-decreasing function satisfying
\begin{enumerate}
\item[(i)] $\vol{B_C(x,r)}\ge V(r)$ for all $x\in C$ and $r>0$, and 
\item[(ii)] $\lim_{r\to\infty} V(r)=+\infty$.
\end{enumerate}
Let $\phi$ be the reciprocal function of $V$. Then for any set $E\subset C$ of finite perimeter we have
\begin{equation}
\label{eq:isopphi}
P_C(E)\ge 8^{-(n+1)}\frac{|E|}{\phi(2|E|)},
\end{equation}
and so
\begin{equation}
\label{eq:profilephi}
I_C(v)\ge 8^{-(n+1)}\frac{v}{\phi(2v)}.
\end{equation}
\end{theorem}

\begin{proof}
Fix any $r>0$ such that $2|E|\le V(r)$. With this choice, $|E|\le |B_C(x,r)|/2$ for any $x\in C$. Moreover, from the definition of $\phi$, the quantity $\phi(2|E|)$ is equal to the infimum of all $r>0$ such that inequality $V(r)\ge 2|E|$ holds.

Consider a maximal family $\{x_j\}_{j\in J}$ of points in $C$ such that $|x_j-x_k|\ge r$ for all $j,k\in J$, $j\neq k$. Then $C=\bigcup_{j\in J} B_C(x_j,r)$ and the balls $B_C(x_j,r/2)$ are disjoint.

By Lemma~\ref{lem:doubling}, the number of balls $B_C(x_i,r)$ that contain a given point $x\in C$ is uniformly bounded: if $J(x):=\{j\in J: x\in B_C(x_j,r)\}$ then $B_C(x_j,r)\subset B_C(x,2r)\subset B_C(x_j,4r)$ when $j\in J(x)$ and
\begin{align*}
|B_C(x,2r)|\ge \sum_{j\in J(x)}|B_C(x_j,r/2)|&\ge 8^{-(n+1)}\sum_{j\in J(x)}|B_C(x_j,4r)|
\\
&\ge 8^{-(n+1)}\# J(x)\,|B_C(x,2r)|,
\end{align*}
so that
\[
8^{n+1}\ge  \# J(x).
\]
Then we have
\[
|E|\le \sum_{j\in J} |E\cap B_C(x_j,r)|\le\sum_{j\in J} rP(E,\intt(B_C(x_j,r)))\le 8^{n+1}rP_C(E).
\]
Taking infimum over all $r>0$ such that $V(r)\ge 2|E|$ we get
\[
|E|\le 8^{n+1}\phi(2|E|)P_C(E),
\]
equivalent to \eqref{eq:isopphi}. Equation \eqref{eq:profilephi} follows from the definition of the isoperimetric profile.
\end{proof}

\begin{remark}
\label{rem:nonoptnondeg}
Let $C\subset \rr^{n+1}$ be a convex body with non-degenerate asymptotic cone $C_\infty$. For every $x\in C$ we know that $x+C_\infty$ is contained in $C$, so that taking any $r>0$ we have $B_{x+C_\infty}(x,r)\subset B_C(x,r)$. Hence $\vol{B_C(x,r)}\ge \vol{B_{C_\infty}(0,r)}= c\,r^{n+1}$, for $c= \vol{B_{C_\infty}(0,1)}$. This implies that $C$ is of uniform geometry. Theorem~\ref{thm:isopphi} then implies $I_C(v)\ge c'v^{n/(n+1)}$ for every $v>0$ and for some positive constant $c'>0$. See also Remark \ref{rem:optnondeg} for the optimal constant. 
\end{remark}

\begin{remark}
Equation~\ref{eq:profilephi} in Theorem~\ref{thm:isopphi} provides a lower estimate of the isoperimetric profile of $C$ whenever there is a lower estimate $V(r)$ of the volume of metric balls in $C$. For any such function $V$ we have $V(r)\le b(r)=\inf_{x\in C} \vol{B_C(x,r)}$. Hence $\phi_V\ge \phi_b$ and
\[
I_C(v)\ge 8^{-(n+1)}\frac{v}{\phi_b(2v)}\ge 8^{-(n+1)}\frac{v}{\phi_V(2v)}.
\]
Hence the best function $\phi$ we can choose in \eqref{eq:profilephi} corresponds to the reciprocal function of $b$.
\end{remark}

\begin{corollary}
\label{cor:asymp-profile}
Let $C\subset\rr^{n+1}$ be a convex body of uniform geometry, and let $\phi$ be the~reciprocal function of $b$. Then the following inequalities
\begin{equation}
\label{eq:asymp-estimate}
(n+1)\frac{v}{\phi(v)}\ge I_C(v)\ge 3^{-1}8^{-(n+1)}\frac{v}{\phi(v)}
\end{equation}
hold.
\end{corollary}

\begin{proof}
To prove the left side inequality, we pick $x\in C$, $r>0$ so that the ball $B_C(x,r)$ has a given volume $v$, and we consider the cone with vertex $x$ over $\ptl B(x,r)\cap C$ to obtain
\[
(n+1)v\ge rP_C(B_C(x,r))\ge r I_C(v).
\]
Since $\phi(v)$ is the reciprocal function of $b(r)$ we have $\phi(v)=\inf\{r>0: b(r)\ge v\}$. Hence, for any radius $r>0$ such that $|B_C(x,r)|=v$, we get $r\ge\phi(v)$. So we obtain
\begin{equation*}
(n+1)\frac{v}{\phi(v)}\ge I_C(v),
\end{equation*}
as claimed.

We now prove the right side inequality of \eqref{eq:asymp-estimate} using \eqref{eq:profilephi} and a relation between $\phi(v)$ and $\phi(2v)$ obtained in the following way: consider a vector $w$ with $|w|=1$ so that the half-line $\{x+\la w:\la\ge 0\}$ is contained in $C$ for all $x\in C$. Take $x\in C$ and $r>0$. Then  $2rw+B_C(x,r)$ is a subset of $C$ disjoint from $B_C(x,r)$. Since $B_C(x,r)\cup\big(2rw+B_C(x,r)\big)$ is contained in the ball $B_C(x,3r)$, we get the estimate
\[
2\vol{B_C(x,r)}\le \vol{B_C(x,3r)},
\]
for any $x\in C$ and $r>0$.

Fix now $v>0$ and take a sequence of radii $\{r_i\}_{i\in\nn}$ so that $b(r_i)\ge v$ and $\lim_{i\to\infty}r_i=\phi(v)$. For every $x\in C$ and $i\in\nn$ we have
\[
\vol{B_C(x,3r_i)}\ge 2|B_C(x,r_i)|\ge 2b(r_i)\ge 2v.
\]
Taking infimum on $x\in C$ when $r_i$ is fixed we obtain $b(3r_i)\ge 2v$. From the definition of $\phi$ we have $\phi(2v)\le 3r_i$ and taking limits we get
\[
\phi(2v)\le 3\phi(v).
\]
Hence, from \eqref{eq:profilephi} we get
\[
I_C(v)\ge 3^{-1}8^{-(n+1)}\frac{v}{\phi(v)},
\]
as desired.
\end{proof}

\begin{remark}
Let $C\subset\rr^{n+1}$ be an unbounded convex body. For every $x\in C$, Lemma~\ref{lem:doubling} implies that
\[
\frac{\vol{B_C(x,s)}}{s^{n+1}}\le \frac{\vol{B_C(x,r)}}{r^{n+1}}, \qquad 0<r<s.
\]
In particular the function
\[
r\mapsto \frac{\vol{B_C(x,r)}}{r^{n+1}}
\]
is non-increasing. 

Taking $s>0$ fixed, the above inequality implies
\[
\frac{\vol{B_C(x,r)}}{r^{n+1}}\ge \frac{\vol{B_C(x,s)}}{s^{n+1}}\ge \frac{b(s)}{s^{n+1}},\qquad 0<r<s.
\]
Taking the infimum over $x\in C$ we get
\[
b(r)\ge \frac{b(s)}{s^{n+1}}\,r^{n+1}=C_sr^{n+1}, \qquad 0<r<s.
\]
and so $\phi(v)\le C_s^{1/(n+1)} v^{1/(n+1)}$ for $v$ in the interval $(0,C_s^{1/(n+1)}s^{1/(n+1)})$.

Hence \eqref{eq:profilephi} implies
\[
I_C(v)\ge 8^{-(n+1)}C_s^{-1/(n+1)}\frac{v}{v^{1/(n+1)}}=8^{-(n+1)}C_s^{-1/(n+1)}\,v^{n/(n+1)},
\]
for $v$ in the interval $(0,C_s^{1/(n+1)}s^{1/(n+1)})$. This way we recover inequality \eqref{eq:exicyl1} in Corollary~\ref{cor:isopinesm}.
\end{remark}

\section{Estimates on the volume growth of balls}

Our aim now is to obtain accurate estimates of $b(r)$ for given special convex sets in order to understand the behavior of the isoperimetric profile for large volumes using \eqref{eq:asymp-estimate}. While $b(r)$ is easy to compute in homogeneous spaces \cite{MR1818180}, it is harder to estimate in unbounded convex bodies. 
The following argument will be of crucial importance to study the behavior of the volume function $x\in C\mapsto |B_C(x,r)|$ for $r>0$ fixed.

Recall that, given a set $E$ of locally finite perimeter in $\rr^{n+1}$, $\xi\in \rr^{n+1}$ and $t,r\ge 0$, then for all $x\in \rr^{n+1}$ one has
\begin{equation}\label{eq:giusti4.5}
|E\cap B(x+t\xi,r)| = |E\cap B(x,r)| - \int_{0}^{t}\int_{\partial^{*}E\cap B(x+s\xi,r)}\xi\cdot \nu_{E}\, d\h^{n}\, ds,
\end{equation}
where $\nu_{E}$ denotes the weak exterior normal to $\partial^{*}E$. The proof of \eqref{eq:giusti4.5} can be found in \cite[Lemma 4.5]{giusti-book}. We notice that the function 
\[
s\mapsto \int_{\partial^{*}E\cap B(x+s\xi,r)}\xi\cdot \nu_{E}\, d\h^{n}
\] 
is in $L^{\infty}(0,t)$, consequently the function 
\[
t\mapsto |E\cap B(x+t\xi,r)|
\]
is Lipschitz-continuous and thus by \eqref{eq:giusti4.5} and for almost all $t>0$
\[
\frac{d}{dt} |E\cap B(x+t\xi,r)| = -\int_{\partial^{*}E\cap B(x+t\xi,r)}\xi\cdot \nu_{E}\, d\h^{n}.
\]
On the other hand, by integrating $0 = \divv \xi$ on $E\cap B(x+t\xi,r)$ and applying Gauss-Green's Theorem we get for almost all $t>0$
\[
\frac{d}{dt} |E\cap B(x+t\xi,r)| = \int_{E\cap \partial B(x+t\xi,r)} \xi\cdot \nu_{B(x+t\xi,r)}\, d\h^{n}.
\]
Finally, if $x = x(z)\in \rr^{n+1}$ is a smooth parametric curve, $z\in \rr$, then the composition 
\[
z\mapsto |E\cap B(x(z),r)|
\]
is Lipschitz and by the chain rule one gets for almost all $z\in \rr$
\begin{equation}
\label{eq:dvolcomp}
\frac{d}{dz} |E\cap B(x(z),r)| = \int_{E\cap \partial B(x(z),r)} \xi(z)\cdot \nu_{B(x+t\xi,r)}\, d\h^{n},
\end{equation}
where $\xi(z)$ denotes the velocity of $x(z)$.

To compute the integral in \eqref{eq:dvolcomp} the following lemma will be extremely useful

\begin{lemma}
\label{lem:dervolume}
Let $S$ be the sphere $\ptl B(x,r)$, $\nu$ the outer unit normal to $S$. For $\xi\in\esf^n$, let $\sg_\xi:S\to S$ be the reflection with respect to the hyperplane orthogonal to $\xi$ passing through $x$. Let $f_\xi:S\to\rr$ be the function $f_\xi(x):=\escpr{\mu(x),\xi}$, and let $H_\xi^+:=\{x\in S: f_\xi(x)\ge 0\}$ and $H_\xi^-:=\{x\in S: f_\xi(x)\le 0\}$.

Let $\Om\subset S$ be a measurable set and $\Om_v^+:=\Om\cap H_v^+$, $\Om_\xi^-:=\Om\cap H^-$. If $\sg_\xi(\Om_\xi^-)\subset \Om_\xi^+$ then
\[
\int_\Om f_\xi\ge 0.
\] 
\end{lemma}

\begin{proof}
Let us drop the subscript $\xi$. The proof easily follows from $f\circ\sg=-f$ and the area formula:
\[
\int_\Om f=\int_{\Om^+}f+\int_{\Om^-} f=\int_{\Om^+}f-\int_{\sg(\Om^-)}f=
\int_{\Om^+\setminus\sg(\Om^-)} f\ge 0.
\]
\end{proof}

Now we restrict ourselves to a class of rotationallly symmetric unbounded convex bodies. Take a strictly convex function $f:[0,+\infty)\to [0,+\infty)$ of class $C^1$ such that $f(0)=f'(0)=0$. We shall assume that $f''(x)$ exists and is positive for $x>0$, and that $f'''(x)\le 0$ for $x>0$. For instance, the functions $f(x):=x^a$, with $1<a\le 2$, satisfy these conditions. The function $f$ determines the unbounded convex body
\[
C_f:=\{(z,t)\in\rr^n\times\rr : t\ge f(|z|)\}.
\]
The asymptotic cone of the epigraph of $f$ is the half-line $\{(0,t):t\ge 0\}$ if and only if
\[
\lim_{s\to\infty} \frac{f(s)}{s}=+\infty.
\]
This limit exists since the quantity $f(s)/s$ is increasing in $s$ (because of the convexity of $f$ and equality $f(0)=0$). The boundary of $C_f$ is the graph of the function $z\in\rr^n\mapsto f(|z|)$.

The function
\[
\kappa(s):=\frac{f''(s)}{(1+f'(s)^2)^{3/2}}, \qquad s>0,
\]
is the geodesic curvature of the planar curve determined by the graph of $f$. It is also the principal curvature of the meridian curves of the graph of $f$. The function $\kappa$ is decreasing when $s>0$ since
\[
\kappa'=\frac{-3f'(f'')^2+f'''(1+(f')^2)}{(1+(f')^2)^{5/2}}<0.
\]
The principal curvatures of the parallel curves of the graph of $f$ are given by
\[
\frac{f'(s)}{s\,(1+f'(s)^2)^{1/2}},\quad s>0.
\]
This function is also decreasing when $s>0$ as
\[
\bigg(\frac{f'}{s(1+(f')^2)^{1/2}}\bigg)'=\frac{-f'(1+(f')^2)+sf''}{s^2(1+(f')^2)^{3/2}}
\]
and $sf''\le f'$ because of the concavity of $f'$ and the fact that $f'(0)=0$.

For the convex set $C_f$ we are going to prove that $b(r)=|B_C(0,r)|$ for all $r>0$. Thus we can easily estimate the reciprocal function $\phi(v)$ to obtain accurate estimates of the isoperimetric profile of $C_f$ using inequalities \eqref{eq:asymp-estimate}.

\begin{theorem}\label{thm:brorigin}
Let $f:[0,\infty)\to [0,\infty)$ be a $C^1$ function such that $f(0)=f'(0)=0$. Assume that $f$ is of class $C^3$ in $(0,\infty)$ with $f''>0$, $f'''\le 0$, and $\lim_{s\to\infty} (f(s)/s)=+\infty$. Consider the convex body of revolution in $\rr^{n+1}$ given by
\[
C_f:=\{(z,t)\in\rr^n\times\rr: t\ge f(|z|)\}.
\]
Then $|B_C(0,r)|=b(r)$ for all $r>0$.
\end{theorem}

\begin{proof}
Let $C=C_f$. For any $x_0=(z_0,f(|z_0|))\in \ptl C_f\setminus\{0\}$, consider the meridian vector
\[
v_{x_0}:=\frac{(\frac{z_0}{|z_0|},f'(|z_0|))}{(1+f'(|z_0|)^2)^{1/2}}
\]
and let $\sg_{x_0}$ be the orthogonal symmetry with respect to the hyperplane
\[
H_{x_0}:=\{x\in\rr^{n+1}:\escpr{x-x_0,v_{x_0}}=0\}.
\]
Define $H_{x_0}^-:=\{x\in\rr^{n+1}:\escpr{x-x_0,v_{x_0}}\le 0\}$, $H_{x_0}^+:=\{x\in\rr^{n+1}:\escpr{x-x_0,v_{x_0}}\ge 0\}$. By \eqref{eq:dvolcomp} and Lemma~\ref{lem:dervolume}, it is enough to prove
\begin{equation}
\label{eq:inclusion}
\sg_{x_0}(C\cap H_{x_0}^-)\subset C\cap H_{x_0}^+
\end{equation}
for any $x_0\in\ptl C\setminus\{0\}$.

To prove \eqref{eq:inclusion} we shall use a deformation argument similar to Alexandrov Reflection. Let
\[
w_\theta:=\big(\sin\theta\frac{z_0}{|z_0|},\cos\theta\big),\qquad  \theta\in [0,\theta_0]
\]
where
\[
\theta_0:=\arccos\big(\frac{f'(|z_0|)}{(1+f'(|z_0|)^2)^{1/2}}\big)<\frac{\pi}{2}.
\]
When $\theta$ moves along $[0,\theta_0]$, the vector $w_\theta$ varies from $(0,1)$ to $v_{x_0}$. Let us consider the~hyperplanes
\[
H_\theta:=\{x\in\rr^{n+1}:\escpr{x-x_0,w_\theta}=0\},
\]
and $H_\theta^-:=\{x\in\rr^{n+1}:\escpr{x-x_0,w_\theta}\le 0\}$, $H_\theta^+:=\{x\in\rr^{n+1}:\escpr{x-x_0,w_\theta}\ge 0\}$.

Let us check first that the set  $C\cap H_\theta^-$ is bounded. The use of hypothesis $\lim_{s\to\infty} (f(s)/s)=+\infty$ is essential here. Any point $(z,t)\in C\cap H_\theta^-$ satisfies the inequalities
\begin{equation}
\label{eq:htheta-}
\escpr{z-z_0,\frac{z_0}{|z_0|}}\sin\theta+(t-t_0)\cos\theta\le 0, \qquad t\ge f(|z|).
\end{equation}
We reason by contradiction, assuming there is a sequence of points $x_i=(z_i,t_i)$ ($i\in\nn$) in $C\cap H_\theta^-$ with $\lim_{i\to\infty} |x_i|=+\infty$. The sequence $|z_i|$ converges to $+\infty$ since, from \eqref{eq:htheta-} and Schwarz inequality
\[
0\le t_i\cos\theta\le t_0\cos\theta+|z_0|\sin\theta+|z_i|\cos\theta.
\]
Hence boundedness of a subsequence of $|z_i|$ would imply boundedness of the corresponding subsequence of $|t_i|$, contradicting that $\lim_{i\to\infty}|x_i|=+\infty$. On the other hand, inequalities \eqref{eq:htheta-}, together with Schwarz inequality, imply
\[
\frac{f(|z_i|)}{|z_i|}\cos\theta\le \frac{t_0\cos\theta+|z_0|\sin\theta}{|z_i|}+\sin\theta.
\]
Taking limits when $i\to\infty$ we get a contradiction since $|z_i|$ and $f(|z_i|)/|z_i|$ converge to $\infty$.

\begin{figure}[h]
\begin{tikzpicture}[y=0.80pt, x=0.8pt,yscale=-1, inner sep=0pt, outer sep=0pt]
  \path[draw=black,line join=miter,line cap=butt,even odd rule,line width=0.800pt]
    (175.4617,98.4533) .. controls (183.5829,151.7639) and (189.0867,279.1075) ..
    (253.8945,279.1075) .. controls (318.8854,279.1075) and (321.2410,161.9141) ..
    (326.8053,114.8312);
  \path[draw=black,line join=miter,line cap=butt,miter limit=4.00,even odd
    rule,line width=0.400pt] (164.7059,138.6367) -- (132.3529,195.0093) --
    (391.6667,228.3426) -- (404.4118,199.9112);
  \path[draw=black,fill=black,miter limit=4.00,line width=0.240pt]
    (320.6933,188.0893)arc(36.631:126.292:2.446)arc(126.291:215.952:2.446)arc(215.952:305.613:2.446)arc(-54.387:35.274:2.446)
        -- (318.7308,186.6301) -- cycle;
  \path[fill=black,line join=miter,line cap=butt,line width=0.800pt]
    (331.85184,193.10297) node[above right] (text4379) {$x_0$};
  \path[draw=black,->,line join=miter,line cap=butt,even odd rule,line width=0.800pt]
   (255.4215,178.5245) -- (262.0882,136.8774);
  \path[fill=black,line join=miter,line cap=butt,line width=0.800pt]
    (249.36162,125.45498) node[above right] (text4651) {$w_\theta$};
  \path[cm={{0.99619,0.08716,-0.08716,0.99619,(0.0,0.0)}},draw=black,miter
    limit=4.00,line width=0.240pt] (319.5003,167.9065)arc(36.631:126.117:66.157356
    and 17.407)arc(126.117:215.603:66.157356 and
    17.407)arc(215.603:305.089:66.157356 and 17.407)arc(-54.911:34.575:66.157356
    and 17.407);
  \path[draw=black,line join=miter,line cap=butt,even odd rule,line width=0.800pt]
    (185.1852,174.2141) .. controls (195.9259,142.7326) and (220.6584,62.4346) ..
    (271.4815,75.3252) .. controls (321.6369,88.0465) and (318.1482,185.6956) ..
    (318.1482,185.6956);
  \path[fill=black,line join=miter,line cap=butt,line width=0.800pt]
    (413.89679,203.29608) node[above right] (text4659) {$H_\theta$};
  \path[fill=black,line join=miter,line cap=butt,line width=0.800pt]
    (313.53372,262.92908) node[above right] (text4682) {$C\cap
    H_\theta^-$};
  \path[fill=black,line join=miter,line cap=butt,line width=0.800pt]
    (315.13339,94.566795) node[above right] (text4686)
    {$\sigma_\theta(C\cap H_\theta^-)$};

\end{tikzpicture}
\caption{Sketch of the reflection procedure}
\end{figure}
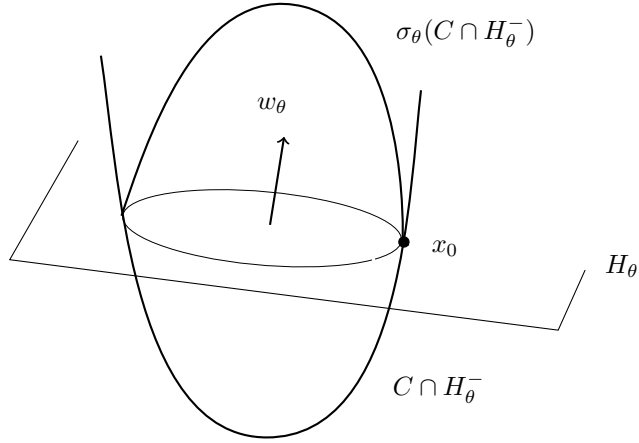

Now we start with a deformation procedure. For $\theta=0$, we have the inclusion $\sg_\theta(C\cap H_\theta^-)\subset C\cap H_\theta^+$ since $H_\theta$ is a horizontal hyperplane and $C$ is the epigraph of a function defined onto this hyperplane. Let $\bar{\theta}$ be the supremum of the closed set
\[
\{\theta\in [0,\theta_0]: \sg_\theta(C\cap H_\theta^-)\subset C\cap H_\theta^+\}.
\]
If $\bar{\theta}=\theta_0$ we are done. Otherwise let us assume that $\bar{\theta}<\theta_0$.

Let us check that, for any $\theta\in [0,\theta_0)$ and $x\in\ptl C\cap H_\theta$, we have
\begin{equation}
\label{eq:bdyinclusion}
\escpr{\bar{\sg}_\theta(v_x),n_x}>0,
\end{equation}
where $\bar{\sg}_\theta$ is the orthogonal symmetry with respect to the linear hyperplane of vectors orthogonal to $w_\theta$ and $n_x$ is the outer unit normal to $\ptl C$ at $x$ given by
\[
n_x=\frac{(f'(|z|)\frac{z}{|z|},-1)}{(1+f'(|z|)^2)^{1/2}}.
\]
Since $\bar{\sg}_x(v_x)=v_x-2\escpr{v_x,w_\theta}w_\theta$ we have $\escpr{\bar{\sg}_\theta(v_x),n_x}=-2\escpr{v_x,w_\theta}\escpr{w_\theta,n_x}$, so that
\begin{multline}
\label{eq:bdysign}
\escpr{\bar{\sg}_\theta(v_x),n_x}=\frac{-2}{1+f'(|z|)^2}\bigg(\frac{\escpr{z,z_0}}{|z||z_0|}\sin\theta+f'(|z|)\cos\theta\bigg)
\\
\times\bigg(f'(|z|)\frac{\escpr{z,z_0}}{|z||z_0|}\sin\theta-\cos\theta\bigg)
\end{multline}

Observe that, if $x\in\ptl C\cap H_\theta$, then $|z|>|z_0|$ unless $z=z_0$. This is easy to prove since
\begin{align*}
0&=\escpr{z-z_0,\frac{z_0}{|z_0|}}\sin\theta+(f(|z|-f(|z_0|))\cos\theta
\\
&\le \big(|z|-|z_0|\big)\sin\theta+(f(|z|-f(|z_0|))\cos\theta.
\end{align*}
In case $|z|<|z_0|$ then $f(|z|)<f(|z_0|)$ and we get a contradiction. If $|z|=|z_0|$ then equality holds in the first inequality and so $z=\la z_0$ for some positive $\la$ which must be equal to one.

To prove \eqref{eq:bdyinclusion}, let us analyze the sign of the factors between parentheses in \eqref{eq:bdysign}. For the first factor, when $x\in\ptl C\cap H_\theta$ we get
\begin{align*}
\frac{\escpr{z,z_0}}{|z||z_0|}\sin\theta&=\frac{|z_0|}{|z|}\sin\theta-\frac{f(|z|)-f(|z_0|)}{|z|}\cos\theta
\\
&\ge \frac{|z_0|}{|z|}\sin\theta-\frac{f(|z|)}{|z|}\cos\theta
\\
&\ge \frac{|z_0|}{|z|}\sin\theta-f'(|z|)\cos\theta
\end{align*}
where for the last inequality we have used $f(x)/x\le f'(x)$, a consequence of the convexity of $f$. Hence
\[
\frac{\escpr{z,z_0}}{|z||z_0|}\sin\theta+f'(|z|)\cos\theta\ge \frac{|z_0|}{|z|}\sin\theta
\]
We thus infer that the quantity in the left-hand side is positive when $\sin\theta>0$ and, when $\sin\theta=0$, it is equal to $f'(|z|)$, which is also positive as $|z|\ge |z_0|>0$.

For the second factor in \eqref{eq:bdysign} we have, for $x\in\ptl C\cap H_\theta$, that the quantity
\begin{equation*}
f'(|z_0|)\frac{\escpr{z,z_0}}{|z||z_0|}\sin\theta-\cos\theta,
\end{equation*}
equal to
\begin{equation*}
\frac{|z_0|}{|z|}f'(|z|)\sin\theta-\frac{f(|z|)-f(|z_0|)}{|z|}f'(|z|)\cos\theta-\cos\theta
\end{equation*}
is strictly smaller than
\begin{equation*}
\frac{1}{(1+f'(|z_0|)^2)^{1/2}}\bigg(\frac{|z_0|}{|z|}f'(|z|)-\frac{f(|z|)-f(|z_0|)}{|z|}f'(|z|)f'(|z_0|)-f'(|z_0|)\bigg).
\end{equation*}
This quantity is negative since $f'(|z|)/|z|\le f'(|z_0|)/|z_0|$ by the concavity of $f'$. In case $\theta=\theta_0$, it is also negative when $z\neq z_0$ since, in this case, $|z|>|z_0|$ and so $f(|z|)>f(|z_0|)$.

So we have proved that the sign of the first factor between parentheses in \eqref{eq:bdysign} is positive and the sign of the second factor is negative. This proves \eqref{eq:bdyinclusion}.

Inequality \eqref{eq:bdyinclusion} guarantees that $\sg_\theta(C\cap H_\theta^-)$ is strictly contained in $C\cap H_\theta^+$ near $H_\theta$ when $\theta<\theta_0$. As in the proof of Alexandrov Reflection principle, it shows that $\sg_\btheta(\ptl C\cap H^-_{\bar{\theta}})$ and $\ptl C\cap H_\btheta^+$ have a tangential contact at some point $x_2\in \ptl C\cap H_\btheta^+$. The point $x_2$ is the image by $\sg_\btheta$ of a point $x_1\in\ptl C\cap H^-_\btheta$ and must lie in the interior of $H_\btheta^+$. Since
\begin{equation*}
x_2=x_1-2\escpr{x_1-x_0,w_\btheta}\,w_\btheta,\qquad n_{x_2}=n_{x_1}-2\escpr{n_{x_1},w_\btheta}\,w_\btheta,
\end{equation*}
we have
\begin{equation}
\label{eq:z2z1}
\begin{split}
z_2&=z_1-2\escpr{x_1-x_0,w_\btheta}\sin\theta\frac{z_0}{|z_0|},
\\
\frac{f'(|z_2|)\,z_2}{|z_2|(1+f'(|z_2|)^2)^{1/2}}&=\frac{f'(|z_1|)\,z_1}{|z_1|(1+f'(|z_1|)^2)^{1/2}}-2\escpr{n_{x_1},w_\btheta}\sin\theta\frac{z_0}{|z_0|}
\\
\frac{-1}{(1+f'(|z_2|)^2)^{1/2}}&=\frac{-1}{(1+f'(|z_1|)^2)^{1/2}}-2\escpr{n_{x_1},w_\btheta}\cos\theta.
\end{split}
\end{equation}
Replacing the value of $z_1$ in the second equation using the first one we get
\begin{multline*}
\bigg(\frac{f'(|z_2|)}{|z_2|(1+f'(|z_2|)^2)^{1/2}}-\frac{f'(|z_1|)}{|z_1|(1+f'(|z_1|)^2)^{1/2}}\bigg)\,z_2
\\
=2\sin\theta\bigg(\frac{f'(|z_1|)}{|z_1|(1+f'(|z_1|)^2)^{1/2}}\,\escpr{x_1-x_0,w_\btheta}-\escpr{n_{x_1},w_\btheta}\bigg)\,\frac{z_0}{|z_0|}.
\end{multline*}
As the function $s\mapsto f'(s)/(s(1+f'(s)^2)^{1/2})$ is strictly decreasing, the constant multiplying $z_2$ is different from zero if and only if $|z_1|\neq|z_2|$. In this case $z_2$ is proportional to $z_0$ and hence $x_2$ (and so $x_1$) belongs to the place generated by $x_0$ and $(0,1)$. Thus we only need to prove that this situation cannot happen in the planar case. Let us check that the case $|z_1|=|z_2|$ is not possible. From the first equation in \eqref{eq:z2z1} we get
\[
|z_2|^2=|z_1|^2+4\escpr{x_1-x_0,w_\btheta}\sin\btheta\bigg(\escpr{x_1-x_0,w_\btheta}\sin\btheta-\frac{\escpr{z_1,z_0}}{|z_1||z_0|}\bigg).
\]
If $|z_1|=|z_2|$ then
\[
\escpr{x_1-x_0,w_\btheta}\sin\btheta-\frac{\escpr{z_1,z_0}}{|z_1||z_0|}=0
\]
and, in particular, $\escpr{z_1,z_0}<0$. Hence
\[
\escpr{n_{x_1},w_\btheta}=\frac{1}{(1+f'(|z_1|)^2)^{1/2}}\bigg(f'(|z_1|)\frac{\escpr{z_1,z_0}}{|z_1||z_0|}\sin\btheta-\cos\btheta\bigg)<0.
\]
From the third equation in \eqref{eq:z2z1} we get
\[
\frac{-1}{(1+f'(|z_2|)^2)^{1/2}}>\frac{-1}{(1+f'(|z_1|)^2)^{1/2}},
\]
and, as the function $s\mapsto -s/(1+f'(s)^2)^{1/2}$ is strictly increasing, we conclude that $|z_2|>|z_1|$, a contradiction.

So we only need to consider the planar case to achieve a contradiction. But in this case Lemma~\ref{lem:planar} gives us a contradiction.
\end{proof}

\begin{lemma}
\label{lem:planar}
Let $f:\rr\to\rr$ be a $C^1$ function satisfying $f(0)=0$ and $f(x)=f(-x)$ for all $x\in\rr$. Assume that $f$ is of class $C^3$ in $\rr\setminus\{0\}$ with $f''>0$ and that the geodesic curvature $\kappa(x)=f''(x)/(1+f'(x)^2)^{3/2}$ is strictly decreasing for $x>0$. Let $C\subset\rr^2$ be the convex epigraph of $f$. Choose $z_0>0$ and any $\theta\in [0,\theta_0]$, where
\[
\theta_0=\arccos\bigg(\frac{f'(z_0)}{(1+f'(z)^2)^{1/2}}\bigg).
\]
Let $x_0=(z_0,f(z_0))$ and $w_\theta=(\sin\theta,\cos\theta)$. Take the line $L_\theta=\{x: \escpr{x-x_0,w_\theta}=0\}$ and the closed half-spaces $H_\theta^-:=\{x: \escpr{x-x_0,w_\theta}\le 0\}$, $H_\theta^+:=\{x: \escpr{x-x_0,w_\theta}\ge 0\}$. Let $\sg_\theta$ be the orthogonal symmetry with respect to $L_\theta$.

Then $\sg_\theta(\ptl C\cap\intt(H_\theta^-))\subset \intt(C)$.
\end{lemma}

\begin{proof}
The curve $\sg_\theta(\ptl C\cap\intt(H_\theta^-))$ is strictly convex and so its tangent vector rotates monotonically. It this tangent vector is never vertical then the curve $\sg_\theta(\ptl C\cap\intt(H_\theta^-))$ is the graph of a function over the $z$-axis lying in $C\cap H_\theta^+$, and so $\sg_\theta(\ptl C\cap\intt(H_\theta^-))$ is trivially contained in $\intt(C)$.

So assume there is a point $x_v$ in $\sg_\theta(\ptl C\cap\intt(H_\theta^-))$ with vertical tangent vector. A straightforward computation shows that $x_v$ is the image $\sg_\theta(x_1)$ of a point $x_1\in \ptl C\cap \intt(H_\theta^-)$ with $x_1=(z_1,f(z_1)$ and
\[
\frac{(1,f'(z_1))}{(1+f'(z_1)^2)^{1/2}}=(\sin(2\theta),\cos(2\theta)).
\]
This implies that $z_1>0$. Define the curves $\Ga_2:=\ptl C\cap\{z_1\le z<z_0\}$, $\Ga_2:= \ptl C\cap\{z>z_0\}$ and let $\kappa_1$, $\kappa_2$ be their geodesic curvatures. Then $\kappa_1(y_1)>\kappa_2(y_2)$ for every pair of points $y_1\in\Ga_1$, $y_2\in\Ga_2$. This implies that $\sg_\theta(\Ga_1)$ and $\Ga_2$ are graphs over a line orthogonal to $L_\theta$ and $\sg_\theta(\Ga_1)$ lies above $\Ga_2$. Since both curves are contained in the half-space $\{z\ge 0\}$ we conclude that $\sg_\theta(\Ga_1)$ is contained in $\intt(C)$. The curve $\sg_\theta ((C\cap H_\theta^-)\setminus\Ga_1)$ has no vertical tangent vector and so it is a graph over the $z$-axis lying over the line $L_\theta$. So it is also contained in $\intt(C)$.
\end{proof}

\section{Examples}

\begin{example}\label{ex:isopdimrev}
We consider the convex body of revolution $C=\{(x,y):y\ge f(|x|)\}\subset\rr^3$ determined by a convex function $f:\rr\to\rr$ such that
\begin{equation*}
\lim_{s\to\infty} \frac{f(s)}{s}=\infty.
\end{equation*}
This condition implies that the asymptotic cone of $C$ has empty interior. Observe that $C$ cannot be cylindrically bounded since $f$ is defined on the whole real line.

In what follows, with a slight abuse of notation, we shall identify the coordinate $x_{1}$ with the pair $(x_{1},0)$, and denote both by $x$. This should not create any confusion thanks to the symmetry of $C$. For every $r>0$, consider the unique point $(x(r),y(r))$, with $x(r)>0$, in the intersection of the graph of $f$ and the circle of center $0$ and radius $r$. Let $\alpha(r)$ be the angle between the vectors $(x(r),y(r))$ and $(0,1)$. Since $y(r)=f(x(r))$ we have
\begin{equation}
\label{eq:cos-sin}
\cos(\alpha(r))=\frac{f(x(r))}{r},\quad \sin(\alpha(r))=\frac{x(r)}{r}.
\end{equation}
An easy application of the coarea formula implies that the volume $V(r)$ of the ball $B_{C}(0,r)$ is given by
\begin{align*}
V(r)&=2\pi\int_0^rs^2(1-\cos(\alpha(s)))\,ds=2\pi\int_0^rs^2\bigg(1-\frac{1}{\sqrt{1+\tan^2(\alpha(s))}}\bigg)\,ds
\\
&=2\pi\int_0^r s^2\bigg(1-\frac{1}{(1+\big(\tfrac{x(s)}{f(x(s))})^2\big)^{1/2}}\bigg)\,ds.
\end{align*}

For $t>-1$, the function $t\mapsto 1-(1+t)^{-1/2}$ is analytic and satisfies
\[
1-\frac{1}{\sqrt{1+t}}=\frac{1}{2}\,tg(t),
\]
where $g:(-1,\infty)\to\rr$ is analytic with $g(0)=1$. Hence we can express $V(r)$ as
\begin{equation}
\label{eq:vrdef}
V(r)=\pi\int_0^r s^2\,\frac{x(s)^2}{f(x(s))^2}\,g\bigg(\frac{x(s)^2}{f(x(s))^2}\bigg)\,ds.
\end{equation}
Note that the estimate $1-(1+t)^{-1/2}<t$ holds for $t>0$ since the derivative of $h(t):=t+(1+t)^{-1/2}$ satisfies $h'(t)=1-\tfrac{1}{2}(1+t)^{-3/2}>\tfrac{1}{2}$ and so $h(t)>h(0)=1$. From this estimate we immediately obtain the inequality $g(t)<2$ for $t>0$.

From \eqref{eq:cos-sin} we get $f(x(s))^2+x(s)^2=s^2$ and so
\[
1+\frac{x(s)^2}{f(x(s))^2}=\frac{s^2}{f(x(s))^2}.
\]
Taking into account that $\lim_{s\to\infty} x(s)=\infty$ we get
\begin{equation}
\label{eq:limits}
\lim_{s\to\infty} \frac{x(s)}{f(x(s))}=0, \qquad
\lim_{s\to\infty}\frac{s}{f(x(s))}=1.
\end{equation}

Using the properties of the function $g$ and \eqref{eq:limits}, we choose $r_0>0$ large enough so that
\begin{equation}
\label{eq:estimatesgx}
1\le g\bigg(\frac{x(s)^2}{f(x(s))^2}\bigg)<2,\quad\frac{1}{2}<\frac{s}{f(x(s))}<\frac{3}{2},\quad\text{for}\ s>r_0.
\end{equation}
Let $h$ be the inverse function of $f$. Observe that the second inequality in \eqref{eq:estimatesgx} yields $h(2s/3)<x(s)<h(2s)$. The concavity of $h$ and equality $h(0)=0$ then imply $(2/3)h(s)<h(2s/3)$ and $h(2s)<2h(s)$. Hence
\begin{equation}
\label{eq:estimatexs}
\frac{2}{3}\,h(s)<x(s)<2h(s),\quad\text{for}\ s>r_0.
\end{equation}
Using the first equation in \eqref{eq:estimatesgx} and \eqref{eq:estimatexs} we get from \eqref{eq:vrdef} the inequalities
\[
\frac{\pi}{9}\int_{r_0}^r h(s)^2ds< V(r)-V(r_0)< 18\pi\int_{r_0}^r h(s)^2ds.
\]
Now let $W(r):=\int_{r_0}^r h(s)^2ds$ for $r>r_0$, and extend it to be equal to $0$ in the interval $[0,r_0]$. Taking $v_0:=V(r_0)$, $D:=\pi/9$ and $E:=18\pi$ we have
\[
v_0+D\,W(r)< V(r)<v_0+E\,W(r),\quad r>r_0,
\]
and so
\[
\phi_W\bigg(\frac{v-v_0}{D}\bigg)>\phi_V(v)>\phi_W\bigg(\frac{v-v_0}{E}\bigg),\qquad v>v_0.
\]

If we take $f(x)=x^a$, with $a\in (1,2]$, then $f''(x)>0$ and $f'''(x)\le 0$ for all $x>0$. In this case $h(x)=x^{1/a}$,
\[
W(r)=\int_{r_0}^r h(s)^2ds=\bigg(\frac{a}{a+2}\bigg)\big(r^{(a+2)/a}-r_0^{(a+2)/a}\big),
\]
and
\[
\phi_W(v)=\bigg(\bigg(\frac{a+2}{a}\bigg)\,v+r_0^{(a+2)/a}\bigg)^{a/(a+2)}.
\]
Hence there exist constants $0<\la<\Lambda>0$ such that
\[
\Lambda v^{a/(a+2)}>\phi_V(v)>\la v^{a/(a+2)}, \quad v>v_0
\]
and so there exists constants $0<\la_1<\la_2$ such that
\[
\la_2 v^{2/(a+2)}>I_C(v)\ge \la_1 v^{2/(a+2)},\quad v>v_0.
\]
This shows that the unbounded convex body
\[
C_{a}= \{(x,y)=(x_{1},x_{2},y)\in \rr^{3} :\ y\ge |x|^{a}\}
\]
has isoperimetric dimension equal to $\frac{a+2}{a}$ for all $a\in (1,2]$.
\end{example}

\backmatter

\bibliographystyle{amsplain}
\bibliography{unbounded}

\providecommand{\bysame}{\leavevmode\hbox to3em{\hrulefill}\thinspace}
\providecommand{\MR}{\relax\ifhmode\unskip\space\fi MR }
\providecommand{\MRhref}[2]{%
  \href{http://www.ams.org/mathscinet-getitem?mr=#1}{#2}
}
\providecommand{\href}[2]{#2}
\begin{thebibliography}{10}

\bibitem{MR2021262}
Gabriel Acosta and Ricardo~G. Dur{\'a}n, \emph{An optimal {P}oincar\'e
  inequality in {$L^1$} for convex domains}, Proc. Amer. Math. Soc.
  \textbf{132} (2004), no.~1, 195--202 (electronic). \MR{2021262 (2004j:26031)}

\bibitem{MR0420406}
F.~J. Almgren, Jr., \emph{Existence and regularity almost everywhere of
  solutions to elliptic variational problems with constraints}, Mem. Amer.
  Math. Soc. \textbf{4} (1976), no.~165, viii+199. \MR{0420406}

\bibitem{MR2178065}
F.~Alter, V.~Caselles, and A.~Chambolle, \emph{A characterization of convex
  calibrable sets in {$\mathbb{R}^N$}}, Math. Ann. \textbf{332} (2005), no.~2,
  329--366. \MR{2178065 (2006g:35091)}

\bibitem{MR887402}
Maria Athanassenas, \emph{A variational problem for constant mean curvature
  surfaces with free boundary}, J. Reine Angew. Math. \textbf{377} (1987),
  97--107. \MR{887402}

\bibitem{AttouchBeer}
H{\'e}dy Attouch and Gerald Beer, \emph{On the convergence of subdifferentials
  of convex functions}, Archiv der Mathematik \textbf{60} (1993), no.~4,
  389--400.

\bibitem{bayle}
Vincent Bayle, \emph{Propri{\'e}t{\'e}s de concavit{\'e} du profil
  isop{\'e}rim{\'e}trique et applications}, Ph.D. thesis, Institut Fourier,
  2003.

\bibitem{bay-rosal}
Vincent Bayle and C{\'e}sar Rosales, \emph{Some isoperimetric comparison
  theorems for convex bodies in {R}iemannian manifolds}, Indiana Univ. Math. J.
  \textbf{54} (2005), no.~5, 1371--1394. \MR{2177105 (2006f:53040)}

\bibitem{be-me}
Pierre B\'erard and Daniel Meyer, \emph{In\'egalit\'es isop\'erim\'etriques et
  applications}, Ann. Sci. \'Ecole Norm. Sup. (4) \textbf{15} (1982), no.~3,
  513--541. \MR{690651 (84h:58147)}

\bibitem{MR2347041}
S.~G. Bobkov, \emph{On isoperimetric constants for log-concave probability
  distributions}, Geometric aspects of functional analysis, Lecture Notes in
  Math., vol. 1910, Springer, Berlin, 2007, pp.~81--88. \MR{2347041
  (2008j:60047)}

\bibitem{bo-sp}
J{\"u}rgen Bokowski and Emanuel Sperner, Jr., \emph{Zerlegung konvexer
  {K}\"orper durch minimale {T}rennfl\"achen}, J. Reine Angew. Math.
  \textbf{311/312} (1979), 80--100. \MR{549959 (81b:52010)}

\bibitem{MR920366}
T.~Bonnesen and W.~Fenchel, \emph{Theory of convex bodies}, BCS Associates,
  Moscow, ID, 1987, Translated from the German and edited by L. Boron, C.
  Christenson and B. Smith. \MR{920366 (88j:52001)}

\bibitem{bostwick-steen}
J.~B. Bostwick and P.~H. Steen, \emph{Stability of constrained capillary
  surfaces}, Annu. Rev. Fluid Mech. \textbf{47} (2015), 539--568.

\bibitem{MR0450480}
Herm~Jan Brascamp and Elliott~H. Lieb, \emph{On extensions of the
  {B}runn-{M}inkowski and {P}r\'ekopa-{L}eindler theorems, including
  inequalities for log concave functions, and with an application to the
  diffusion equation}, J. Functional Analysis \textbf{22} (1976), no.~4,
  366--389. \MR{0450480}

\bibitem{bbi}
Dmitri Burago, Yuri Burago, and Sergei Ivanov, \emph{A course in metric
  geometry}, Graduate Studies in Mathematics, vol.~33, American Mathematical
  Society, Providence, RI, 2001. \MR{MR1835418 (2002e:53053)}

\bibitem{MR936419}
Yu.~D. Burago and V.~A. Zalgaller, \emph{Geometric inequalities}, Grundlehren
  der Mathematischen Wissenschaften [Fundamental Principles of Mathematical
  Sciences], vol. 285, Springer-Verlag, Berlin, 1988, Translated from the
  Russian by A. B. Sosinski{\u\i}, Springer Series in Soviet Mathematics.
  \MR{936419 (89b:52020)}

\bibitem{MR1849187}
Isaac Chavel, \emph{Isoperimetric inequalities}, Cambridge Tracts in
  Mathematics, vol. 145, Cambridge University Press, Cambridge, 2001,
  Differential geometric and analytic perspectives. \MR{1849187}

\bibitem{MR2229062}
\bysame, \emph{Riemannian geometry}, second ed., Cambridge Studies in Advanced
  Mathematics, vol.~98, Cambridge University Press, Cambridge, 2006, A modern
  introduction. \MR{2229062}

\bibitem{Cheeger}
Jeff Cheeger, \emph{A lower bound for the smallest eigenvalue of the
  {L}aplacian}, Problems in analysis, Princeton University Press, 1969,
  pp.~195--199.

\bibitem{MR2215458}
Jaigyoung Choe, Mohammad Ghomi, and Manuel Ritor{\'e}, \emph{Total positive
  curvature of hypersurfaces with convex boundary}, J. Differential Geom.
  \textbf{72} (2006), no.~1, 129--147. \MR{2215458 (2007a:53076)}

\bibitem{MR2329803}
\bysame, \emph{The relative isoperimetric inequality outside convex domains in
  {${\bf R}^n$}}, Calc. Var. Partial Differential Equations \textbf{29} (2007),
  no.~4, 421--429. \MR{2329803 (2008k:58042)}

\bibitem{MR2338131}
Jaigyoung Choe and Manuel Ritor{\'e}, \emph{The relative isoperimetric
  inequality in {C}artan-{H}adamard 3-manifolds}, J. Reine Angew. Math.
  \textbf{605} (2007), 179--191. \MR{2338131 (2009c:53044)}

\bibitem{MR1058436}
F.~H. Clarke, \emph{Optimization and nonsmooth analysis}, second ed., Classics
  in Applied Mathematics, vol.~5, Society for Industrial and Applied
  Mathematics (SIAM), Philadelphia, PA, 1990. \MR{1058436}

\bibitem{MR1232845}
Thierry Coulhon and Laurent Saloff-Coste, \emph{Isop\'erim\'etrie pour les
  groupes et les vari\'et\'es}, Rev. Mat. Iberoamericana \textbf{9} (1993),
  no.~2, 293--314. \MR{1232845 (94g:58263)}

\bibitem{MR1141926}
Martin Dyer and Alan Frieze, \emph{Computing the volume of convex bodies: a
  case where randomness provably helps}, Probabilistic combinatorics and its
  applications ({S}an {F}rancisco, {CA}, 1991), Proc. Sympos. Appl. Math.,
  vol.~44, Amer. Math. Soc., Providence, RI, 1991, pp.~123--169. \MR{1141926
  (93a:52004)}

\bibitem{ev-ga}
Lawrence~C. Evans and Ronald~F. Gariepy, \emph{Measure theory and fine
  properties of functions}, Studies in Advanced Mathematics, CRC Press, Boca
  Raton, FL, 1992. \MR{1158660 (93f:28001)}

\bibitem{fall}
Mouhamed~Moustapha Fall, \emph{Area-minimizing regions with small volume in
  {R}iemannian manifolds with boundary}, Pacific J. Math. \textbf{244} (2010),
  no.~2, 235--260. \MR{2587431 (2011b:53070)}

\bibitem{MR0006422}
F.~Fiala, \emph{Le probl\`eme des isop\'erim\`etres sur les surfaces ouvertes
  \`a courbure positive}, Comment. Math. Helv. \textbf{13} (1941), 293--346.
  \MR{0006422}

\bibitem{FI}
A.~Figalli and E.~Indrei, \emph{A {S}harp {S}tability {R}esult for the
  {R}elative {I}soperimetric {I}nequality {I}nside {C}onvex {C}ones}, J. Geom.
  Anal. \textbf{23} (2013), no.~2, 938--969. \MR{3023863}

\bibitem{MR816345}
Robert Finn, \emph{Equilibrium capillary surfaces}, Grundlehren der
  Mathematischen Wissenschaften [Fundamental Principles of Mathematical
  Sciences], vol. 284, Springer-Verlag, New York, 1986. \MR{816345}

\bibitem{gr}
Matteo Galli and Manuel Ritor{\'e}, \emph{Existence of isoperimetric regions in
  contact sub-riemannian manifolds}, arXiv:1011.0633v1, 2010.

\bibitem{gallot}
Sylvestre Gallot, \emph{In\'egalit\'es isop\'erim\'etriques et analytiques sur
  les vari\'et\'es riemanniennes}, Ast\'erisque (1988), no.~163-164, 5--6,
  31--91, 281 (1989), On the geometry of differentiable manifolds (Rome, 1986).
  \MR{999971 (90f:58173)}

\bibitem{giusti-book}
Enrico Giusti, \emph{Minimal surfaces and functions of bounded variation},
  Monographs in Mathematics, vol.~80, Birkh\"auser Verlag, Basel, 1984.
  \MR{775682 (87a:58041)}

\bibitem{MR684753}
E.~Gonzalez, U.~Massari, and I.~Tamanini, \emph{On the regularity of boundaries
  of sets minimizing perimeter with a volume constraint}, Indiana Univ. Math.
  J. \textbf{32} (1983), no.~1, 25--37. \MR{684753 (84d:49043)}

\bibitem{grom}
Misha Gromov, \emph{Metric structures for {R}iemannian and non-{R}iemannian
  spaces}, english ed., Modern Birkh\"auser Classics, Birkh\"auser Boston Inc.,
  Boston, MA, 2007, Based on the 1981 French original, With appendices by M.
  Katz, P. Pansu and S. Semmes, Translated from the French by Sean Michael
  Bates. \MR{2307192 (2007k:53049)}

\bibitem{MR862549}
Michael Gr{\"u}ter, \emph{Boundary regularity for solutions of a partitioning
  problem}, Arch. Rational Mech. Anal. \textbf{97} (1987), no.~3, 261--270.
  \MR{862549 (87k:49050)}

\bibitem{MR1803974}
John~E. Hutchinson and Yoshihiro Tonegawa, \emph{Convergence of phase
  interfaces in the van der {W}aals-{C}ahn-{H}illiard theory}, Calc. Var.
  Partial Differential Equations \textbf{10} (2000), no.~1, 49--84.
  \MR{1803974}

\bibitem{MR1318794}
R.~Kannan, L.~Lov{\'a}sz, and M.~Simonovits, \emph{Isoperimetric problems for
  convex bodies and a localization lemma}, Discrete Comput. Geom. \textbf{13}
  (1995), no.~3-4, 541--559. \MR{1318794 (96e:52018)}

\bibitem{krahn1925rayleigh}
Edgar Krahn, \emph{{\"U}ber eine von {R}ayleigh formulierte
  {M}inimaleigenschaft des {K}reises}, Mathematische Annalen \textbf{94}
  (1925), no.~1, 97--100.

\bibitem{MR2008339}
Ernst Kuwert, \emph{Note on the isoperimetric profile of a convex body},
  Geometric analysis and nonlinear partial differential equations, Springer,
  Berlin, 2003, pp.~195--200. \MR{2008339 (2004g:49065)}

\bibitem{le-ri}
G.~P. Leonardi and S.~Rigot, \emph{Isoperimetric sets on {C}arnot groups},
  Houston J. Math. \textbf{29} (2003), no.~3, 609--637 (electronic).
  \MR{MR2000099 (2004d:28008)}

\bibitem{MR834360}
P.-L. Lions, \emph{The concentration-compactness principle in the calculus of
  variations. {T}he limit case. {I}}, Rev. Mat. Iberoamericana \textbf{1}
  (1985), no.~1, 145--201. \MR{834360}

\bibitem{MR850686}
\bysame, \emph{The concentration-compactness principle in the calculus of
  variations. {T}he limit case. {II}}, Rev. Mat. Iberoamericana \textbf{1}
  (1985), no.~2, 45--121. \MR{850686}

\bibitem{lions-pacella}
Pierre-Louis Lions and Filomena Pacella, \emph{Isoperimetric inequalities for
  convex cones}, Proc. Amer. Math. Soc. \textbf{109} (1990), no.~2, 477--485.
  \MR{90i:52021}

\bibitem{maggi}
Francesco Maggi, \emph{Sets of finite perimeter and geometric variational
  problems}, Cambridge Studies in Advanced Mathematics, vol. 135, Cambridge
  University Press, Cambridge, 2012. \MR{2976521}

\bibitem{michael}
D.~Michael, \emph{Meniscus stability}, Annu. Rev. Fluid Mech. \textbf{13}
  (1981), 189--215.

\bibitem{MR2507637}
Emanuel Milman, \emph{On the role of convexity in isoperimetry, spectral gap
  and concentration}, Invent. Math. \textbf{177} (2009), no.~1, 1--43.
  \MR{2507637 (2010j:28004)}

\bibitem{MR866718}
Luciano Modica, \emph{The gradient theory of phase transitions and the minimal
  interface criterion}, Arch. Rational Mech. Anal. \textbf{98} (1987), no.~2,
  123--142. \MR{866718}

\bibitem{1210.0567}
Andrea Mondino and Stefano Nardulli, \emph{Existence of isoperimetric regions
  in non-compact {R}iemannian manifolds under {R}icci or scalar curvature
  conditions}, \href{http://arxiv.org/abs/1210.0567}{arXiv:1210.0567}, 1 Oct
  2015.

\bibitem{MR1286892}
Frank Morgan, \emph{Clusters minimizing area plus length of singular curves},
  Math. Ann. \textbf{299} (1994), no.~4, 697--714. \MR{1286892}

\bibitem{morganpolytops}
\bysame, \emph{In polytopes, small balls about some vertex minimize perimeter},
  J. Geom. Anal. \textbf{17} (2007), no.~1, 97--106. \MR{2302876 (2007k:49090)}

\bibitem{MR2438911}
\bysame, \emph{The {L}evy-{G}romov isoperimetric inequality in convex manifolds
  with boundary}, J. Geom. Anal. \textbf{18} (2008), no.~4, 1053--1057.
  \MR{2438911 (2009m:53079)}

\bibitem{MR2455580}
\bysame, \emph{Geometric measure theory}, fourth ed., Elsevier/Academic Press,
  Amsterdam, 2009, A beginner's guide. \MR{2455580 (2009i:49001)}

\bibitem{MR1803220}
Frank Morgan and David~L. Johnson, \emph{Some sharp isoperimetric theorems for
  {R}iemannian manifolds}, Indiana Univ. Math. J. \textbf{49} (2000), no.~3,
  1017--1041. \MR{1803220 (2002e:53043)}

\bibitem{nardulli2014generalized}
Stefano Nardulli, \emph{Generalized existence of isoperimetric regions in
  non-compact riemannian manifolds and applications to the isoperimetric
  profile}, Asian Journal of Mathematics \textbf{18} (2014), no.~1, 1--28.

\bibitem{1506.04892}
Stefano Nardulli and Pierre Pansu, \emph{{ A discontinuous isoperimetric
  profile for a complete Riemannian manifold}},
  \href{http://arxiv.org/abs/1506.04892}{arXiv:1506.04892}, 16 Jun 2015.

\bibitem{OhtaKawasaki1986}
Takao Ohta and Kyozi Kawasaki, \emph{Equilibrium morphology of block copolymer
  melts}, Macromolecules \textbf{19} (1986), no.~10, 2621--2632.

\bibitem{MR2032110}
Frank Pacard and Manuel Ritor{\'e}, \emph{From constant mean curvature
  hypersurfaces to the gradient theory of phase transitions}, J. Differential
  Geom. \textbf{64} (2003), no.~3, 359--423. \MR{2032110}

\bibitem{MR0117419}
L.~E. Payne and H.~F. Weinberger, \emph{An optimal {P}oincar\'e inequality for
  convex domains}, Arch. Rational Mech. Anal. \textbf{5} (1960), 286--292
  (1960). \MR{0117419 (22 \#8198)}

\bibitem{pedri}
Renato H.~L. Pedrosa and Manuel Ritor{\'e}, \emph{Isoperimetric domains in the
  {R}iemannian product of a circle with a simply connected space form and
  applications to free boundary problems}, Indiana Univ. Math. J. \textbf{48}
  (1999), no.~4, 1357--1394. \MR{1757077 (2001k:53120)}

\bibitem{MR1818180}
Ch. Pittet, \emph{The isoperimetric profile of homogeneous {R}iemannian
  manifolds}, J. Differential Geom. \textbf{54} (2000), no.~2, 255--302.
  \MR{1818180 (2002g:53088)}

\bibitem{MR1483543}
Manuel Ritor{\'e}, \emph{Applications of compactness results for harmonic maps
  to stable constant mean curvature surfaces}, Math. Z. \textbf{226} (1997),
  no.~3, 465--481. \MR{1483543}

\bibitem{MR1472144}
\bysame, \emph{Examples of constant mean curvature surfaces obtained from
  harmonic maps to the two sphere}, Math. Z. \textbf{226} (1997), no.~1,
  127--146. \MR{1472144}

\bibitem{MR1883725}
\bysame, \emph{Constant geodesic curvature curves and isoperimetric domains in
  rotationally symmetric surfaces}, Comm. Anal. Geom. \textbf{9} (2001), no.~5,
  1093--1138. \MR{1883725}

\bibitem{MR1857855}
\bysame, \emph{The isoperimetric problem in complete surfaces of nonnegative
  curvature}, J. Geom. Anal. \textbf{11} (2001), no.~3, 509--517. \MR{1857855
  (2002f:53109)}

\bibitem{1503.07014}
\bysame, \emph{{Continuity of the isoperimetric profile of a complete
  {R}iemannian manifold under sectional curvature conditions}}, Revista
  Matem{\'a}tica Iberoamericana (to appear).
  \href{http://arxiv.org/abs/1503.07014}{arXiv:1503.07014}, 24 Mar 2015.

\bibitem{MR1161286}
Manuel Ritor{\'e} and Antonio Ros, \emph{Stable constant mean curvature tori
  and the isoperimetric problem in three space forms}, Comment. Math. Helv.
  \textbf{67} (1992), no.~2, 293--305. \MR{1161286}

\bibitem{MR1322955}
\bysame, \emph{The spaces of index one minimal surfaces and stable constant
  mean curvature surfaces embedded in flat three manifolds}, Trans. Amer. Math.
  Soc. \textbf{348} (1996), no.~1, 391--410. \MR{1322955}

\bibitem{r-r}
Manuel Ritor{\'e} and C{\'e}sar Rosales, \emph{Existence and characterization
  of regions minimizing perimeter under a volume constraint inside {E}uclidean
  cones}, Trans. Amer. Math. Soc. \textbf{356} (2004), no.~11, 4601--4622
  (electronic). \MR{2067135 (2005g:49076)}

\bibitem{rv2}
Manuel Ritor{\'e} and Efstratios Vernadakis, \emph{Large isoperimetric regions
  in the product of a compact manifold with {E}uclidean space},
  \href{http://arxiv.org/abs/1312.1581}{arXiv:1312.1581}, 5 Dec 2013.

\bibitem{MR3385175}
\bysame, \emph{Isoperimetric inequalities in convex cylinders and cylindrically
  bounded convex bodies}, Calc. Var. Partial Differential Equations \textbf{54}
  (2015), no.~1, 643--663. \MR{3385175}

\bibitem{MR3335407}
\bysame, \emph{Isoperimetric inequalities in {E}uclidean convex bodies}, Trans.
  Amer. Math. Soc. \textbf{367} (2015), no.~7, 4983--5014. \MR{3335407}

\bibitem{MR3441524}
\bysame, \emph{Isoperimetric inequalities in conically bounded convex bodies},
  J. Geom. Anal. \textbf{26} (2016), no.~1, 474--498. \MR{3441524}

\bibitem{roc}
R.~Tyrrell Rockafellar, \emph{Convex analysis}, Princeton Mathematical Series,
  No. 28, Princeton University Press, Princeton, N.J., 1970. \MR{0274683 (43
  \#445)}

\bibitem{MR2051615}
Antonio Ros, \emph{Isoperimetric inequalities in crystallography}, J. Amer.
  Math. Soc. \textbf{17} (2004), no.~2, 373--388 (electronic). \MR{2051615
  (2005a:53012)}

\bibitem{MR2167260}
\bysame, \emph{The isoperimetric problem}, Global theory of minimal surfaces,
  Clay Math. Proc., vol.~2, Amer. Math. Soc., Providence, RI, 2005,
  pp.~175--209. \MR{2167260 (2006e:53023)}

\bibitem{MR1338315}
Antonio Ros and Enaldo Vergasta, \emph{Stability for hypersurfaces of constant
  mean curvature with free boundary}, Geom. Dedicata \textbf{56} (1995), no.~1,
  19--33. \MR{1338315}

\bibitem{MR2010323}
C{\'e}sar Rosales, \emph{Isoperimetric regions in rotationally symmetric convex
  bodies}, Indiana Univ. Math. J. \textbf{52} (2003), no.~5, 1201--1214.
  \MR{2010323 (2004h:58018)}

\bibitem{sch}
Rolf Schneider, \emph{Convex bodies: the {B}runn-{M}inkowski theory},
  Encyclopedia of Mathematics and its Applications, vol.~44, Cambridge
  University Press, Cambridge, 1993. \MR{MR1216521 (94d:52007)}

\bibitem{48.0837.03}
P.~Steinhagen, \emph{{\"Uber die gr\"o{\ss}te Kugel in einer konvexen
  Punktmenge.}}, Hamb. Abh. \textbf{1} (1921), 15--26 (German).

\bibitem{MR930124}
Peter Sternberg, \emph{The effect of a singular perturbation on nonconvex
  variational problems}, Arch. Rational Mech. Anal. \textbf{101} (1988), no.~3,
  209--260. \MR{930124}

\bibitem{MR1674097}
Peter Sternberg and Kevin Zumbrun, \emph{On the connectivity of boundaries of
  sets minimizing perimeter subject to a volume constraint}, Comm. Anal. Geom.
  \textbf{7} (1999), no.~1, 199--220. \MR{1674097 (2000d:49062)}

\bibitem{MR1669207}
Edward Stredulinsky and William~P. Ziemer, \emph{Area minimizing sets subject
  to a volume constraint in a convex set}, J. Geom. Anal. \textbf{7} (1997),
  no.~4, 653--677. \MR{1669207 (99k:49089)}

\bibitem{Talenti1976best}
Giorgio Talenti, \emph{Best constant in {S}obolev inequality}, Annali di
  Matematica pura ed Applicata \textbf{110} (1976), no.~1, 353--372.

\bibitem{hoffman-nature}
Edwin~L. Thomas, David~M. Anderson, Chris~S. Henkee, and David Hoffman,
  \emph{Periodic area-minimizing surfaces in block copolymers}, Nature
  \textbf{334} (1988), 598--601.

\bibitem{MR889635}
Thomas~I. Vogel, \emph{Stability of a liquid drop trapped between two parallel
  planes}, SIAM J. Appl. Math. \textbf{47} (1987), no.~3, 516--525. \MR{889635}

\end{thebibliography}

\printindex

\end{document}